\documentclass[usebiblatex=true, pub=false, color=false, todo=false]{kulieee}
\configuretemplate{P. Coppens}{Provably stable learning control of linear dynamics with multiplicative noise}

\usepackage[mode=arxiv]{arxivix}

\graphicspath{{assets/}}

\usepackage{scalerel}
\usepackage{float}

\newenvironment{smallarray}[1]
 {\null\,\vcenter\bgroup\scriptsize
  \arraycolsep=.2em
  \hbox\bgroup$\array{@{}#1@{}}}
 {\endarray$\egroup\egroup\,\null}

\usetikzlibrary{
    patterns,
    decorations.pathmorphing,
    decorations.pathreplacing,
    positioning,
    quotes,
    arrows.meta,
    calc,
}

\newcommand{\smoment}{W}
\newcommand{\smomenth}{\hat{\smoment}}

\newcommand{\cmoment}{W}

\newcommand{\cmomenth}{\hat{\cmoment}}

\newcommand{\mean}{\mu}
\newcommand{\meanh}{\hat{\mean}}

\newcommand{\meas}{\nu}

\newcommand{\W}{\set{W}}
\newcommand{\nsample}{N}

\begin{document}

\title{Provably stable learning control of linear dynamics with multiplicative noise}
\author{Peter Coppens and Panagiotis Patrinos$^\dagger$
\thanks{$^\dagger$P. Coppens and P. Patrinos are with the Department of Electrical
Engineering (ESAT-STADIUS), KU Leuven, Kasteelpark Arenberg
10, 3001 Leuven, Belgium.
        {Email: \tt\footnotesize peter.coppens@kuleuven.be, panos.patrinos@kuleuven.be}}%
\thanks{This work was supported by: the Research Foundation
Flanders (FWO) PhD grant 11E5520N and research projects G086518N and G086318N;
Research Council KU Leuven C1 project No. C14/18/068; Fonds de la Recherche 
Scientifique -- FNRS and the Fonds Wetenschappelijk Onderzoek -- Vlaanderen under 
EOS project no 30468160 (SeLMA).}%
}%
 
\maketitle

\begin{abstract}
Control of linear dynamics with multiplicative noise naturally introduces robustness against dynamical uncertainty.
Moreover, many physical systems are subject to multiplicative disturbances. In this work we show how these dynamics 
can be identified from state trajectories. The least-squares scheme enables exploitation of prior information and comes with 
practical data-driven confidence bounds and sample complexity guarantees. We complement this scheme with an associated 
control synthesis procedure for LQR which robustifies against distributional uncertainty, 
guaranteeing stability with high probability and converging to the true optimum at a 
rate inversely proportional with the sample count. Throughout we exploit the underlying multi-linear problem 
structure through tensor algebra and completely positive operators. The scheme is validated through numerical experiments.
\end{abstract}


\section{Introduction} \label{sec:introduction}
\IEEEPARstart{R}{ecently} the control community has gained renewed interest in learning control \cite{Recht2018}. 
This can be viewed as a reaction to the practical success of machine learning --- specifically 
\emph{Reinforcement Learning (RL)} --- when applied to control dynamical systems. An opportunity arises there in the lack of
theoretical guarantees, particularly related to reliability and safety. On that regard, control theory has much to offer 
\cite{Recht2018, Hewing2020d}. \update*{This work focuses on stability specifically \cite[\S2.1.3]{Brunke2022}.}

One approach towards learning with guarantees \shorten*{is to focus on} specific classes of dynamics.
The \emph{Linear Quadratic Regulator (LQR)} \cite[\S4.1]{Bertsekas2005V1} has proven a valuable subject for such 
investigations. We categorize two approaches: \update*{the \emph{model-free} approach} takes a RL algorithm and
applies it to LQR, proving convergence, stability, etc. \cite{Fazel2018, Bradtke1994, Lewis2009}; \update*{and the \emph{model-based} approach}
develops learning control schemes specialized to LQR with guarantees 
\cite{Dean2019,Mania2019,Abeille2020}. 

\update*{The latter, model-based approach} \cite[p. 360]{Astrom2008}, \cite[p. 537]{Ljung1999}, \cite[\S6.1]{Bertsekas2005V1},
\shorten*{involves first identifying the dynamics before synthesizing a controller.} 
As suggested in \cite{Recht2018} under the term \emph{coarse-id}
learning, \shorten*{integrating model-based approaches with statistical learning theory yields verifiably safe} 
controllers \cite{Dean2019}. We also present such a coarse-id method here. 

Specifically, \shorten*{we present a \emph{distributionally robust (DR)} generalization \cite{Rahimian2019} 
of this coarse-id framework for deterministic LQR, tailored to \emph{linear dynamics with multiplicative noise}.
It incorporates a least-squares estimator of the dynamics paired with tight error bounds and accompanied by a
control synthesis scheme.} \shorten*{In this manner, we expand upon our prior work} \cite{Coppens2019} and interpolate between 
robust and stochastic control, becoming less conservative as data is gathered while guaranteeing stability with high probability.

LQR under multiplicative noise was first considered by Wonham \cite{Wonham1967}, who developed a generalized \emph{Riccati equation}.
Similarly many recent developments in learning LQR are generalizable to multiplicative noise 
\cite{Wang2018,Gravell2019,Pang2021,Coppens2019}. This is interesting for two reasons. 

First, \shorten*{multiplicative noise affects numerous real-world systems.} 
These models occur frequently when disturbances enter a model through the parameters (e.g. through vibrations, 
nonlinear effects in spring systems, thermal and aerodynamic influences, etc.).
A list of examples with references is provided in \cite[\S1.9]{Damm2004} with concrete applications in
aerospace and vehicle control. Multiplicative noise has also been observed in biological applications like 
sensorimotor control \cite{Todorov2002,Todorov2005}, cell populations \cite{Russo2018, Mohler1980a}, 
population migration \cite{Mohler1980a} and immune systems \cite{Mohler1980b}. Other applications 
include nuclear fission \cite{Mohler1980a}, power grids \cite{Afshari2020}, 
communication channels \cite{Wang2002},
\update*{electrical networks \cite{Willsky1976b}, sampled data feedback \cite{Willsky1976b}} and 
climate models \cite{Sura2005}.

Second, the use of multiplicative noise induces robustness against parametric uncertainty. In \cite{Bernstein1987}
this is motivated by the lack of stability margins in output-feedback LQG \cite{Doyle1978}.
This same fact was argued in a historical review by J. Doyle \cite{Doyle1996} to have motivated the development of 
$\set{H}_{\infty}$ methods. In fact, as we formalize later, there is a strong relation between 
\emph{mean square stability (MSS)} for multiplicative noise systems and robust stability 
\cite{Bernstein1987,Gravell2020,Pascoe2019}. 

Next, we provide a detailed overview of related work.


\subsection{Related work}
\paragraph{Historical Overview} The history of multiplicative noise is extensive, and \shorten*{has 
been investigated under various names. As \emph{linear dynamics with state-dependent noise}, it was first considered by \cite{Wonham1967}, who developed an optimal control scheme.
This led to several other early works, outlined in \cite[p.6]{Damm2004}.} 
\update*{Notably, the renowned discovery of the \emph{uncertainty threshold principle} \cite{Athans1977}, thwarts system stabilization upon its violation.}
The term \emph{stochastic bilinear systems}
has also been used to describe \shorten*{multiplicative noise} systems \cite{Kubrusly1985,Mohler1980a}, \update*{\cite{Willsky1976b}}. 
Most recent developments 
\shorten*{primarily involve control synthesis with indefinite stage costs \cite{Rami2002d},
numerical solutions for generalized Riccati equations \cite{Damm2004}, \update*{output feedback \cite{Ding2019}} and application of $\set{H}_\infty$ methods to 
stochastic control \cite{Damm2004}, \update*{\cite{Gershon2006}}}. We similarly consider robustification 
against perturbations to the distribution of the disturbance (similar to \cite{ElGhaoui1995}),
but do so in a data-driven fashion.

\paragraph{Learning control for multiplicative noise}
In the top-down approach several results from LQR have been generalized to multiplicative noise already.
\update*{Policy iteration} has been investigated in \cite{Wang2018}, \update*{\cite{Pang2021, Coppens2023}} and policy gradient in \cite{Gravell2019,Pang2021}.

Developments in the bottom-up approach remain limited \shorten*{due to challenges in identifying multiplicative noise dynamics}.
\shorten*{Except for partial identification in some cases \cite{Kubrusly1981c}, comprehensive}
system identification has only been investigated recently. In \cite{Xing2020,Xing2021} 
a scheme based on deterministic linear system identification was developed by averaging
over trajectories. \shorten*{Unlike our method, this method lacks the ability to incorporate prior model structure and 
has loose confidence bounds.}
The same authors also considered an identification scheme where the mean dynamics are identified 
first and the noise is characterized using bootstrapping \cite{Gravell2020a}. Similarly, this scheme comes with few guarantees. 
Finally \cite{Di2021} considers a least-squares estimator similar to ours. 
They consider scalar measurements, thwarting complete identification.
Their confidence bounds are valid for single-trajectories (while we use multiple). 
We show experimentally that our method is also applicable in the single-trajectory case,
but leave a formal proof to further work.

\paragraph{Mathematical Tools}
Throughout this work we illustrate what mathematical tools are 
applicable in the analysis of linear systems with multiplicative noise. 

We \shorten*{express the underlying bilinear structure using tensor algebra} \cite{Kolda2006,Kolda2009}, \shorten*{facilitating}
the formulation of least squares estimators and studying their accuracy. 

\shorten*{Additionally, we leverage} \emph{completely positive (CP)} operators, \update*{classically} used to model quantum channels \cite{Pascoe2019}.
They were linked to multiplicative noise dynamics before in \cite{Damm2004}
and to stability of linear systems and Lyapunov operators in \cite{Bhatia2002,Damm2003,Lindblad1976}. The
\update*{relation} between MSS and robust stability too was developed originally for CP operators \cite{Pascoe2019}. 

Our estimator is similar to those considered in compressed covariance sensing \cite{Romero2013}
and those in bilinear estimation \cite{Kukush2003}. 

To develop our concentration inequalities we use modern tools from matrix concentration \cite{Tropp2015}. Our controller
synthesis uses results from robust optimization \cite{Ben-Tal2000}. Specifically we use 
\emph{distributionally robust} optimization \cite{Rahimian2019} to interpolate between robust and stochastic 
control synthesis.


\subsection{Contributions}
The main contributions of this work are as follows.
\begin{enumerate}[noitemsep, topsep=0pt, label=(\roman*)]
    \item \label{item:cp} We show how \emph{linear quadratic (LQ)} control of multiplicative noise is equivalent to LQ control of CP dynamics;
    \item \label{item:id} We consider an identification procedure that can both include prior knowledge about the dynamics and
    also performs well in the model-free setting. The optimal estimate is accompanied with a tight confidence set;
    \item \label{item:l4dc} We introduce a synthesis procedure for DR LQR;
    \item \label{item:cdc} We analyze the sample complexity of the scheme.
\end{enumerate}
As mentioned above, CP operators have been applied in multiplicative noise before \cite{Damm2003}. A
formal equivalence of LQ control has not been shown however. The other results all strongly rely on this equivalence. 

\update*{The system identification problem solved in \cref{item:id} was previously 
considered in \cite{Xing2020, Xing2021}. However, their confidence bound is too conservative for practical purposes 
and there is no possibility to include prior information about how the uncertainty enters into the dynamics. Particularly, 
when using longer trajectories in the data, their bound grows when more data is gathered,
which is opposite to the behavior observed in experiments (cf. \cref{fig:toy-rollout} and \cite[p.~15]{Xing2021}). Furthermore, unlike our bound in \cref{thm:error-bound-full},
their bound is not data-driven. It instead depends on unknown system constants." }
In \cite{Di2021} a partial
ambiguity set is derived, but the full system dynamics are never identified. We consider \cref{item:id} as the most 
important contribution.

We also generalize our previous work to the model-free setting: \cref{item:l4dc} extends upon \cite{Coppens2019} 
and \cref{item:cdc} extends upon \cite{Coppens2020}, and applies the result to the DR control synthesis procedure. 

\shorten*{Numerical experiments demonstrate practical applicability.}

The remainder of this paper describes the construction, solution and analysis of DR control for LQR under multiplicative noise.
We provide a problem statement in \cref{sec:problem} and preliminary results in \cref{sec:preliminary}.
We then consider the nominal case for LQR under multiplicative noise, when the dynamics are known.
Next, in \cref{sec:construction}, we describe an ambiguity set based on state measurements; in \cref{sec:synthesis} we solve the resulting DR problem; and
in \cref{sec:analysis} we study convergence of the solution to the nominal controller as more data is gathered. \ilarxiv{Then we show how
to add more prior information to the model --- related to a deterministic term --- in \cref{sec:structural}.}
We end the paper with numerical results in \cref{sec:numerical} and
a conclusion in \cref{sec:conclusion}.

\ilpub{This paper is accompanied by a technical report \citear{} including more detailed proofs, experiments and a section on including a deterministic term in 
the prior model structure.}

\subsection{Notation overview}
Below we provide an overview of the notation. 

Let $\N$ denote the integers and $\Re$ ($\eRe$) the (extended) reals. We introduce the shorthand $\N_{a:b} = \{a, \dots, b\}$ for $a, b \in \N$, with $a \leq b$.
For $x \in \Re^m$ let $\nrm{x}_2 = \sqrt{\trans{x}x}$ denote the Euclidean norm and $\ball{n} \subset \Re^{n}$ the unit-norm ball.
Assuming column vectors, let $(x, y) = [x; y]$ denote the vertical concatenation between vectors.

For a random vector $w$ let $\E_{\meas}[w]$ denote the expectation with respect to the measure $\meas$.
The subscript $\meas$ is omitted when the true distribution is implied.

\paragraph{Matrices}
For $Z \in \Re^{m \times n}$ let $\lambda(Z) = (\lambda_1, \dots, \lambda_m)$
(when $m = n$) denote the vector of eigenvalues in descending order. \update*{Let $\lambda_{\mathrm{max}}(Z) = \lambda_1$, $\lambda_{\mathrm{min}}(Z) = \lambda_n$ 
and $\rho(Z)$ be the spectral radius.}
Similarly $\sigma(Z) = (\sigma_1, \dots, \sigma_r)$ denotes the vector of non-zero 
singular values in descending order with $r = \rk(Z)$ the rank. Denote by $\pinv{Z}$ and $\trans{Z}$ the pseudo-inverse 
and the transpose respectively. Also let $\nrm{Z}_2$ ($\nrm{Z}_F$) be the spectral (Frobenius) norm. 
Let $\kron$ denote the Kronecker product and $\vec(X) \in \Re^{mn}$ the vectorization of $X$ resulting from stacking 
the columns, and $\unvec(\cdot)$ its reverse. When $m = n$ let $\tr(Z)$ denote the trace. 

For matrices $X, Y$ \update*{of conformable dimensions} we use $[X; Y]$ ($[X, Y]$) for vertical (horizontal) concatenation. \update*{Let $\blkdiag(X, Y)$ 
be the block diagonal concatenation of $X$ and $Y$} and let $I_d \in \Re^{d\times d}$ be the identity.
Elements of matrices $X \in \Re^{m \times n}$ (and vectors $x \in \Re^d$) are indexed using $[X]_{ij}$ ($[x]_i$) for 
$i \in \N_{1:m}$, $j \in \N_{1:n}$ ($i \in \N_{1:d}$). \update*{We can select columns (rows) using colon notation $[X]_{:j}$ ($[X]_{i:}$).}
Moreover, to take the first $d$ rows (columns) we write $[X]_{:d, j}$ ($[X]_{i, :d}$). 


We denote by $\sym{d}$ the set of symmetric $d$ by $d$ matrices and by $\sym{d}_{++}$ ($\sym{d}_{+}$) the
positive (semi)definite matrices. For $P, Q \in \sym{d}$, we write $P \sgt Q$ ($P \sgeq Q$)
to signify $P - Q \in \pd{d}$ ($P - Q \in \psd{d}$). \ilarxiv{Similarly $P \nsgeq 0$ implies $0 \neq P \sgeq 0$.}
Let $\sd{d} = (d+1)d/2$ denote the degrees of freedom of elements of $\sym{d}$ and $\svec(X) \in \Re^{\sd{d}}$ 
denotes the symmetrized vectorization (cf. \cref{def:svec}) such that $\trans{\svec(X)} \svec(X) = \nrm{X}_F^2$,
with $\unsvec(\cdot)$ its inverse. Let $\skron$ the symmetrized Kronecker product (cf. \cref{def:skron}).


\paragraph{Tensors}
We use bold calligraphic letters to denote (third-order) tensors $\ten{A} \in \Re^{p \times q \times r}$.
We follow the notation of \cite{Kolda2006} and provide standalone definitions in \cref{sec:tensors}. A tensor is indexed as 
$[\ten{A}]_{ijk}$ with $i \in \N_{1:p}$, $j\in \N_{1:q}$, $k \in \N_{1:r}$. \update*{We generalize the colon notation of 
taking rows and columns for matrices for selecting subtensors (e.g. $[\ten{A}]_{:ij}$).} We use ($\btimes{n}$) $\ttimes{n}$ to denote the \emph{n-mode (vector) product}
and $\ten{A}_{(n)}$ for the \emph{n-mode matricization}. Let $\tucker{\ten{A}; X_1, X_2, X_3} \dfn \ten{A} \ttimes{1} X_1 \ttimes{2} X_2 \ttimes{3} X_3$ 
denote the Tucker operator (cf. \cref{rem:tucker}).
\update*{%
\paragraph{Multiplicative Noise}
For convenience we illustrate the use of tensors for autonomous multiplicative noise dynamics given as $x_{t+1} = (\sum_{i=1}^{n_w} A_i [w_t]_i)x_t$,
with $x_t \in \Re^{n_x}$ and $w_t \in \Re^{n_w}$. The associated tensors is 
$\ten{A} \in \Re^{n_x \times n_x \times n_w}$ populated as $[\ten{A}]_{::i} = A_i$ for $i \in [n_w]$. The matricizations are 
\begin{align*}
    &\ten{A}_{(1)} = \begin{bmatrix}
        \trans{A}_1 & \trans{A}_2 & \trans{A}_{n_w}
    \end{bmatrix}, \, \ten{A}_{(2)} = \begin{bmatrix}
        {A}_1 & {A}_2 & {A}_{n_w}
    \end{bmatrix}\\
    &\quad\ten{A}_{(3)} = \trans{\begin{bmatrix}
        \vec(A_1) & \vec({A}_2) & \vec({A}_{n_w})
    \end{bmatrix}}.
\end{align*}
Using this notation $x_{t+1} = \tucker{\ten{A}; I_{n_x}, x_t, w_t} = \ten{A}_{(1)} (w_t \kron x_t) = \trans{x_t} \ten{A}_{(2)} (w_t \kron I_{n_x}) = \trans{w_t} \ten{A}_{(3)} (x_t \kron I_{n_x})$.%
}%

\section{Problem statement}  \label{sec:problem}
In this work we consider linear systems with input- and state-multiplicative noise given for all $t \in \N$ as
\begin{equation} \label{eq:dyn}
    x_{t+1} = A(v_t) x_t + B(v_t) u_t,
\end{equation}
where $x_t \in \Re^{n_x}$ is the state, $u_t \in \Re^{n_u}$ the input and $v_t \in \Re^{n_v}$ is the disturbance,
which follows an i.i.d. random process. The matrices are given by
$A(v) \dfn \sum_{i=1}^{n_v} [v]_i A_i$ and $B(v) \dfn \sum_{i=1}^{n_v} [v]_i B_i$.

The primary goal is to study solutions of the following stochastic LQR problem:
\begin{equation} \label{eq:slqr}
    \begin{aligned}
        &\minimize_{u_0, u_1, \dots} && \E\left[ \sum_{t=0}^\infty \trans{x}_t Q x_t + \trans{u}_t R u_t \right], 
    \end{aligned} \tag{$\mathcal{LQR}_\star$}
\end{equation}
subject to \eqref{eq:dyn}, for given $x_0$ and $Q \sgt 0$ and $R \sgt 0$\footnote{These assumptions are not necessary for solvability \cite{Rami2002d}, 
yet are added here for convenience.}. 

In practical applications the dynamics \eqref{eq:dyn} are often unknown. Hence we introduce a system identification 
scheme to estimate \eqref{eq:dyn}, supporting both the setting where the modes $A_i$, $B_i$ are unknown and the 
setting where they are (partially) known. This information will be encoded through the use of a model tensor $\ten{M}$
described in \cref{sec:multiplicative}.

Besides the modes, characterizing the distribution of $v_t$ is a more considerable challenge.
We however are aided by the following fact. By linearity, the cost of \eqref{eq:slqr} only involves second moments of the states $X_t = \E[x_t \trans{x_t}]$ 
and the inputs $U_t = \E[u_t \trans{u_t}]$. 
In this work we thus tackle LQR in terms of the moment dynamics:
\begin{equation} \label{eq:dynsm}
    X_{t+1} = \op{E}(Z_t), \quad \text{with } Z_t = \E[z_t \trans{z_t}],
\end{equation}
for augmented state $z_t = (x_t, u_t) \in \Re^{n_z}$. These dynamics are easier to estimate,
compared to \eqref{eq:dyn}, yet contain all information required to solve the stochastic LQR problem. We view \eqref{eq:dynsm} as an implicit
definition for the operator $\op{E}$ \update*{-- the explicit definition being \cref{eq:dynsm_} --} and argue in \cref{sec:multiplicative} how the fundamental operator $\op{E}$ \emph{(i)} is linear in $Z_t$,  
characterizing its matrix $\opm{E}$ in terms of the quantities \eqref{eq:dyn}, \emph{(ii)} is \emph{completely positive (CP)} \update*{(cf. \cref{def:cpop} and \cref{lem:tencp})}, \emph{(iii)} has 
a bilinear structure since $\opm{E}$ depends linearly on $V = \E[v_t \trans{v_t}]$. 
These properties serve as a central theme of this work. 

Since we will never get a completely accurate estimate of the distribution it makes sense in safety critical scenarios to use data to construct a
set of likely distributions $\amb$ containing the true distribution with high probability (see the formal setup in \cite{Coppens2021}).
Given $\amb$ we then solve the DR LQR problem:
\begin{equation} \label{eq:drlqr}
    \begin{alignedat}{3}
        &\minimize_{u_0, u_1, \dots} &\, & \maximize_{\meas \in \amb} &\,\,& \E_{\meas}\left[ \sum_{t=0}^\infty \trans{x}_t Q x_t + \trans{u}_t R u_t \right],
    \end{alignedat} \tag{$\bar{\mathcal{LQR}}$}
\end{equation}
which upper bounds the cost of \eqref{eq:slqr} with high probability. In this work, only the second moment $V = \E[v_t \trans{v_t}]$ 
affects the problem. So we formulate the ambiguity set over moments.

\section{Preliminary results} \label{sec:preliminary}
We introduce the main tools used throughout our work here.
Specifically this section presents operators
over the space of symmetric matrices in \cref{sec:operators} and tensor algebra 
in \cref{sec:tensors}. 
\subsection{Operators over symmetric matrices} \label{sec:operators}
As argued later, the second moment of the disturbance $v$ is identified by solving
a linear system of equations, where the variable of interest is a symmetric matrix. 

Useful tools when dealing with such problems are the Kronecker product $\otimes$ \cite[Def.~4.2.1]{Horn1991}, the 
vector operator $\vec(\cdot)$ \cite[Def.~4.2.9]{Horn1991} and its inverse $\unvec(\cdot)$. These are connected through the fundamental property \cite[Lem.~4.3.1]{Horn1991}
\begin{equation} \label{eq:kronfund}
    \vec(U X \trans{V}) = (V \otimes U) \vec(X), 
\end{equation}
for any $U \in \Re^{m \times n}, V \in \Re^{\ell \times k}$ and $X \in \Re^{n \times k}$. 

This property is often used to rewrite linear equations over matrices in terms of vectors. When $X \in \sym{d}$ however, we often run into non-invertible systems since
$\vec(X) \in \Re^{d^2}$ has unnecessary degrees of freedom. After all, a symmetric matrix is uniquely determined by $\sd{d} \dfn d(d+1)/2$ values. 

To exploit this fact, we will use \cite[Def.~E.1]{DeKlerk2002}:
\begin{definition} \label{def:svec}
    For any $X \in \sym{d}$, define $\svec \colon \sym{d} \to \Re^{\sd{d}}$,
    \begin{align*}
        \svec(X) \dfn (&[X]_{11}, \sqrt{2}[X]_{21}, \dots, \sqrt{2}[X]_{d1}, \\ &\quad [X]_{22}, \sqrt{2}[X]_{32}, \dots, \sqrt{2}[X]_{d2}, \dots, [X]_{dd} )
    \end{align*}
    and let $Q_d \in \Re^{\sd{d} \times d}$ be such that $\svec(X) = Q_d \vec(X)$.
\end{definition}
It is easy to verify that $Q_d$ is a unique matrix for which $Q_d \trans{Q_d} = I_{\sd{d}}$ \cite[\S{}E.2]{DeKlerk2002}.
We can then introduce the symmetrized Kronecker product \cite[Def.~E.3]{DeKlerk2002},
\begin{definition} \label{def:skron}
    For any $U, V \in \Re^{m \times n}$, let
    \begin{equation*}
        V \skron U = \frac{1}{2} Q_{m} (U \kron V + V \kron U) \trans{Q}_n. 
    \end{equation*}
\end{definition}
By definition, we have a generalization of \eqref{eq:kronfund} \cite[Def.~E.3]{DeKlerk2002}
\begin{equation} \label{eq:skronfund}
    \svec(Z X \trans{Z}) = (Z \skron Z) \svec(X).
\end{equation}


We will also consider more general linear operators on symmetric matrices, $\op{S} \colon \sym{n} \to \sym{m}$,
which can be expressed for any $X \in \sym{n}$ as
\begin{equation*}
    \op{S}(X) = \unsvec(\opm{S} \svec(X)),
\end{equation*}
for \update*{a unique} $\opm{S} \in \Re^{\sd{m} \times \sd{n}}$. The set of such operators is denoted $\slop{n}{m}$. The adjoint $\adj{\op{S}}$, defined to satisfy 
$\tr[\op{S}(X) Y] = \tr[X \adj{\op{S}}(Y)]$, has matrix $\trans{\opm{S}}$. 
Similar to \cite{Coppens2020} we will often need to bound the 
spectral norm of the image of such matrix operators. As such we introduce the following norm on $\slop{n}{m}$:
\[\nrm{\op{S}}_2 \dfn \max_{X \in \sym{d}} \{\nrm{\op{S}(X)}_2 \colon \nrm{X}_2 \leq 1\},\] for which we have
$\nrm{\op{S}(X)}_2 \leq \nrm{\op{S}}_2 \nrm{X}_2$ for any $X \in \sym{n}$. This norm is bounded in
terms of the spectral norm of $\opm{S}$:
\begin{lemma} \label{lem:sopbnd}
    Let $\op{S} \in \slop{n}{m}$ with matrix $\opm{S}$. Then,\todoright[text width=4cm]{tight $\opm{S} = \vec(I) \trans{\vec(I)}$ and $X = I$.}
    \[\nrm{\op{S}}_2 \leq \sqrt{m} \nrm{\opm{S}}_2.\]
\end{lemma}
\begin{proof}
    \update*{Deferred to \ilpub{\cite[App.~A]{Arxiv}}\ilarxiv{Appendix~\ref{app:tensors}}.}
\end{proof}

We are especially interested in linear operators that preserve positive semidefiniteness of the argument. That is, a map 
$\op{S} \colon \sym{n} \to \sym{m}$ 
is \emph{positive} iff $\op{S}(X) \sgeq 0$ for all $X \sgeq 0$. More specifically we consider a specific type of positive map:
\begin{definition} \label{def:cpop}
    A $\op{S} \in \slop{n}{m}$ is \emph{completely positive (CP)}  \cite{Choi1975} if, for some $A_i \in \Re^{m \times n}$ with $i \in \N_{1:r}$,
    \begin{equation} \label{eq:cp-def}
        \op{S}(X) = \sum_{i=1}^r A_i X \trans{A_i}.
    \end{equation}
    The matrix of $\op{S}$ is $\opm{S} = \ssum_{i=1}^r (A_i \skron A_i)$. We write $\op{S} \in \cpop{n}{m}$.
\end{definition}

We specialize \cref{lem:sopbnd} for CP operators \cite[Eq.~2.1]{Bhatia2002}:
\begin{lemma} \label{lem:sopbnd-kron}
    Let $\op{S} \in \cpop{n}{m}$. Then, $\nrm{\op{S}}_2 = \nrm{\op{S}(I)}_2$. 
\end{lemma}

Like for linear dynamics, stability of CP operators can be defined 
in terms of the spectral radius. 
\begin{lemma} \label{lem:stabcp}
    For any $\op{S} \in \cpop{n}{m}$ the following assertions are equivalent:
    \begin{enumerateq}
        \item $\rho(\op{S}) \dfn \rho(\opm{S}) < 1$;
        \item $\nrm{\op{S}^t(X)} \to 0$ as $t \to \infty$, $\forall X \sgeq 0$.\footnote{The specific norm is left unspecified since the convergence would remain the same by equivalence of matrix norms \cite[Ex.~5.6.P23]{Horn2012}.}
    \end{enumerateq}
\end{lemma}
\begin{proof}
    The proof is analogous to \cite[Prop.~1]{Kubrusly1985}.
\end{proof}

The spectral radius of a CP map \eqref{eq:cp-def} is related to the \emph{outer spectral radius} $\hat{\rho}(A_1, \cdots, A_r) \dfn \sqrt{\rho(\ssum_{i=1}^r A_i \kron A_i)}$
for square matrices $\{A_i\}_{i=1}^r$. This radius bounds the \emph{joint spectral radius} \[\rho(A_1, \cdots, A_r) \dfn \lim_{t\to \infty} \left\{\sup_{1 \leq i_1, \dots, i_t \leq r} \nrm{A_{i_1} \dots A_{i_t}}^{1/t}_2\right\}.\]
Specifically \cite[Cor.~1.5]{Pascoe2019}:
\begin{theorem} \label{thm:stab-eq}
    Consider the matrices $\{A_i\}_{i=1}^r$. Then \[\hat{\rho}(A_1, \dots, A_r)/\sqrt{r} \leq \rho(A_1, \dots, A_r) \leq \hat{\rho}(A_1, \dots, A_r).\]
\end{theorem}
Note that the joint spectral radius describes stability the dynamics of switched linear system. That is, dynamics $x_{t+1} = A_t x_t$  
are stable for any sequence of matrices $A_t$ from $\{A_i\}_{i=1}^r$ iff $\rho < 1$ \cite[Cor.~1.1]{Jungers2009}. 
Therefore \cref{thm:stab-eq} gives a strong relation between stability of such systems and stability of 
$\op{S}$ (and as we will see later mean square stability of linear dynamics with multiplicative noise).

We finally consider a Lyapunov criterion for stability:
\begin{proposition} \label{prop:lyapcp}
    Let $\op{S} \in \cpop{d}{d}$. Then
    \begin{enumerate}
        \item   If $\rho(\op{S}) < 1$, then $\forall H \sgt 0$
                \begin{equation} \label{eq:cp-lyap}
                    \exists P \sgeq 0 \colon P - \adj{\op{S}}(P) = H. 
                \end{equation}
                Moreover, $P$ is unique, positive definite and $\tr[PX] = \tr[H \sum_{t=0}^\infty \op{S}^t(X)]$, $\forall X \sgeq 0$. 
        \item   If \eqref{eq:cp-lyap} for some $H \sgt 0$, then $\rho(\op{S}) < 1$. 
    \end{enumerate}
\end{proposition}
\begin{proof}
    For linear systems with multiplicative noise this statement was shown in \cite[Lem.~1-2]{Kubrusly1985}. 
    The proof for CP operators is similar\ilarxiv{.}\ilpub{ and is included in \citear{}.}
    \begin{arxiv}
    \paragraph*{(i)} Note that \eqref{eq:cp-lyap} iff
    \begin{equation*}
        \exists P \sgeq 0 \colon (I - \trans{\opm{S}}) \svec(P) = \svec(H),
    \end{equation*}
    since the matrix of $\adj{\op{S}}$ is $\trans{\opm{S}}$. 
    When $\rho(\op{S}) = \rho(\opm{S}) < 1$, then $(I - \trans{\opm{S}})^{-1} = \sum_{t=0}^{\infty} (\trans{\opm{S}})^t = \sum_{t=0}^{\infty} \trans{(\opm{S}^t)}$ \cite[Cor.~5.6.16]{Horn2012}. 
    So, we take $P = \sum_{t=0}^\infty \adj{(\op{S}^t)}(H)$. Then, $\forall X \nsgeq 0$,
    \begin{equation*}
        \tr[PX] = \tr[\ssum_{t=0}^{\infty} \adj{(\op{S}^t)}(H) X] = \tr[H \ssum_{t=0}^\infty \op{S}^t(X)] > 0,
    \end{equation*}
    where the last inequality follows from $\ssum_{t=0}^\infty \op{S}^t(X) = X + \ssum_{t=1}^\infty \op{S}^t(X) \nsgeq 0$
    and $H \sgt 0$, which implies $\tr[H X] \sgt 0$ by \cref{lem:trivialinequality}. 
    Therefore we have shown that $\tr[PX] > 0$ for all $X \nsgeq 0$. So \cref{lem:trivialinequality} 
    also implies $P \sgt 0$. Uniqueness then follows by invertibility of $(I - \opm{S})$. 

    \paragraph*{(ii)} We assume \eqref{eq:cp-lyap} holds. So, $\forall X_t \sgeq 0$,
    \begin{align*}
        &\tr[X_t (P - \adj{\op{S}}(P))] = \tr[(X_t - \op{S}(X_t)) P] \\
        &\qquad \tr[X_t P] - \tr[X_{t+1} P] = \tr[X_t H],
    \end{align*}
    with $X_{t+1} = \op{S}(X_t)$. Telescoping over $t \in \N$ gives
    \begin{equation*}
        \tr[X_0 P] -\lim_{t\to \infty} \tr[X_t P] = \tr[H \ssum_{t=0}^\infty X_t], \, \forall X_0 \sgeq 0.
    \end{equation*}
    Since $X_t \sgeq 0$ and $P \sgeq 0$ we have $\lim_{t\to \infty} \tr[X_t P] \geq 0$. Hence $\infty > \tr[X_0 P] \geq \tr[H \ssum_{t=0}^{\infty} X_t]$.
    Therefore the terms $\tr[H X_t] \to 0$. By $H \sgt 0$ and \cref{lem:trivialinequality} this implies $X_t \to 0$.
    Thus $\rho(\op{S}) < 1$ by \cref{lem:stabcp}. 
    \end{arxiv}
\end{proof}



\subsection{Tensor algebra} \label{sec:tensors}
The dynamics in \eqref{eq:dyn} inherently have bilinear structure. This warrants the use of multi-linear structures and 
specifically third order tensors. In this section we introduce tensors and show 
how they are used to model CP operators.

Let $\ten{T} \in \Re^{q_1 \times q_2 \times q_3}$ denote a generic third order tensor. 
A good introduction is given in \cite{Kolda2006}, \cite[\S12.4]{Golub2013}. 
We restate some parts here for completeness.

We index tensors using $[\ten{T}]_{ijk} \in \Re$. Subtensors are selected with \emph{colon notation} \cite[\S12.4]{Golub2013}
(e.g. $[\ten{T}]_{:jk} \in \Re^{q_1}$, $[\ten{T}]_{i:k} \in \Re^{q_2}$ and $[\ten{T}]_{::k} \in \Re^{q_1 \times q_2}$).
We take a \emph{slice} along mode-$n$ for $n \in \{1, 2, 3\}$ using $[\ten{T}]_{n|i} \in \Re^{\times_{i \neq n} q_i}$ (e.g., $[\ten{T}]_{1|i} = [\ten{T}]_{i::}$).

Next consider \emph{n-mode matricizations} $\ten{T}_{(n)}\in \Re^{q_n \times (\Pi_{i\neq n} q_i)}$:
\begin{equation} \label{eq:1modeflat}
    \ten{T}_{(n)} \dfn \trans{[\vec([\ten{T}]_{n|1}), \dots, \vec([\ten{T}]_{n|q_n})]} .
\end{equation}

These matricizations extract the linearity of the multi-linear tensor operator in one of its arguments. This is illuminated 
by introducing the \emph{mode-$n$ product}. This operation multiplies $\ten{T} \in \Re^{q_1 \times q_2 \times q_3}$ along 
axis $n$ with matrix $X \in \Re^{p_n \times q_n}$ resulting in a tensor where $q_n$ is replaced by $p_n$. The 
definition can be given in terms of matricizations:
\begin{equation} \label{eq:modendef}
    (\ten{T} \ttimes{n} X)_{(n)} \dfn X \ten{T}_{(n)}.
\end{equation}
Since matricization is one-to-one, this uniquely defines the tensor $\ten{T} \ttimes{n} X$.
Alternatively, e.g. for $n=2$, we have $[\ten{T} \ttimes{2} X]_{ijk} = \sum_{\ell=1}^{q_2} [X]_{j\ell} [\ten{T}]_{i\ell k}$ with 
$i \in \N_{1:q_1}$, $j \in \N_{1:p_2}$ and $k \in \N_{1:q_3}$. 
The extraction of linearity in one of the arguments is most explicit in \cref{eq:modendef}. Similarly the \emph{mode-$n$ vector product} 
is $\ten{T} \btimes{n} x = \sum_{j=1}^{q_n} [x]_{j} [\ten{T}]_{n|j}$. So the order of the tensor is reduced to 
$(\ten{T} \btimes{n} x) \in \Re^{\times_{i\neq n} q_i}$. 

We can extend the single mode-$n$ product to the multi-linear one, similarly to \eqref{eq:modendef}.
\begin{definition}[Tucker product] \label{def:tucker}
    \update*{For $n \in \{1, 2, 3\}$, let $X_n \in \Re^{p_n \times q_n}$. We call $\tucker{\ten{T}; X_1, X_2, X_3}$ the \emph{Tucker product}
    with}
    \begin{align}
        \ten{Y} &\dfn ((\ten{T} \ttimes{3} X_3) \ttimes{2} X_2) \ttimes{1} X_1 \label{eq:multimodendef}\\
                &\nfd \tucker{\ten{T}; X_1, X_2, X_3} \in \Re^{p_1 \times p_2 \times p_3}. \nonumber
    \end{align}
    When $p_n = 1$ which is the case when $X_n = \trans{x_n}$ for (column) vector $x_n \in \Re^{q_n}$ then that axis is removed,
    reducing the order of $\ten{Y}$. That is, $\ttimes{n}$ is replaced by $\btimes{n}$. As common in literature 
    we omit the transposes in this case writing $\tucker{\ten{T}; x_1, x_2, x_3}$.
\end{definition}

\begin{remark} \label{rem:tucker}
    The order of operations in \eqref{eq:multimodendef} is important when acting on vectors.
    After all since $\ten{T} \btimes{2} x_2 \in \Re^{q_1 \times q_3}$ writing $(\ten{T} \btimes{2} x_2) \btimes{3} x_3$
    is not defined. \update*{Instead we would need to write $(\ten{T} \btimes{2} x_2) \btimes{2} x_3 = (\ten{T} \btimes{2} x_2) x_3$.} 
    The Tucker operator helps to avoid such complexities, since we assume the order is reduced 
    after \emph{all} mode-$n$ products have been completed. 
\end{remark}

We frequently use a generalization of \eqref{eq:modendef} \cite[Prop.~3.7]{Kolda2009}:
\begin{proposition}\label{prop:kronunfold}
    For $\ten{Y} \in \Re^{p_1 \times p_2 \times p_3}$ as in \cref{def:tucker}:
    \begin{equation} \label{eq:unfoldtucker}
        \ten{Y}_{(n)} = X_n \ten{T}_{(n)} \trans{(X_j \kron X_i)}, 
    \end{equation}        
    for $n, i, j \in \{1, 2, 3\}$ with $j > i$ and $n \neq i \neq j$. 
\end{proposition}

A graphical illustration of \eqref{eq:unfoldtucker} is given in \cref{fig:unfoldtucker}. Note how the values along the third axis 
are split up both by the matricization and the Kronecker product. 

\begin{figure}
    \centering
    \definecolor{top-color}{HTML}{99D5C9}
\definecolor{front-color}{HTML}{6C969D}  
\definecolor{right-color}{HTML}{645E9D} 

\tikzset{
    xy cuboid/.pic={
        \tikzset{%
            /cuboid/.cd,
            #1
        }
        
        \draw[pic actions] 
            (0, 0, 0)                           coordinate (\cubelabel-luf)
            -- ++(0.5*\cubescale*\cubex, 0, 0)  coordinate (\cubelabel-cuf)
            -- ++(0.5*\cubescale*\cubex, 0, 0)  coordinate (\cubelabel-ruf) 
            -- ++(0, -0.5*\cubescale*\cubey, 0) coordinate (\cubelabel-rcf)
            -- ++(0, -0.5*\cubescale*\cubey, 0) coordinate (\cubelabel-rdf)
            -- ++(-0.5*\cubescale*\cubex, 0, 0) coordinate (\cubelabel-cdf)
            -- ++(-0.5*\cubescale*\cubex, 0, 0) coordinate (\cubelabel-ldf)
            -- ++(0, 0.5*\cubescale*\cubey, 0)  coordinate (\cubelabel-lcf)
            -- cycle (\cubelabel-luf);

        \node[pic actions, fill=none, draw=none] (\cubelabel-ccf) at ($(\cubelabel-ruf) - (0.5*\cubescale*\cubex, 0.5*\cubescale*\cubey)$) {\tikzpictext};
    },
    xz cuboid/.pic={
        \tikzset{%
            /cuboid/.cd,
            #1
        }
        
        \draw[pic actions] 
            (1, 0, 0)                           coordinate (\cubelabel-ruf)
            -- ++(-0.5*\cubescale*\cubex, 0, 0) coordinate (\cubelabel-cuf)
            -- ++(-0.5*\cubescale*\cubex, 0, 0) coordinate (\cubelabel-luf)
            -- ++(0, 0, -0.5*\cubescale*\cubez) coordinate (\cubelabel-luc)
            -- ++(0, 0, -0.5*\cubescale*\cubez) coordinate (\cubelabel-lub)
            -- ++(0.5*\cubescale*\cubex, 0, 0)  coordinate (\cubelabel-cub)
            -- ++(0.5*\cubescale*\cubex, 0, 0)  coordinate (\cubelabel-rub)
            -- ++(0, 0, 0.5*\cubescale*\cubez)  coordinate (\cubelabel-ruc)
            -- cycle (\cubelabel-ruf);

        \node [sloped, xslant=0.5, pic actions, fill=none, draw=none] (\cubelabel-cuc) at ($(\cubelabel-ruf) - (0.5*\cubescale*\cubex, 0, 0.5*\cubescale*\cubez)$) {\tikzpictext};
    },
    yz cuboid/.pic={
        \tikzset{%
            /cuboid/.cd,
            #1
        }
        
        \draw[pic actions] 
            (1, 0, 0)                           coordinate (\cubelabel-ruf)
            -- ++(0, -0.5*\cubescale*\cubey, 0) coordinate (\cubelabel-rcf) 
            -- ++(0, -0.5*\cubescale*\cubey, 0) coordinate (\cubelabel-rdf) 
            -- ++(0, 0, -0.5*\cubescale*\cubez) coordinate (\cubelabel-rdc)
            -- ++(0, 0, -0.5*\cubescale*\cubez) coordinate (\cubelabel-rdb)
            -- ++(0, 0.5*\cubescale*\cubey, 0)  coordinate (\cubelabel-rcb)
            -- ++(0, 0.5*\cubescale*\cubey, 0)  coordinate (\cubelabel-rub)
            -- ++(0, 0, 0.5*\cubescale*\cubez)  coordinate (\cubelabel-ruc)
            -- cycle (\cubelabel-ruf);
            
        \node [sloped, yslant=0.5, pic actions, fill=none, draw=none] (\cubelabel-rcc) at ($(\cubelabel-ruf) - (0, 0.5*\cubescale*\cubey, 0.5*\cubescale*\cubez)$) {\tikzpictext};
    },
    /cuboid/.search also={/tikz},
    /cuboid/.cd,
    width/.store in=\cubex,
    height/.store in=\cubey,
    depth/.store in=\cubez,
    units/.store in=\cubeunits,
    scale/.store in=\cubescale,
    label/.store in=\cubelabel,
    width=1,
    height=1,
    depth=1,
    units=cm,
    scale=1,
    label=cube-node
}

\begin{tikzpicture}    
    \pgfmathsetmacro{\localscale}{0.5}
    \pgfmathsetmacro{\modex}{1}
    \pgfmathsetmacro{\modey}{1}
    \pgfmathsetmacro{\modez}{1}
    \pgfmathsetmacro{\offset}{4.5}
    \pgfmathsetmacro{\shift}{4.25}

    \foreach \x in {0,...,1} {
        \foreach \z in {0,...,1} {
            \pic [fill=top-color, scale=\localscale] at ($\localscale*(\x, 0, -\z)$) {xz cuboid={label=cube-\x-0-\z}};
        }
    }

    \foreach \y in {0,...,1} {
        \foreach \z in {0,...,1} {
            \pic [fill=right-color, scale=\localscale] at ($\localscale*(1, -\y, -\z)$) {yz cuboid={label=cube-1-\y-\z}};
        }
    }

    \foreach \x in {0,...,1} {
        \foreach \y in {0,...,1} {
            \pgfmathtruncatemacro{\label}{1+2*\x + \y}
            \pic [fill=front-color, "$\label$", scale=\localscale] at ($\localscale*(\x, -\y, 0)$) {xy cuboid={label=cube-\x-\y-0}};
        }
    }

    \node[anchor=south east] at (cube-0-0-1-lub) {$\ten{T}$};

    \foreach \x in {1,...,2} {
        \foreach \z in {1,...,\modez} {
            \pic [fill=top-color, scale=\localscale] at ($\localscale*(\x-1, 0, -\z+1)+\localscale*(0, -2, \modez+0.75)$) {xz cuboid={label=mode2-\x}};
        }
    }

    \node[anchor=north, sloped, xslant=0.5] at (mode2-1-ruf) {$X_{2}$};

    \foreach \x in {1,...,\modex} {
        \foreach \z in {1,...,2} {
            \pic [fill=right-color, scale=\localscale] at ($\localscale*(\x-1, 0, -\z+1)+\localscale*(2.5, -2, 0)$) {xz cuboid={label=mode3-\z}};
        }
    }

    \node[anchor=south, sloped, xslant=0.5] at (mode3-2-cub) {$X_{3}$};

    \foreach \x in {1,...,\modey} {
        \foreach \y in {1,...,2} {
            \pic [fill=front-color, scale=\localscale] at ($\localscale*(\x-1, -\y+1, 0)+\localscale*(-\modey-0.5, 0, 0)$) {xy cuboid={label=mode1-\y}};
        }
    }

    \node[anchor=south] at (mode1-1-cuf) {$X_{1}$};

    \coordinate (cur) at ($\localscale*(2, 0, 0) + \localscale*(\offset, 0, 0)$);

    \foreach \x in {1,...,\modey} {
        \foreach \y in {1,...,2} {
            \pic [fill=front-color, scale=\localscale] at ($(cur) + \localscale*(\x-1, -\y+1, 0)+\localscale*(-\modey-0.5, 0, 0)$) {xy cuboid={label=mode1-\y}};
        }
    }

    \node[anchor=south] at (mode1-1-cuf) {$X_{1}$};

    \foreach \s in {0,...,1} {
        \foreach \x in {0,...,1} {
            \foreach \z in {0,...,0} {
                \pic [fill=top-color, scale=\localscale] at ($(cur) + \localscale*(\shift*\s, 0, 0) + \localscale*(\x, 0, -\z)$) {xz cuboid={label={cube\s-\x-0-\z}}};
            }
        }

        \foreach \y in {0,...,1} {
            \foreach \z in {0,...,0} {
                \pic [fill=right-color, scale=\localscale] at ($(cur) + \localscale*(\shift*\s, 0, 0) + \localscale*(1, -\y, -\z)$) {yz cuboid={label={cube\s-1-\y-\z}}};
            }
        }

        \foreach \x in {0,...,1} {
            \foreach \y in {0,...,1} {
                \pic [fill=front-color, scale=\localscale] at ($(cur) + \localscale*(\shift*\s, 0, 0) + \localscale*(\x, -\y, 0)$) {xy cuboid={label={cube\s-\x-\y-0}}};
            }
        }

        \foreach \x in {1,...,2} {
            \foreach \z in {1,...,\modez} {
                \pic [fill=top-color, scale=\localscale] at ($(cur) + \localscale*(\shift*\s, 0, 0) + \localscale*(\x-1, 0, -\z+1)+\localscale*(0, -2, \modez+0.75)$) {xz cuboid={label=mode2-\s-\x}};
            }
        }

        \pic [fill=none, draw=none, scale=\localscale, "\Large$\cdot$"] at ($(cur) + \localscale*(\shift*\s, 0, 0) + \localscale*(1.5, 0, 0)+\localscale*(0, -2, 1.75)$) {xz cuboid={width=0.5}};
        \pic [fill=right-color, scale=\localscale] at ($(cur) + \localscale*(\shift*\s, 0, 0) + \localscale*(2.5, 0, 0)+\localscale*(0, -2, 1.75)$) {xz cuboid={label=mode3-\s}};
    }

    \foreach \x in {0,...,1} {
        \foreach \y in {0,...,1} {
            \pgfmathtruncatemacro{\label}{1+2*\x + \y}
            \pic [draw=none, fill=none, scale=\localscale, "$\label$"] at ($(cur) + \localscale*(\x, -\y, 0)$) {xy cuboid};
        }
    }

    \draw [decorate,decoration={brace,amplitude=5pt,mirror,raise=4pt},yshift=0pt]
        (mode2-0-1-luf) -- (mode3-1-ruf) node [midway, yshift=-0.6cm] {$X_{3} \otimes X_{2}$};
    
    \draw [decorate,decoration={brace,amplitude=5pt,raise=4pt},yshift=0pt]
        (cube0-0-0-0-lub) -- (cube1-1-0-0-rub) node [midway, yshift=0.6cm] {$\ten{T}_{(1)} = [[\ten{T}]_{::1}, [\ten{T}]_{::2}]$}; 
\end{tikzpicture}
\vspace{-1em}
    \caption{Visualization of $1$-mode matricization.} \label{fig:unfoldtucker}
    \vspace{-1em}
\end{figure}
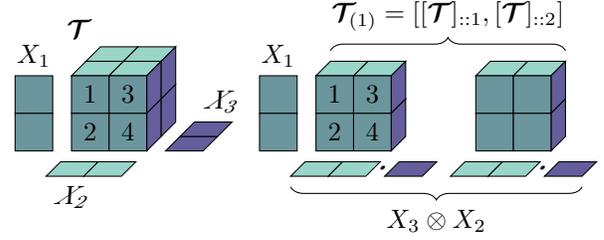

We can describe CP maps using tensors:
\begin{lemma} \label{lem:tencp}
    Any $\op{S} \in \cpop{n}{m}$ can be expressed as:
    \begin{equation} \label{eq:cp-tensor-def}
        \op{S}(X) = \ten{A}_{(1)} (W \kron X) \trans{\ten{A}_{(1)}}
    \end{equation}
    for some $\ten{A} \in \Re^{m \times n \times r}$ and $W \in \psd{r}$. We will refer to $\ten{A}$ as the \emph{model tensor}
    and $W$ as the \emph{parameter} matrix.
\end{lemma}
\begin{proof}
    \update*{Deferred to \ilpub{\cite[App.~A]{Arxiv}}\ilarxiv{Appendix~\ref{app:tensors}}.}
\end{proof}

\begin{remark} \label{rem:parameter}
    Dependency of $\op{S}$ on $W$ in \cref{eq:cp-tensor-def} can be made explicit by writing $\op{S}(W; X)$.
    That is $\op{S}(W; \cdot) \in \slop{n}{m}$. We distinguish the parameters of 
    the map from its arguments using a semicolon. So $\adj{\op{S}}(W; P)$ 
    denotes the adjoint w.r.t. the argument, not the parameter $W$.
\end{remark}

We show that $\op{S}(W; X)$ is also CP in its parameter $W$. 
\begin{corollary} \label{cor:cpmon}
    Let $\op{S}(W; X)$ as in \cref{eq:cp-tensor-def}. Then 
    \begin{enumerate}
        \item $\forall X \sgeq 0$, $\op{S}(\cdot; X) \in \cpop{n}{m}$;
        \item $\op{S}(\bar{W}; X) \sgeq \op{S}(W; X)$ if $\bar{W} \sgeq W$. \label{cor:cpmon:b}
    \end{enumerate}
\end{corollary}
    
\begin{proof}
    \update*{Deferred to \ilpub{\cite[App.~A]{Arxiv}}\ilarxiv{Appendix~\ref{app:tensors}}.}
\end{proof}
We will see later in the setting of multiplicative noise dynamics that \cref{cor:cpmon:b} is useful for finding sufficient 
conditions for stability. Specifically, it allows constructing dynamics whose second moments dominate those of 
the true dynamics in the psd. sense.

There is a bilinear structure present in \eqref{eq:cp-tensor-def}. This becomes clear when considering $\tr[P \op{S}(W; X)]$,
which is linear in $P$, $X$ and $W$. 
Hence, there is a version of \cref{prop:kronunfold} for CP maps (cf. \ilpub{\cite[Prop.~A.1]{Arxiv}}\ilarxiv{\cref{prop:cpunfold}}), which allows different characterizations of $\op{S}$ and $\adj{\op{S}}$.

We define Kronecker products on $\ten{T} \in \Re^{q_1 \times q_2 \times q_3}$ \cite{Lee2014}:
\begin{equation} \label{eq:tensor-kron}
    [\ten{T} \kron \ten{T}]_{(i + (j-1)q_1)::} = [\ten{V}]_{j::} \kron [\ten{V}]_{i::} \text{ for } i, j \in \N_{1:q_1}.
\end{equation}
Also let $\ten{V} \skron \ten{V} = \tucker{\ten{V} \kron \ten{V}; Q_{q_1}, Q_{q_2}, Q_{q_3}}$. 

Next we state a generalization of a property for the Kronecker product 
between matrices, which is used to rewrite CP operators.

\begin{lemma} \label{lem:tensor-kron}
    Suppose $\ten{T} \in \Re^{q_1 \times q_2 \times q_3}$ and $X_i \in \Re^{p_i \times q_i}$ for $i \in \{1, 2, 3\}$
    and let $\ten{Y} = \tucker{\ten{T}; X_1, X_2, X_3}$. Then
    \begin{equation*}
        \ten{Y} \skron \ten{Y} = \tucker{\ten{T} \skron \ten{T}; X_1 \skron X_1, X_2 \skron X_2, X_3 \skron X_3}
    \end{equation*}
\end{lemma}
\begin{proof}
    \update*{Deferred to \ilpub{\cite[App.~A]{Arxiv}}\ilarxiv{Appendix~\ref{app:tensors}}.}
\end{proof}

\begin{corollary} \label{cor:cpadj}
    Let $\op{S}(W; X)$ be CP and parameterized as in \cref{eq:cp-tensor-def} with some $\ten{V}$. 
    Then $\adj{\op{S}}(W; X) = \ten{V}_{(2)} (W \kron X) \trans{\ten{V}}_{(2)}$. Moreover its matrix $\opm{S} = (\ten{V} \skron \ten{V}) \btimes{3} \svec(W)$. 
\end{corollary}
\begin{proof}
    \update*{Deferred to \ilpub{\cite[App.~A]{Arxiv}}\ilarxiv{Appendix~\ref{app:tensors}}.}
\end{proof}
\section{Known distribution}
We consider the nominal setting for multiplicative noise, where the distribution is known.
The (CP) operator $\op{E}$ from \cref{eq:dynsm}, which fully describes the second moment dynamics of \cref{eq:dyn}, is examined.
We prove that it is \emph{(i)} linear in $Z_t$; \emph{(ii)} completely positive; and \emph{(iii)} is linear in $\E[v_t\trans{v_t}]$. 
Moreover, we argue how identifying $\op{E}$ is sufficient to describe both stability and quadratically optimal control. 

\subsection{The model tensor} \label{sec:multiplicative}
We first illustrate how the previous tools interface with multiplicative noise. 
Consider the true model tensor $\ten{V} \in \Re^{n_x \times n_z \times n_v}$, which is populated 
as $[\ten{V}]_{::i} = [A_i, B_i]$ for $i \in \N_{1:n_v}$ using the modes in \eqref{eq:dyn}. This splits up the 
dynamics into two unknowns: the tensor $\ten{V}$ and the distribution of $v_t$. 

To encode known information about $\ten{V}$ 
we introduce the model tensor $\ten{M} \in \Re^{n_x \times n_z \times n_w}$ 
and the auxiliary random vector $w_t \in \Re^{n_w}$. When the modes are known, we can simply pick $\ten{M} = \ten{V}$
and $w_t = v_t$. However, when $\ten{V}$ is only partially known we can still select a $\ten{M}$ that has sufficient 
modeling capacity. Specifically we would like that for any realization of the disturbance $v_t$ and hence of 
the random matrices $A(v_t)$ and $B(v_t)$, there exists a realization of $w_t$ producing the same matrices by the modes 
encoded in $\ten{M}$. 

A sufficient assumption for this property is:
\begin{assumption}[Model Equivalence] \label{asm:model}
    Consider $\ten{V} \in \Re^{n_x \times n_z \times n_v}$ with $n_z \dfn n_x + n_u$
    s.t. $[\ten{V}]_{::i} = [A_i, B_i]$ for $i \in \N_{1:n_v}$ and $\{A_i, B_i\}_{i=1}^{n_v}$ as in \cref{eq:dyn}. 

    Assume $\ten{M} \in \Re^{n_x \times n_z \times n_w}$ with $n_w \in \N$ is such that 
    \begin{equation*}
        \rk(\ten{M}_{(3)}) = \rk([\ten{M}_{(3)}; \ten{V}_{(3)}]) = n_w.
    \end{equation*}
    Here $\ten{M}_{(3)} \in \Re^{n_w \times n_x n_z}$ and $\ten{V}_{(3)} \in \Re^{n_v \times n_x n_z}$ 
    denote the $3$-mode matricizations as in \cref{eq:1modeflat}. 
\end{assumption}

Note that $\ten{M}$
can be partitioned into $\ten{A} \dfn [\ten{M}]_{:,1:n_x,:}$ 
and $\ten{B} \dfn [\ten{M}]_{:,n_x+1:n_z,:}$. \update*{These parts serve as an analogy of $A$ and $B$ for linear time invariant systems.}
When $\ten{M} = \ten{V}$ we have $[\ten{A}]_{::i} = A_i$ and $[\ten{B}]_{::i} = B_i$ for $i \in \N_{1:n_v}$. 

Given \cref{asm:model} and this partitioning, we can show that
\begin{align} 
    x_{t+1} &= \tucker{\ten{A}; I_{n_x}, x_t, {w}_t} + \tucker{\ten{B}; I_{n_x}, u_t, {w}_t} \nonumber \\
            &= \tucker{\ten{M}; I_{n_x}, z_t, {w}_t},\label{eq:dynten}
\end{align}
describes the same dynamics as \eqref{eq:dyn}, given a suitable choice of $w_t$.
We formalize equivalence of \eqref{eq:dynten} and \eqref{eq:dyn} below.

\begin{lemma} \label{lem:model-eq}
    When \cref{asm:model} holds and $w_t = \pinv{(\trans{\ten{M}_{(3)}})} \trans{\ten{V}_{(3)}} v_t$, $\forall t \in \N$, 
    then \eqref{eq:dynten} produces the same trajectories as \eqref{eq:dyn}.
\end{lemma}
\begin{proof}
    \update*{The statement in \cref{asm:model} implies that the columns of $\trans{\ten{V}_{(3)}}$ 
    are in the range of $\trans{\ten{M}_{(3)}}$. This implies that 
    $\trans{\ten{M}_{(3)}} w_t = \trans{\ten{M}_{(3)}} \pinv{(\trans{\ten{M}_{(3)}})} \trans{\ten{V}_{(3)}} w_t = \trans{\ten{V}_{(3)}} v_t$,
    since $\trans{\ten{M}_{(3)}} \pinv{(\trans{\ten{M}_{(3)}})}$ is an orthogonal projection onto the range of $\trans{\ten{M}_{(3)}}$.}

    \update*{Using \eqref{eq:modendef} on both sides of $\trans{w_t} \ten{M}_{(3)} =  \trans{v}_t \ten{V}_{(3)}$
    and reshaping gives $(\ten{M} \ttimes{3} \trans{w_t}) = (\ten{V} \ttimes{3} \trans{v_t})$. 
    Replacing $\ttimes{3}$ with $\btimes{3}$ only alters the dimensions, without changing the total number of elements. 
    So $(\ten{M} \btimes{3} w_t) = (\ten{V} \btimes{3} v_t)$. Then, by the definitions of $\btimes{3}$, $\ten{V}$ and $A(v)$, $B(v)$ we have respectively
    \begin{align*}
        \ten{V} \btimes{3} v &= \ssum_{i=1}^{n_v} [v]_i [\ten{V}]_{::i}\\
                                &= \ssum_{i=1}^{n_v} [v]_i [A_i, B_i] = [A(v), B(v)].
    \end{align*}
    Finally, by \cref{def:tucker}, $\tucker{\ten{M}; I_{n_x}, z_t, {w}_t} = (\ten{M} \btimes{3} w_t) \btimes{2} z_t = [A(v_t), B(v_t)] z_t = A(v_t) x_t + B(v_t) u_t$. 
    So the produced trajectories will be the same.}
\end{proof}

\begin{example} \label{ex:modeltensor}
    We illustrate the need for \cref{asm:model}. 
    Assume $n_x = 2$, $n_u = 1$ and $n_v = 3$ with true modes:
    \begin{equation*}
        A_1 = \mat{
            2 & 0 \\ 1 & 0
        }, \, A_2 = \mat{
            0 & 3 \\ 0 & 0
        }, \, A_3 = \mat{
            0 & 0 \\ 0 & 1
        }, \, B_3 = \mat{
            0 \\ 2
        },
    \end{equation*}
    and $B_1 = B_2 = 0$. In this case $\ten{V}_{(3)}$ is given as
    \begin{equation*}
        \left[
            \begin{smallarray}{cccc|cc}
                2 & 1 & 0 & 0 & 0 & 0 \\
                0 & 0 & 3 & 0 & 0 & 0 \\
                0 & 0 & 0 & 1 & 0 & 2
            \end{smallarray}
        \right] = \left[
            \begin{smallarray}{c|c}
                \trans{\vec(A_1)} & \trans{\vec(B_1)} \\
                \trans{\vec(A_2)} & \trans{\vec(B_2)} \\
                \trans{\vec(A_3)} & \trans{\vec(B_3)} \\
            \end{smallarray}
        \right].
    \end{equation*}
    Say a designer, through some erroneous prior analysis, assumed that 
    the top-left component in $A_1$ was equal to zero instead of two. This corresponds to the following model tensor:
    \begin{equation*}
        \ten{M}_{(3)} = \left[
            \begin{smallarray}{cccc|cc}
                \mathbf{0} & 1 & 0 & 0 & 0 & 0 \\
                0 & 0 & 3 & 0 & 0 & 0 \\
                0 & 0 & 0 & 1 & 0 & 2
            \end{smallarray}
        \right].
    \end{equation*}
    for this choice \cref{asm:model} fails and so will \cref{lem:model-eq}.
    To see this, take $v_t = (1, 0, 0)$. Then $\vec([A(v), B(v)]) = (2, 1, 0, 0, 0, 0)$. Observe
    that such a vector cannot be constructed by any linear combination of the rows of $\ten{M}_{(3)}$.
    So there is no distribution for $w_t$ that reproduces the true dynamics. 
\end{example}

Informally \cref{asm:model} requires that the span of $\trans{\ten{M}_{(3)}}$ should be sufficiently large 
to model the outputs of $A(v)$ and $B(v)$ given by realizations of $v \in \Re^{n_v}$. 

Note that this is always possible, by taking $n_w = n_x (n_x + n_u)$ and 
$\ten{M}_{(3)} = I_{n_w}$. We refer to this as the \emph{model-free basis}, since no prior structural assumptions
are imposed on the dynamics. If we know the \emph{true modes} $A_i$ and $B_i$, we can simply select $\ten{M} = \ten{V}$.
As such, $\ten{M}$ is used to introduce prior information into the dynamics.

Since the sequence $v_t$ is i.i.d., the sequence $w_t$ as defined in \cref{lem:model-eq} is i.i.d.\ too. We thus omit 
$t$ as the distribution remains constant. We can then use the transformation in \cref{lem:model-eq} to 
find the true $W = \E[w \trans{w}]$ in terms of $V = \E[v \trans{v}]$:
\begin{equation} \label{eq:true_moment}
    W = \pinv{(\trans{\ten{M}_{(3)}})} \trans{\ten{V}_{(3)}} V \ten{V}_{(3)} \pinv{\ten{M}_{(3)}}.
\end{equation}
This allows translating the \emph{ground truth} dynamics to another model tensor and will aid in
validation of our method during the experiments in \cref{sec:numerical}.



We next illustrate the connection with CP operators by considering the dynamics of $X_t = \E[x_t \trans{x_t}]$ as in \cref{eq:dynsm}. 
By \cref{prop:kronunfold}, $\tucker{\ten{M}; I, z_t, {w}_t} = \ten{M}_{(1)} ({w}_t \kron z_t)$. 
Taking the outer product, using $(w_t \kron z_t) \trans{(w_t \kron z_t)} = (w_t \trans{w_t} \kron z_t \trans{z_t})$ \cite[Thm.~E.1.3]{DeKlerk2002}
and taking the expectation allows us to derive:
\begin{equation} \label{eq:dynsm_}
    X_{t+1} = \op{E}(W; Z_t) \dfn {\ten{M}}_{(1)} (\smoment \otimes Z_t) \trans{{\ten{M}}_{(1)}},
\end{equation}
where $X_t = \E[x_t \trans{x_t}]$, $Z_t = \E[z_t \trans{z_t}]$ and $W = \E[w_t \trans{w_t}]$. 
We write $\op{E}(\cdot)$ when the specific value of $W$ is unimportant. 
The operator $\op{E} \colon \sym{n_z} \to \sym{n_x}$ is clearly CP by \cref{lem:tencp}. It is linear with its 
matrix characterized by \cref{cor:cpadj}. Moreover it has a bilinear structure, depending linearly on $\smoment$ (and by \eqref{eq:true_moment}
on $V$). So we confirmed the claims from the problem statement. \ilarxiv{We provide a brief illustration of how $\op{E}$ encodes
the dynamics and then }\ilpub{We next }discuss nominal stability and LQR, before introducing the system identification scheme.

\begin{arxiv}
    \begin{example} \label{ex:smdyn}
        Continuing on from \cref{ex:modeltensor}, we now assume $\ten{M} = \ten{V}$ and $W = I$. 
        We will examine $\op{E}$. First note that
        \begin{align*}
            \ten{M}_{(1)} &= \left[
                \begin{array}{ccc|ccc|ccc}
                    2 & 0 & 0 & 0 & 3 & 0 & 0 & 0 & 0 \\
                    1 & 0 & 0 & 0 & 0 & 0 & 0 & 1 & 2
                \end{array}
            \right]  \\ &= \left[
                \begin{array}{cc|cc|cc}
                    A_1 & B_1 & A_2 & B_2 & A_3 & B_3
                \end{array}
            \right].
        \end{align*}
    \end{example}
    If we partition $Z_t = [X_t, V_t; \trans{V_t}, U_t]$ with $X_t = \E[x_t \trans{x_t}]$, $V_t = \E[x_t \trans{u_t}]$ and $U_t = \E[u_t \trans{u_t}]$ then 
    \begin{align*}
        \op{E}(Z_t) &= \ssum_{i=1}^{3} \begin{bmatrix}
            A_i, \, B_i
        \end{bmatrix} Z_t \begin{bmatrix}
            \trans{A}_i; \, \trans{B}_i
        \end{bmatrix} \\
        &= \ssum_{i=1}^3 A_i X_t \trans{A}_i + A_i V_t \trans{B}_i + B_i \trans{V}_t \trans{A}_i + B_i U_t \trans{B}_i.
    \end{align*}
    For $u_t = 0$ this reduces to  $\sum_{i=1}^3 \E[(A_i x_t) \trans{(A_i x_t)}]$.

    The matrix associated with $\op{E}$ as described in \cref{cor:cpadj} is given as follows:
    \begin{equation*}
        \opm{E} = Q_2 \begin{bmatrix}
            4 & 0 & 0 & 0 & 9 & 0 & 0 & 0 & 0 \\ 
            2 & 0 & 0 & 0 & 0 & 0 & 0 & 0 & 0 \\
            2 & 0 & 0 & 0 & 0 & 0 & 0 & 0 & 0 \\
            1 & 0 & 0 & 0 & 1 & 2 & 0 & 2 & 4
        \end{bmatrix} \trans{Q_2}.
    \end{equation*}
    with 
    \begin{equation*}
        Q_2 = \begin{bmatrix}
            1 & 0 & 0 & 0 \\ 0 & 1/\sqrt{2} & 1/\sqrt{2} & 0 \\ 0 & 0 & 0 & 1
        \end{bmatrix},
    \end{equation*}
    which is added as prescribed in the definition of $\skron$.    
\end{arxiv}



\subsection{Stability}
We consider stability of the autonomous case of \eqref{eq:dynten}:
\begin{equation} \label{eq:aut-dyn}
    x_{t+1} = \tucker{\ten{A}; I_{n_x}, x_t, {w}_t}.
\end{equation}
Stability of stochastic systems is intimately related to 
convergence of random variables, which can be defined in several ways \cite[\S7.2]{Grimmett2001d}. We specifically consider convergence to zero
of the second moment \cite[Def.~1.a]{Grimmett2001d},
\begin{definition}
    Consider the autonomous system \eqref{eq:aut-dyn}. It is considered \emph{mean-square stable (MSS)} if
    \begin{equation*}
        \E[x_t \trans{x_t}] \to 0 \quad \text{as} \quad t \to \infty.
    \end{equation*}
\end{definition}

From the definition, MSS is equivalent to the stability of 
\begin{equation*}
    X_{t+1} = \op{F}(X_t) \dfn \ten{A}_{(1)} (W \kron X_t) \trans{\ten{A}_{(1)}}, 
\end{equation*}
where $X_t = \E[x_t \trans{x_t}]$ as before and $\op{F}$
is the autonomous version of $\op{E}$. Clearly this system is stable 
for all $X_0 \sgeq 0$ iff $\rho(\op{F}) < 1$ by \cref{lem:stabcp}. Therefore:
\begin{corollary}
    The system \eqref{eq:aut-dyn} is MSS iff $\rho(\op{F}) < 1$. 
\end{corollary}
We can verify stability with the Lyapunov condition \cref{prop:lyapcp} and, as mentioned earlier, stability of \eqref{eq:aut-dyn}
is \update*{related} with robust stability of the switching system with modes $\{A_i\}_{i=1}^{n_w}$ 
in the sense of \cref{thm:stab-eq}. This \update*{relation} was observed before in \cite{Gravell2020} and \cite{Bernstein1987}, 
yet not fully characterized.

\subsection{Linear Quadratic Regulation} \label{sec:lqr}
The goal of this section is twofold \emph{(i)} we argue how LQ control of CP dynamics is equivalent 
to \cref{eq:slqr}; and \emph{(ii)} we show how the optimal policy is determined by solving a SDP,
generalizing the result of \cite{Balakrishnan2003}.  

We start by stating the CP equivalent of \cref{eq:slqr}
\begin{equation} \label{eq:lqrcp}
    \begin{alignedat}{2}
        &\minimize_{Z_t} &\quad& \sum_{t=0}^\infty \tr[Z_t H] \\
                             &\stt && Z_t = \begin{bmatrix}
                                 X_t & V_t \\ \trans{V_t} & U_t
                             \end{bmatrix} \sgeq 0, \\&&&X_{t+1} = \op{E}(Z_t), \quad \forall t \in \N,
    \end{alignedat} \tag{$\mathcal{LQR}_{\mathrm{cp}}$}
\end{equation}
with $H \dfn \blkdiag(Q, R) \sgt 0$. The psd. constraint on $Z_t$ ensures that it acts 
like the second moment of a random vector. 

We show that \eqref{eq:slqr} is equivalent to \eqref{eq:lqrcp}
\begin{proposition} \label{thm:cpmulteq}
    Let $\mathrm{Val}(\cdot)$ denote the optimal value. Then 
    \begin{equation*}
        \mathrm{Val}\eqref{eq:slqr} = \mathrm{Val}\eqref{eq:lqrcp},
    \end{equation*}
    given \cref{asm:model}. Moreover sequence $(Z_t)_{t \in \N}$ is feasible for \eqref{eq:lqrcp} iff there is a feasible sequence
    in \eqref{eq:slqr} $u_t = K_t x_t + \delta_t$ s.t.
    $Z_t = \E[z_t \trans{z_t}]$ where $z_t = (x_t, u_t)$. 
    Here $K_t \in \Re^{n_x \times n_u}$
    and $\delta_t \in \Re^{n_u}$ is a random vector $\forall t \in \N$.
\end{proposition}
\begin{proof}
    Deferred to Appendix~\ref{app:lqr}.
\end{proof}

To find the exact solution of \eqref{eq:lqrcp}, similarly 
to usual LQR, we need to consider a Riccati equation:
\begin{equation} \label{eq:ric-def}
    \op{R}(P) \dfn Q + \adj{\op{F}}(P) - \trans{(\adj{\op{H}}(P))}(R + \adj{\op{G}}(P))^{-1}\adj{\op{H}}(P).
\end{equation}
where $\adj{\op{F}}(P) = \ten{A}_{(2)} (\smoment \kron P) \trans{\ten{A}_{(2)}}$, 
$\adj{\op{G}}(P) = \ten{B}_{(2)} (\smoment \kron P) \trans{\ten{B}_{(2)}}$ and 
$\adj{\op{H}}(P) = \ten{B}_{(2)}(\smoment \otimes P) \trans{\ten{A}_{(2)}}$ with tensors $\ten{A}$ and $\ten{B}$ 
as in \cref{eq:dynten}. \update*{We use adjoint CP operators to be consistent with the $2$-mode matricizations as they occur in \cref{cor:cpadj}.}

\begin{remark}
    The operators $\adj{\op{F}}$, $\adj{\op{H}}$ and $\adj{\op{G}}$ form
    the blocks of $\adj{\op{E}}$ when it is partitioned similarly to $Z_t$ (cf. \cref{lem:opadj}). 

    In the deterministic setting with $n_w = 1$ and $w_t = 1$ (so $\E[W] = 1$), we have $\adj{\op{F}}(P) = \trans{A}_1 P A_1$, $\adj{\op{H}}(P) = \trans{B}_1 P A_1$ and $\adj{G}(P) = \trans{B}_1 P B_1$.
    Hence we recover the classical Riccati operator as a special case. 
\end{remark}

We introduce the \emph{closed-loop operators} as
\begin{equation} \label{eq:clops}
    \Pi_K(X) \dfn [I; K] X [I, \trans{K}], \, \op{E}_K(X) \dfn \op{E}(\Pi_K(X)),
\end{equation}
The policy $Z_t = \Pi_K(X_t)$ is recovered when taking $u_t = Kx_t$ 
in the equivalent \eqref{eq:slqr} (cf. \cref{thm:cpmulteq}).

We can then state the following theorem.

\begin{theorem} \label{thm:cplqr}
    \shorten*{Assume there is some $K$ such that $\rho(\op{E}_K) < 1$ and that $H \dfn \blkdiag(Q, R) \sgt 0$.} Then 
    $Z_t = \Pi_{K_{\star}}(X_t)$ is \update*{an} optimal policy, 
        with $K_{\star} = -(R + \adj{\op{G}}(P_{\star}))^{-1} \adj{\op{H}}(P_{\star})$ and $P_{\star} \sgt 0$
        solves the Riccati equation $P_{\star} = \op{R}(P_\star)$. Moreover
    \begin{enumerate}
        \item \label{thm:cplqr:b} $\rho(\op{E}_{K_\star}) < 1$; 
        \item $\mathrm{Val}\eqref{eq:lqrcp} = \tr[P_{\star} X_0]$;
    \end{enumerate}
\end{theorem}
\begin{proof}
    \ilarxiv{See Appendix~\ref{app:lqr} for the full proof.}\ilpub{This result was shown for multiplicative noise in \cite{Coppens2019, Morozan1983} 
    and generalized to the CP setting by \cref{thm:cpmulteq}. Nonetheless, a full proof is included in the technical report \citear{}.}
\end{proof}

\begin{prepupdate}
    \begin{remark} \label{rem:stabilizability}
        Here the existence of a $K$ such that $\rho(\op{E}_K(\cdot)) < 1$ 
        is referred to as \emph{stabilizability}. Among other things 
        it implies the system being below the \emph{uncertainty threshold} \cite{Athans1977}.
    \end{remark}
\end{prepupdate}

\begin{arxiv}
    \begin{example}
        Let us consider a deterministic system with $n_w = 1$ and
        \begin{equation*}
            A_1 = \begin{bmatrix}
                1 & 0.1 \\ 
                0 & 1
            \end{bmatrix}, \, B_1 = \begin{bmatrix}
                0 \\ 0.1
            \end{bmatrix}.
        \end{equation*}
        This will help us understand how \cref{eq:lqrcp} generalizes the usual notions of LQR. We pick $H = \diag(1, 1, 0.1)$. Then 
        the solution described by \cref{thm:cplqr} is approximately:
        \begin{equation*}
            K = [-1.75, -2.64].
        \end{equation*}
        One can easily verify that this is the usual LQR solution in the deterministic setting. Let's spend some time noting what happens when 
        we set $X_0 = x_0 \trans{x_0}$. In that case, for some $u_0$, let $Z_0 = (x_0, u_0) \trans{(x_0, u_0)}$. Then (similarly to \cref{ex:smdyn}):
        \begin{align*}
            \op{E}(Z_0) &= A_1 x_0 \trans{x_0} \trans{A}_1 + A_1 x_0 \trans{u_0} \trans{B_1} \\ &\qquad + B_1 u_0 \trans{x_0} \trans{A}_1 + B_1 u_0 \trans{u_0} \trans{B_1}.
        \end{align*}
        It is not difficult to verify that this equals $x_1 \trans{x_1}$ for $x_1 = A_1 x_0 + B_1 u_0$. We could repeat this for each time step 
        noting that $X_t = x_t \trans{x_t}$ and $Z_t = (x_t, u_t) \trans{(x_t, u_t)}$ in \cref{eq:lqrcp}. Trivially we have $Z_t \sgeq 0$,
        so the trajectory is feasible. Moreover the cost becomes $\tr[Z_t H] = \trans{x_t} I x_t + 0.1 \trans{u_t} I u_t$. So we can see 
        that the classical LQR controller is at least a feasible solution to \cref{eq:lqrcp} where we then have 
        \begin{align*}
            Z_t = \Pi_K(X_t) &= \begin{bmatrix}
                x_t \trans{x_t} & x_t \trans{x}_t \trans{K} \\ K x_t \trans{x_t} & K x_t \trans{x}_t \trans{K}
            \end{bmatrix} \\ &= (x_t, Kx_t) \trans{(x_t, Kx_t)}.
        \end{align*}
        Note that we have shown in \cref{thm:cplqr} that such policies are in fact optimal. In terms of \eqref{eq:slqr} this 
        implies that, a random policy $u_t = Kx_t + \delta_t$ for some random $\delta_t$ will not improve performance. 
    \end{example}
\end{arxiv}

Even though we directly tackled second moment dynamics in \cref{thm:cplqr}, 
the optimal policy $Z_t = \Pi_{K_{\star}}(X_t)$ 
is realizable for multiplicative noise dynamics using 
$u_t = K_{\star} x_t$. So this controller is also optimal 
for \eqref{eq:slqr} as is well known (cf. \cite[Prop.~3]{Coppens2019})
and achieves the equality in \cref{thm:cpmulteq}. 

Thus we can conclude that LQR of CP dynamics
is (i) identical to LQR of multiplicative noise dynamics; (ii) 
a natural generalization of classical LQR; and (iii) a relaxation for LQR of 
switching systems in the sense of \cref{thm:stab-eq}. 

The final question is how to find a $P_\star$ satisfying \eqref{eq:ric-def}. Numerical solution 
of such equations is considered in detail in \cite{Damm2003}. We consider a reformulation
as a SDP, akin to \cite{Balakrishnan2003}. 

\begin{theorem} \label{thm:ric-sdp-primal}
    \update*{For $\adj{\op{F}}$, $\adj{\op{H}}$ and $\adj{\op{G}}$ as in \eqref{eq:ric-def}, consider
    \begin{equation} \label{eq:ric-sdp-primal}
        \begin{alignedat}{2} 
            &\maximize_{P \sgeq 0} &\,\,& \tr[P X_0] \\
            &\stt && \begin{bmatrix}
                P - \adj{\op{F}}(P) & -\trans{(\adj{\op{H}}(P))} \\
                -\adj{\op{H}}(P) & -\adj{\op{G}}(P)
            \end{bmatrix} \sleq \begin{bmatrix}
                Q \\ & R
            \end{bmatrix},
        \end{alignedat}
    \end{equation}
    with $Q \sgt 0$ and $R \sgt 0$. Then 
    \begin{enumerate}
        \item \cref{eq:ric-sdp-primal} is bounded iff $\exists K \colon \rho(\op{E}_K) < 1$;
        \item For bounded \cref{eq:ric-sdp-primal}, the $P_\star$ that solves \cref{eq:ric-def} is an optimizer.
        Moreover, for $X_0 \sgt 0$, $P_\star$ is the unique solution.
    \end{enumerate}}
\end{theorem}
\begin{proof}
    The proof is analogous to the one in \cite{Balakrishnan2003} and \cite[Prop.~3]{Coppens2019}. \ilarxiv{The full proof is deferred to Appendix~\ref{app:lqr}.}\ilpub{A full proof in the CP case is deferred to \citear{}.}
\end{proof}

\update*{So the SDP in \cref{thm:ric-sdp-primal} both provides a way of solving the Riccati equation and a way to verify stabilizability.}

\section{Identifying the second moment} \label{sec:construction}

Our goal in this work is to generalize \cref{thm:cplqr} to the data-driven DR setting \eqref{eq:drlqr}, while retaining stability
for a linear controller. This problem was considered in \cite{Coppens2019}, for sub-Gaussian normalized disturbance $\xi$
(i.e. $\xi \dfn \V[w]^{-1/2}(w - \E[w])$). Therefore \cite{Coppens2019} supports among others Gaussian $w$ or bounded $\xi$. 
This setup had two main limitations: \emph{(i)} bounded $w$ does not imply bounded $\xi$,
so the link with classical robust control (as exploited for additive noise in \cite{Coppens2021}) is lost; 
\emph{(ii)} the setup is inapplicable when only state measurements are available. This section resolves 
these issues.


\subsection{Measurement model} \label{sec:measurement-model}
Considering the dynamics \eqref{eq:dynten}, the goal is to estimate the 
\emph{true}\footnote{We will henceforth
denote the \emph{true} values using a star subscript.} 
$\smoment_\star \dfn \E[w \trans{w}]$. \shorten*{This should be accompanied by a similar 
concentration inequality as in \cref{lem:dd-moment}} without assuming direct access to samples $(w_i)_{i=1}^N$. Instead
we measure a sequence of $N$ independent samples $(x_{i+1}, z_i)_{i=1}^N$, with $z_i = (x_i, u_i)$ the augmented state,
satisfying the measurement model:
\begin{equation} \label{eq:measurement-model}
    x_{i+1} = \tucker{\ten{M}; I_{n_x}, z_i, w_i},
\end{equation}
where $(w_i)_{i=1}^N$ denotes the noise sequence. 

There are two important properties that \eqref{eq:measurement-model} should satisfy:
\emph{(i)} the `span' of the \emph{(mixed) model tensor} $\ten{M}$ should describe the true dynamics (i.e. \cref{asm:model}); and \emph{(ii)}
the data should be sufficiently rich to render \update*{$\smoment_\star$ observable. This final property is 
akin to \emph{persistency of excitation}. We specifically assume that the support of $z_i$ is not degenerate
as is formalized in the following assumption.}
\begin{assumption}[Nondegenerate support] \label{asm:data}
    We assume that $\{z_i\}_{i\in\N_{1:\nsample}}$ is a sequence of i.i.d. copies of a random vector $z$,
    with a distribution that is dominated by the Lebesgue measure.
\end{assumption}
A direct consequence of \cref{asm:data} is that\ilpub{\footnote{More details on \cref{asm:data} are given in \citear{}.}}\ilarxiv{, by \cref{lem:rank-condition}}, if $N \geq n_z$,
\begin{equation} \label{eq:rank-condition}
    \prob[\rk([z_1, \dots, z_N]) = n_z] = 1.
\end{equation}

\begin{remark} \label{rem:asm}
    Independence is guaranteed for each $i$ by either \emph{(i)} taking a known random 
    $z_i = (x_0, u_0)$, updating \eqref{eq:dyn} once and taking $x_{i+1} = x_1$; or \emph{(ii)} 
    taking a known $x_0$, updating \eqref{eq:dyn} $T$ times for random $u_t$ and taking $x_{i+1} = x_T$ and 
    $z_i = (x_{T-1}, u_{T-1})$.\todo{ctrb and $\Xi$?} 
    We refer to \emph{(i)} as \emph{random initialization} and to 
    \emph{(ii)} as \emph{rollout} similar to \cite{Xing2020, Dean2019}.

    For rollout it is difficult to verify \eqref{eq:rank-condition}. Usually a notion of controllability and persistency of excitation 
    is employed. Full generalizations of these concepts for multiplicative noise do not exist to our knowledge. 
    Note however that the event in \eqref{eq:rank-condition} 
    is related to the zeros of a polynomial. So an argument similar to the one in \cref{prop:obsv} below is applicable. In practice one can sample $u_t$
    from a distribution dominated by the Lebesgue measure and test whether \eqref{eq:rank-condition} holds \update*{for $N = n_z$}. 
    \update*{When taking more samples, the validity of \cref{eq:rank-condition} will stay the same with probability one.}
\end{remark}

\subsection{Noise observability}
In some cases, the disturbances can be inferred exactly from state measurements. 
\update*{Specifically in \eqref{eq:measurement-model} we can expand the definition of the Tucker product as in \cref{rem:tucker} and plug in the definition of $\btimes{2}$ to get}:
\begin{equation} \label{eq:bilinear-form-alt}
    x_{i+1} = (\ten{M} \btimes{2} z_i) \btimes{2} w_i = (\ssum_{j=1}^{n_z} [z_i]_j [\ten{M}]_{:j:}) w_i.
\end{equation}
Thus if $(\ten{M} \btimes{2} z_i)$ is left invertible, we can uniquely identify $(w_i)_{i=1}^N$ 
from $(x_{i+1}, z_i)_{i=1}^N$. Invertibility of $(\ten{M} \btimes{2} z_i)$ however is analogous to invertibility of a 
linear subspace of matrices (or linear forms),
which is an unsolved problem in general (cf. \cite{Testa2018} for $3 \times 3$ matrices).
We have the following:
\begin{proposition} \label{prop:obsv}
    Let $z$ denote a random vector with non-degenerate support (cf. \cref{asm:data})
    and $n_x \geq n_w$. Then, $(\ten{M} \btimes{2} z)$ is left invertible either with probability one or zero. 
\end{proposition}
\begin{proof}
    A high-level proof is given in \cite[p.2]{Lovasz1989} in the setting of invertibility of linear forms.
\end{proof}
This suggests that it is often sufficient to sample $(\ten{M} \btimes{2} z)$ for a random 
$z$ and check its invertibility. \update*{The conclusion will then generalize to other realizations, except for a set of measure zero.}
\update*{When invertibility holds, \cref{eq:bilinear-form-alt} can be solved for ${w}_i$ exactly.
Then the following holds:}
\begin{lemma} \label{lem:dd-moment}
    Let $\W = \left\{ w \in \Re^{n_w} \colon \nrm*{w}_2 \leq r_w \right\}$. 
    Assume we have a set of i.i.d. samples $\{w_i\}_{i\in\N_{1:M}}$ of a random vector $w$ and let
    $\smomenth \dfn \ssum_{i=1}^{\nsample} w_i \trans{w_i}/\nsample$.    
    Then
    \begin{equation*}
            \prob[\nrm{\smomenth - \smoment_\star}_2 \leq \beta_{\cmoment}] \geq 1-\delta
    \end{equation*}
    when $\beta_{\smoment} = r^2_w \sqrt{2\ln(2n_w/\delta)/\nsample}$.
\end{lemma}
\begin{proof}
    The result follows from a direct application of \cref{lem:mathfd} by noting that 
    the terms in the error satisfy $-r_w^2 I \sleq -w_i \trans{w_i} \sleq W_\star - w_i \trans{w_i} \sleq W_\star \sleq r_w^2I$. 
\end{proof}

We also want to estimate $\smoment_\star$ in the other case (i.e., when $n_x < n_w$). 
For example, for a model-free basis we have $n_w = n_x n_z > n_x$.  
To do so we design least squares estimators. 

\subsection{Least squares estimator}
To \update*{design} a \emph{least square (LS)} estimator, we need to convert \eqref{eq:measurement-model}
into an expression linear in $w_i \trans{w_i}$. We can do so by using \cref{lem:tensor-kron} on \cref{eq:measurement-model} \update*{and use $I_{n_x} \skron I_{n_x} = I_{\sd{n_x}}$}, which gives:
\begin{equation} \label{eq:smoment-model}
    x_{i+1} \skron x_{i+1} = \tucker{\ten{M} \skron \ten{M}; \update*{I_{\sd{n_x}},\,} z_i \skron z_i, w_i \skron w_i}.
\end{equation}
Note $\E[w_i \skron w_i] = \E[\svec(w_i \trans{w_i})] \nfd \svec(\smoment_\star)$. We can write this in terms of linear 
equations by expanding the Tucker operator \update*{(cf. \cref{def:tucker} and \cref{rem:tucker})}:
\begin{align*}
    x_{i+1} \skron x_{i+1} &= (\ten{W} \btimes{2} (z_i \skron z_i))  \svec(W_\star) \\ &\qquad + (\ten{W} \btimes{2} (z_i \skron z_i)) \eta_i,
\end{align*}
with $\ten{W} = \ten{M} \skron \ten{M}$ and noise $\eta_i = w_i \skron w_i - \svec(W_\star)$. 
 
As in usual least squares, stacking the equations gives:
\begin{equation} \label{eq:smoment-stacked}
    Y_N = \opm{Z}_N \svec(W_\star) + E_N,
\end{equation}
where we stack that data as follows%
\footnote{\update*{Applying properties of tensors and $\skron$ (cf. \cref{eq:bilinear-form-alt}) shows
\[\ten{W} \btimes{2} (z \skron z) = \left(\sum_{i=1}^{n_z} [\ten{M}]_{:i:} [z]_i\right) \skron \left(\sum_{i=1}^{n_z} [\ten{M}]_{:i:} [z]_i\right).\]
with $\ten{M}_{:i:} = [[A_1 \, B_1]_{:i}\, \dots \, [A_{n_w}\, B_{n_w}]_{:i}]$ for $i \in [n_z]$.}}
\begin{subequations}
    \begin{align}
        Y_N &= (x_{2} \skron x_2, \dots, x_{N+1} \skron x_{N+1}) \nonumber\\ 
        \opm{Z}_N &= [\ten{W} \btimes{2} (z_1 \skron z_1); \dots; \ten{W} \btimes{2} (z_N \skron z_N)] \label{eq:zmatdef}\\
        E_N &= (\tucker{\ten{W}; z_1 \skron z_1, \eta_1}, \dots,\tucker{\ten{W}; z_N \skron z_N, \eta_N}).\label{eq:ematdef}
    \end{align}
\end{subequations}

Note that $\E[\eta_i] = 0$ so $\E[E_N] = 0$, which suggests the LS estimate:
\begin{equation} \label{eq:ls-estimator}
    \svec(\smomenth) = \pinv{\opm{Z}_N} Y_N. 
\end{equation}
A similar estimator to $\cmomenth$ is used in compressed covariance sensing \cite{Romero2013}, where 
it is observed that it acts as a good heuristic for the \emph{maximum likelihood estimator} when $w$ is Gaussian. Noting that each element of the vector
equation \eqref{eq:smoment-model} constitutes a bilinear form of $z_i$, reveals a connection with bilinear estimation \cite{Kukush2003}. There, adjusted LS
estimators exist that are consistent even when $z_i$ is perturbed by random noise. For simplicity we consider exact state measurements here, 
which enables the use of ordinary LS, for which it is convenient to derive concentration inequalities. 

\subsection{Error analysis} \label{sec:error-analysis}
We analyze the error of the LS estimator \eqref{eq:ls-estimator} under \cref{asm:model}--\ref{asm:data}. 
The discussion \shorten*{differs somewhat} from the classical case \shorten*{due to the biased estimate \eqref{eq:ls-estimator}}. 
\shorten*{However, this bias is inconsequential as our focus is on} estimating the second moment dynamics. 
We demonstrate that the estimate $\op{E}(\smomenth; \cdot)$ \shorten*{captures the second moment dynamics
$\op{E}_\star \dfn \op{E}(\smoment_\star; \cdot)$ without bias.} This phenomenon is \shorten*{commonly observed} in system identification 
of multiplicative noise, \shorten*{as noted in \cite{Xing2021,Di2021}.}
\shorten*{Since this complicates the error analysis, we provide only the intuition here,
with formal statements deferred to Appendix~\ref{app:identification}.}

The overall estimation error is given as 
\begin{equation} \label{eq:error-model-full}
    \svec(\smomenth - \smoment_\star) = (I - \pinv{\opm{Z}_N} \opm{Z}_N) \svec(\smoment_\star) + \pinv{\opm{Z}}_N E_N. 
\end{equation}

We first discuss the nuances associated with the first term on a high level, before formally stating the error bound. 

\paragraph*{Estimator bias}
Noting that the second term is zero mean, the first term in \cref{eq:error-model-full} constitutes
the bias in the estimate. This bias is characterized exactly in \cref{lem:kernel-zn} and is often non-zero. 
The reason for this can be given in terms of the matrix associated with $\op{E}(\smoment; \cdot)$ denoted as 
$\opm{E} \in \Re^{\sd{n_x} \times \sd{n_z}}$, which uniquely determines the second moment dynamics. Our procedure parametrized $\opm{E}$ in terms of $W \in \sym{n_w}$. 
In other words, we use $\sd{n_w}$ parameters to describe $\sd{n_x} \sd{n_z}$ degrees of freedom. 
In the model-free case (i.e. when $n_w = n_z n_x$) we would always over parametrize $\opm{E}$ for $n_x > 1$. 



One might therefore argue that $\opm{E}$ should be parametrized directly. We instead use $\smoment$
for three reasons: \emph{(i)} often a description like \eqref{eq:dyn} with a bound on $\nrm{w}_2$ is more 
natural in practical applications and such bounds do not translate well into bounds on $\opm{E}$; \emph{(ii)} 
control synthesis using a confidence set over $\smoment$ is more tractable; and \emph{(iii)} the bias
is inconsequential for the LQR cost and stability analysis. 

Here \emph{(iii)} is caused by the fact that the bias always lies in the kernel of the second moment dynamics (cf. \eqref{eq:dynsm_})
$\op{E}(W; Z)$ w.r.t. $W$. This kernel in fact is also the cause of the bias in the first place as formalized in \cref{lem:kernel-sm}. 


\paragraph{Data-driven error}
We are now ready to state a data-driven bound, suitable for control synthesis. 
To do so we first consider some auxiliary operators. Letting $\ten{M}$ denote the model tensor (cf. \cref{asm:model}) and
$\ten{W} \dfn \ten{M} \skron \ten{M}$ as before, we introduce\update*{
\begin{align*}
    \op{W}(z\trans{z}) &\dfn \trans{(\ten{W} \ttimes{2} (z \skron z))} (\ten{W} \ttimes{2} (z \skron z)) \\
              &= \ten{W}_{(3)} ( (z\trans{z} \skron z\trans{z}) \kron I_{\sd{n_x}}) \trans{\ten{W}_{(3)}},
\end{align*}
where the equality is shown in \cref{lem:wop:a}.} Moreover, let
\begin{equation} \label{eq:hopmat}
    \opm{H}_i \dfn \pinv{\left[ \ssum_{j=1}^N \op{W}\left( z_j \trans{z_j} \right) \right]} \op{W}\left(  z_i \trans{z_i} \right),
\end{equation}
where $\op{H}_i \in \slop{n_z}{n_w}$, the operator associated with $\opm{H}_i$ has norm $\nrm{\op{H}_i}_2 \leq \sqrt{n_w} \nrm{\opm{H}_i}_2$ (cf. \cref{lem:sopbnd}). 
One can show that (cf. \cref{cor:error-model-wop})
the second term in \eqref{eq:error-model-full} equals
\begin{equation*}
    \svec\left(\ssum_{i=1}^N \op{H}_i(w_i \trans{w_i} - W_\star)\right),
\end{equation*}
which is an i.i.d. sum of bounded random matrices. Hence a matrix Hoeffding bound is applicable:

\begin{theorem} \label{thm:error-bound-full}
    Let $\op{E}_\star$ describe the true moment dynamics as in \cref{eq:dynsm_}. Assume access to samples $(x_{i+1}, z_i)_{i=1}^N$
    satisfying \eqref{eq:measurement-model} and \cref{asm:data} and a model tensor $\ten{M}$ satisfying \cref{asm:model}. 

    If $\smomenth$ is the LS estimate \eqref{eq:ls-estimator} and $N \sgeq \sd{n_z}$. Then
    \begin{equation*}
        \prob[\op{E}(\ubar{W}; Z) \sleq \op{E}_\star(Z) \sleq \op{E}(\bar{W}; Z)] \geq 1-\delta
    \end{equation*}
    for all $Z \in \sym{n_z}$ with $\ubar{W} \dfn \smomenth - \beta_W I_{n_w}$ and $\bar{W} \dfn \smomenth + \beta_W I_{n_w}$ 
    where $\beta_W = r_w^2 \zeta_W \sqrt{2 \log(2n_w/\delta)}$ and $\zeta_W^2 \dfn \ssum_{i=1}^N \nrm{\op{H}_i}_2^2$,
    where $\op{H}_i$ denote the operators associated with $\opm{H}_i$ in \eqref{eq:hopmat}. 
\end{theorem}
\begin{proof}
    Deferred to Appendix~\ref{app:identification}.
\end{proof}

\paragraph{Sample complexity} Due to independence one may expect the matrices $\op{W}(z_i \trans{z_i})$ to behave 
similarly. Hence, one can expect $\nrm{\op{H}_i}_2$ to decrease with $1/N$. So $\zeta_W^2 = {\ssum_{i=1}^N \nrm{\op{H}_i}_2^2}$
would decrease with $1/N$ too. We formalize this intuition in a simplified setting below.
\begin{lemma} \label{lem:moment-sample-complexity}
    Following the setting of \cref{thm:error-bound-full}
    and assuming additionally that $\nrm{z_i}_2 \leq r_z$ a.s. for all $i \in \N_{1:N}$. Then, if $N \geq (\nrm{\ten{W}}_2^2 \tau_{\cmoment}/ \gamma_{\cmoment})^2$,
    \begin{equation*}
        \prob\left[\zeta_{\cmoment} \leq \frac{\update*{\sqrt{n_w}} \nrm{\ten{W}}_2^2}{\sqrt{N} \gamma_{\cmoment} - \nrm{\ten{W}}_2^2 \tau_{\cmoment}}\right] \leq \delta,
    \end{equation*}    
    with $d_{\cmoment} = \rk(\ten{W}_{(3)})$, $\tau_{\cmoment} \dfn \sqrt{2\ln(2 d_{\cmoment}/\delta)}$
    and \[\gamma_{\cmoment} \dfn \lambda_{d_{\cmoment}}\left[\ten{W}_{(3)} \left({\E[z \trans{z} \otimes z \trans{z}]}/{r_z^4} \otimes I_{\update*{\sd{n_x}}}\right) \trans{\ten{W}}_{(3)}\right],\]
    with $\ten{W} = \ten{M} \skron \ten{M}$.
\end{lemma}
\begin{proof}
    Deferred to Appendix~\ref{app:identification}.
\end{proof}

\begin{remark}
    Details on bounding the \emph{tensor spectral norm} $\nrm{\ten{W}}_2$ are given in \cite{Chen2020a}.
\end{remark}

\begin{remark} \label{rem:mixed}
    Note that convergence of $\zeta_W$ depends on the kurtosis of $z$ (i.e. the fourth order moment) through $\gamma_W$, which can be difficult to quantify in practice.
    We observe that the value of $\gamma_W$ is invariant to scaling of $z$ and is larger 
    when $z \skron z$ is spread out close to the boundary of its support, which will imply faster convergence. 
\end{remark}

\section{Distributionally robust control synthesis} \label{sec:solution}
In this section we design approximate reformulations of \eqref{eq:drlqr}. 
The goal is to preserve the properties of \cref{thm:cplqr}.
Specifically we want to synthesize a distributionally robustly stabilizing controller, by solving 
a SDP as in \cref{thm:ric-sdp-primal}.
The synthesis procedure is described in \cref{sec:synthesis}. We also \shorten*{evaluate} the
sample complexity in \cref{sec:analysis}. 

\subsection{Synthesis} \label{sec:synthesis}
In this section, we use our confidence bound in \cref{thm:error-bound-full}
to synthesize a controller that stabilizes the true system with high probability. 
Similarly, to \cref{sec:lqr} we do so by investigating \eqref{eq:lqrcp}. Specifically
consider the CP version of \eqref{eq:drlqr}:
\begin{equation} \label{eq:drlqrcp}
    \begin{alignedat}{2}
        &\minimize_{Z_t \sgeq 0} \,\, \maximize_{\ubar{W} \sleq W \sleq \bar{W}} &\quad& \ssum_{t=0}^\infty \tr[Z_t H].
    \end{alignedat} \tag{$\bar{\mathcal{LQR}}_{\mathrm{cp}}$}
\end{equation}
subject to $Z_t = [X_t, V_t; \trans{V_t}, U_t] \sgeq 0$, $X_{t+1} = {\op{E}}(W; Z_t)$ and
with $\ubar{W}$ and $\bar{W}$ as in \cref{thm:error-bound-full}. The constraints on $W$ define the equivalent of the 
ambiguity set in \cref{eq:drlqr}.

Next we solve \eqref{eq:drlqr} by generalizing \eqref{thm:cplqr}. 
We will use the closed-loop policy $\Pi_K$ as defined in \cref{eq:clops} 
and:
\begin{equation*}
    \op{E}_K(W; \cdot) = \op{E}(W; \Pi_K(\cdot)). 
\end{equation*}

We have the following:
\begin{theorem} \label{thm:drlqrcp}
        Assume $H = \blkdiag(Q, R) \sgt 0$ and 
        \begin{equation*}
            \exists K: \rho({\op{E}}_K(W; \cdot)) < 1, \,\forall W \colon \ubar{W} \sleq W \sleq \bar{W}.
        \end{equation*}\
        Then the policy $Z_t = \Pi_{\bar{K}}(X_t)$ 
        is optimal for \eqref{eq:drlqrcp}, with $(\bar{P}, \bar{K})$ the optimal solution 
        as described in \cref{thm:cplqr} where we assume $\op{E} = \op{E}(\bar{W}, \cdot)$.  
        Moreover
        \begin{enumerate}
            \item \label{thm:drlqrcp:a} $\rho(\op{E}_K(W; \cdot)) < 1$, $\forall W \colon \ubar{W} \sleq W \sleq \bar{W}$;
            \item $\mathrm{Val}\eqref{eq:drlqrcp} = \tr[\bar{P} X_0]$. 
        \end{enumerate}
\end{theorem}
\begin{proof}
    \begin{pub}%
    The proof is similar to that of \cite[Prop.~9]{Coppens2019} for multiplicative noise,
    and is based on showing that the Bellman operator is of the same structure as in \cref{thm:cplqr}
    by using \cref{cor:cpmon}. A full proof for the CP case is given in \citear{}.%
    \end{pub}%
    \begin{arxiv}%
        We do not repeat the full proof of \cref{thm:cplqr}. Instead we consider only the Bellman operator and 
        show that it has a similar structure to the one of \eqref{eq:lqrcp}. 

        The analogous Bellman operator to \eqref{eq:bellman} is given as 
        \begin{equation} \label{eq:drbellman}
            \minimize_{Z \sgeq 0} \, \maximize_{\ubar{W} \sleq W \sleq \bar{W}} \, \{
                \tr[Z H] + \tr[P_{k} {\op{E}}(W; Z)]\}.
        \end{equation}
        Observing that (cf. \cref{cor:cpmon}) ${\op{E}}(Z) = \op{E}(W; Z) \sleq \op{E}(\bar{W}; Z)$
        if $W \sleq \bar{W}$
        shows that $\op{E}(\bar{W}; Z)$ achieves the maximum in \eqref{eq:drbellman}
        $\forall Z$. Therefore, value iteration becomes:
        \begin{equation*}
            \tr[P_{k+1} X] = \min_{Z \sgeq 0} \{
                \tr[Z H] + \tr[P_{k} \op{E}(\bar{W}; Z)]\}.
        \end{equation*}
        The proof is then completed by following the remainder of the proof of \cref{thm:cplqr},
        for $\op{E}(\bar{W}; \cdot)$ instead of $\op{E}(\cdot)$.%
    \end{arxiv}%
\end{proof}

\update*{Stabilizability is replaced by a robust equivalent. When violated, \cref{eq:drlqrcp} is infeasible 
and the SDP in \cref{thm:ric-sdp-primal} is unbounded. 
If this occurs the user should gather more data or verify if the nominal system is stabilizable using 
first principles (e.g. whether the uncertainty threshold is exceeded \cite{Athans1977}). When 
$\op{E}(\ubar{W}; \cdot)$ is not stabilizable (i.e. \cref{eq:ric-sdp-primal} is unbounded), then the true system will not be either with high probability.
In that case, not much can be done from a control perspective besides re-designing actuators.}

\cref{thm:error-bound-full} implies the true $\op{E}_\star$ satisfies $\op{E}(\ubar{W}; \cdot) \sleq \op{E}_\star \sleq \op{E}(\bar{W}; \cdot)$
with probability at least $1 - \delta$. So we have:
\begin{corollary} \label{cor:stab-generalization}
    Let $\bar{K}$ denote the DR policy as in \cref{thm:drlqrcp}
    and assume the setup of \cref{thm:error-bound-full} holds. Moreover,
    let $\op{E}_\star$ denote the true moment dynamics.
    
    Then with probability at least $1-\delta$,
    \begin{enumerate}
        \item $\rho(\op{E}_{\star}(\Pi_K(\cdot))) < 1$;
        \item $\mathrm{Val}\eqref{eq:drlqrcp} \geq \mathrm{Val}\eqref{eq:lqrcp} = \mathrm{Val}\eqref{eq:slqr}$. 
    \end{enumerate}
\end{corollary}
\begin{proof}
    By \cref{thm:error-bound-full}, $\prob[\op{E}(\ubar{W}; \cdot) \sleq \op{E}_\star \sleq \op{E}(\bar{W}; \cdot)] \geq 1-\delta$. 
    Thus, $\exists W \colon \ubar{W} \sleq W \sleq \bar{W}$ such that $\op{E}(W; \cdot) = \op{E}_\star(\cdot)$. 
    So stability \emph{(i)} follows from \cref{thm:drlqrcp:a} and
    \emph{(ii)} follows by maximization over $W$ in \eqref{eq:drlqrcp} for the inequality and
    \cref{thm:cpmulteq} for the equality.
\end{proof}

\begin{remark}
    As in \cref{thm:cpmulteq}, we have $\mathrm{Val}\eqref{eq:drlqrcp} = \mathrm{Val}\eqref{eq:drlqr}$. 
    To save space we omit a proof and instead claim only $\mathrm{Val}\eqref{eq:drlqrcp} \geq \mathrm{Val}\eqref{eq:slqr}$.
    This is sufficient for safety critical applications where an upper bound is sufficient for stability. We show 
    consistency of \eqref{eq:drlqrcp} later to further strengthen the result.
\end{remark}

\subsection{Sample complexity} \label{sec:analysis}
As shown in the previous section, the solution of the DR problem 
\cref{sec:construction} is described by a perturbed Riccati equation. As such, we use 
perturbation analysis to bound the sub-optimality of the DR controller when applied to 
the true dynamics. We leverage a previous result by the authors in \cite{Coppens2020} to do so. 

The final complexity bound is as follows:
\begin{theorem} \label{thm:sample-complexity}
    \shorten*{Assume \cref{asm:model}--\ref{asm:data} \update*{(for $\nrm{z_i}_2$ bounded a.s.)} and let
    $\bar{K}$ be the optimal controller in \cref{thm:drlqrcp}.}
    Then, for sufficient samples $N$ and with probability at least $1-\delta$,
    \begin{equation*}
        \tr\left[\sum_{t=0}^\infty Z_t H\right] - \mathrm{Val}\eqref{eq:lqrcp} = \op{O}\left( \frac{1}{N} \right),
    \end{equation*}
    where the first term is the cost of \cref{eq:lqrcp} achieved for $Z_t = \Pi_{\bar{K}}(X_t)$
    with true moment dynamics $X_{t+1} = \op{E}_\star(Z_t)$. 
    
    Moreover by \cref{thm:cpmulteq}, for sufficient samples $N$ and with probability at least $1-\delta$,
    \begin{equation*}
        \E\left[ \sum_{t=0}^\infty \trans{x_t} Q x_t + u_t R u_t \right] - \mathrm{Val}\eqref{eq:slqr} = \op{O}\left( \frac{1}{N} \right),
    \end{equation*}
    where $x_{t+1} = A(v_t) x_t + B(v_t) u_t$ as in \cref{eq:dyn}, $u_t = \bar{K} x_t$. So the first term 
    is the cost achieved when applying $\bar{K}$ to the true multiplicative noise dynamics \cref{eq:dyn}.  
\end{theorem}
\begin{proof}
    We suggest some modifications of the proof of \cite[Cor.~III.4]{Coppens2020} to show the second result.
    The first result then follows directly from \cref{thm:cpmulteq}.

    Modifications are required as \cite[Asm.~II.2]{Coppens2020} assumes 
    a relative bound like $\ubar{\alpha} W_\star \sleq \hat{W} - W_\star \sleq \bar{\alpha} W_\star$,
    with $\ubar{\alpha} \leq \bar{\alpha} = \mathcal{O}(1/\sqrt{N})$ with high probability, $W_\star$ the true 
    value of $\E[w \trans{w}]$ and $\hat{W}$ some estimate.   
    Instead we have $0 \sleq \bar{W} - W_\star \sleq 2\beta_W I$ with $\beta_W = \mathcal{O}(1/\sqrt{N})$ by 
    \cref{thm:error-bound-full} and \cref{lem:moment-sample-complexity} (where we ignore the effect of the bias,
    without loss of generality due to \cref{lem:kernel-sm}). So the proofs need to be modified. Specifically we replace \cite[Lem.~IV.3]{Coppens2020}
    with a bound like 
    \begin{equation*}
        \nrm{\ten{A}_{(2)} (\bar{W} - W_\star \kron P) \trans{\ten{A}_{(2)}}}_2 \leq 2 \beta_W \nrm{\ten{A}_{(2)} (I \kron P) \trans{\ten{A}_{(2)}}}_2.
    \end{equation*}
    The proof is analogous and is based on \cref{cor:cpmon}. Following the proof of \cite[Lem~IV.3]{Coppens2020} and its dependencies, 
    clearly only the constants in \cite[Table.~1]{Coppens2020} are affected. The rate 
    shown in \cite[Cor.~III.4]{Coppens2020} remains the same. \update*{\ilarxiv{The full proof with explicit constants is given in \cref{app:sample-complexity}.}\ilpub{The full proof with explicit constants is given in \cite[App.~F]{Arxiv}.}}
\end{proof}

\begin{remark}
    Since the estimation error bounds are of the same order, we get the same sample complexity 
    in the certainty equivalent setting \cite[\S3.1]{Astrom2008}, where we use $\hat{W}$ instead of $\bar{W}$. The proof is analogous. 

    \update*{In both cases the rate in terms of $N$ 
    is the same as in the additive noise case \cite{Mania2019,Simchowitz2020}. This implies that, qualitatively in terms of samples, 
    learning additive noise or multiplicative noise is similar. The dimensional dependency of the error however is $n_x^2 (n_x + n_u)$ for $n_x \geq n_u$ and omitting logarithmic terms. 
    This is worse compared to additive noise \cite{Mania2019,Simchowitz2020}, where a lower bound is established of $\sqrt{n_x^2 n_u}$. 
    Note that our bound is not a lower bound, so further research might lead to a tighter dimensional dependency. Exploiting $\ubar{W}$ in \cref{thm:error-bound-full} could be a first step.}
\end{remark}

\begin{arxiv}
\section{Imposing prior structure} \label{sec:structural}
The tools we introduce for identification are also applicable when prior structure is imposed on the dynamics
\eqref{eq:dyn}. Let
\begin{equation} \label{eq:dyn-structured}
    x_{t+1} = \left(\sum_{i=1}^{n_v} A_{i+1} [v_t]_i + A_1\right) x_t + \left(\sum_{i=1}^{n_v} B_{i+1} [v_t]_i + B_1\right) u_t,
\end{equation}
where $A_i$ and $B_i$ are assumed known for $i\in\N_{1:n_v+1}$. We were able to handle the
terms linear in $v_t$ in the previous sections by adequate selection of $\ten{M}$ in \cref{asm:model}. 
However, the introduction of the known terms $A_1$ and $B_1$ are not supported\footnote{They can be supported by assuming $A_1$ and $B_1$ 
are multiplied by an unknown random value. However, by doing so, structural information is lost.}. In this section 
we illustrate how this prior information can be exploited in a DR control scheme. 

In this setting we slightly alter \cref{asm:model}:
\begin{assumption} \label{asm:model-structured}
    Consider $\ten{V} \in \Re^{n_x \times n_z \times n_v + 1}$
    s.t. $[\ten{V}]_{::i} = [A_i, B_i]$ for $i \in \N_{1:n_v+1}$ and $\{A_i, B_i\}_{i=1}^{n_v+1}$ as in \cref{eq:dyn-structured}. 

    Assume $\ten{M} \in \Re^{n_x \times n_z \times n_w}$ is s.t. $[\ten{M}]_{::1} = [\ten{V}]_{::1}$ and 
    \begin{align*}
        &\rk(\widetilde{\ten{M}}_{(3)}) = \rk([\widetilde{\ten{M}}_{(3)}; \widetilde{\ten{V}}_{(3)}]) = n_w,
    \end{align*}
    with $\widetilde{\ten{M}} \dfn [\ten{M}]_{:,:,2:n_w}$ (similarly for $\widetilde{\ten{V}}$) the non-structured part
    of the model tensor.
\end{assumption}
We can also partition $\ten{M}$ similarly as before in $\ten{A}$ and $\ten{B}$ and 
write the dynamics
\begin{align} \label{eq:dynten-structured}
    x_{t+1} &= \tucker{\ten{A}; x_t, w_t} + \tucker{\ten{B}; u_t, w_t} \nonumber \\
    &= \tucker{\ten{M}; z_t, w_t}.
\end{align}

Then, \cref{lem:model-eq} generalizes to
\begin{lemma} \label{lem:model-eq-structured}
    When \cref{asm:model} holds and $w_t = (1, \widetilde{w}_t)$ with
    $\widetilde{w}_t = \pinv{(\trans{\widetilde{\ten{M}}_{(3)}})} \trans{\widetilde{\ten{V}}_{(3)}} v_t$, $\forall t \in \N$.
    Then \eqref{eq:dynten-structured} produce the same trajectories as \eqref{eq:dyn-structured}.
\end{lemma}
\begin{proof}
    The proof is analogous to that of \cref{lem:model-eq}. 
\end{proof}

The second moment dynamics for \eqref{eq:dyn-structured} are then:
\begin{equation} \label{eq:dynsm-structured}
    X_{t+1} = \ten{M}_{(1)} (W \skron Z_t) \trans{\ten{M}}_{(1)} \nfd \op{E}(W; Z_t).
\end{equation}
where $X_t = \E[x_t\trans{x_t}]$ and $Z_t = \E[z_t \trans{z_t}]$ with
\begin{equation*}
    W \dfn \blkdiag(0, \Sigma) + (1, \mu) \trans{(1, \mu)},
\end{equation*}
with $\widetilde{W} \dfn \E[\widetilde{w}_t \trans{\widetilde{w}}_t]$, $\Sigma \dfn \widetilde{W} - \mu \trans{\mu}$
and $\widetilde{\mu} \dfn \E[\widetilde{w}_t]$. Similar to before we just write 
$\op{E}(\cdot) = \op{E}(W; \cdot)$ when the specific value of $W$ is unimportant.

When estimating the second moment dynamics, estimating $\widetilde{W}$ is similar to before.
However we now also estimate the mean of ${\mu}$. We provide details on estimation in \cref{sec:mean-id}. 
Then we design a DR controller in \cref{sec:str-synthesis}. 

\subsection{Mean identification} \label{sec:mean-id}
The measurement model \eqref{eq:measurement-model}, under \cref{asm:model-structured} is:
\begin{equation} \label{eq:measurement-model-structured}
    x_{i+1} = \tucker{\ten{M}; z_i, (1, \widetilde{w}_i)}, \quad \forall i \in \N_{1:N}.
\end{equation}
By expanding the Tucker product, \eqref{eq:measurement-model-structured} is equivalent to
$y_{i+1} \dfn x_{i+1} - [A_1, B_1] z_i = \tucker{\widetilde{\ten{M}}; I_{n_x}, z_i, w_i}$.
This is identical to the original measurement model \eqref{eq:measurement-model}, but with a perturbed 
value of $x_{i+1}$ and truncated tensor $\widetilde{\ten{M}}$. 
Nonetheless, we can still identify $\widetilde{W}$ using \eqref{eq:ls-estimator}.

Moreover, identification of the mean is also possible:
\begin{equation*}
    y_{i+1} = (\widetilde{\ten{M}} \btimes{2} z_i) ({\mu}_\star + \epsilon_i),
\end{equation*}
with $\epsilon_i = \widetilde{w}_i - {\mu}_\star$ a zero-mean error term
and with ${\mu}_\star$ the true value of the mean ${\mu}$. 

As such, introducing the stacked operators:
\begin{align*}
    Y_N^\mu &= (y_{2}, \dots, y_{N+1}), \, \opm{Z}_N^\mu &= [\widetilde{\ten{M}} \btimes{2} z_1; \dots; \widetilde{\ten{M}} \btimes{2} z_N] 
\end{align*}
The LS estimator then becomes:
\begin{equation} \label{eq:mean-estimator}
    \hat{{\mu}} = \pinv{(\opm{Z}_N^\mu)} Y_N^\mu. 
\end{equation}

\paragraph*{Data-driven bound}
Compared to identification of $\widetilde{W}$, developing a data-driven bound for the mean is easier
as the estimator is unbiased.
We have the following error model:
\begin{equation} \label{eq:error-model-mean}
    {\mean}_\star - \hat{{\mu}} = \ssum_{i=1}^{\nsample} G_i \epsilon_i,
\end{equation}
where we introduce $G_i \dfn \pinv{(\ssum_{j=1}^{\nsample} \op{M}(z_j \trans{z_j}))} \op{M}(z_i \trans{z_i})$
and $\op{M}(Z) \dfn \widetilde{\ten{M}}_{(3)} (Z \kron I_{n_x}) \trans{\widetilde{\ten{M}}_{(3)}}$. 

So, under \cref{asm:data}, the error is a sum of norm-bounded zero-mean i.i.d. random vectors. 
Hence \cref{lem:vechfd} is applicable:
\begin{lemma} \label{lem:mean-dd}
    Following \cref{asm:data}--\ref{asm:model-structured} with $N \geq n_z$. 
    Let $\meanh = \pinv{(\opm{Z}_N^\mu)} Y_N^\mu$ and $\zeta^2_{\mean} \dfn \ssum_{i=1}^N \nrm{G_i}_2^2$. Then,
    \begin{equation*}
        \prob[\nrm{{\mean}_\star - \hat{{\mu}}}_2 \leq \beta_{\mean}] \geq 1-\delta,
    \end{equation*}
    with $\beta_{\mean} = r_w \zeta_{\mean}  (2 + \sqrt{2 \ln(1/\delta)})$.
\end{lemma}
\begin{proof}
    A more formal proof is provided in Appendix~\ref{app:identification}. 
\end{proof}

\paragraph*{Moment ambiguity} 
We will synthesize a controller that is robust for all $\op{E}$ in the following ambiguity set
\begin{definition} \label{def:structured-ambiguity}
    Consider the structured ambiguity set:
    \begin{equation} \label{eq:structured-ambiguity}
        \mmathcal{W} \dfn \left\{ W \in \sym{n_w} \colon
            \begin{aligned}
                &W = \blkdiag(0, {\Sigma}) + (1, {\mu}) \trans{(1, {\mu})}, \\
                &\nrm{\hat{{\mu}} - {\mu}}_2 \leq \beta_\mean, \, \nrm{\hat{{\Sigma}} - {\Sigma}}_2 \leq \beta_\Sigma
            \end{aligned} \right\},
    \end{equation}
    with $\beta_{\mean} = r_w \zeta_{\mean}  (2 + \sqrt{2 \ln(1/\delta_\mean)})$, $\beta_\Sigma = \beta_W + \beta_\mean (\beta_\mean + 2 \nrm{\meanh}_2)$ 
    and $\beta_W =  r_w^2 \zeta_W \sqrt{2 \log(2n_w/\delta_W)}$ as in \cref{thm:error-bound-full} for probability $\delta_W$. Moreover $\meanh$ is as in \cref{eq:mean-estimator}
    and $\hat{\Sigma} \dfn \hat{\widetilde{W}} - \hat{\mean} \trans{\hat{\mean}}$ for $\hat{\widetilde{W}}$ as in \cref{eq:ls-estimator},
    but with the truncated model tensor and $y_{i+1}$ instead of $x_{i+1}$.     
\end{definition}

The ambiguity set is valid in the following sense:
\begin{theorem} \label{thm:structured-ambigutiy}
    Let $\op{E}_\star$ denote the true moment dynamics
    and let \cref{asm:data}--\ref{asm:model-structured} hold. Then with probability at least $1 - \delta_{\mu} - \delta_W$:
    \begin{equation*}
        \exists W \in \mmathcal{W} \colon \op{E}_{\star}(Z) = \op{E}(W; Z), \forall Z \in \sym{n_z}
    \end{equation*}
    with $\mmathcal{W}$ as in \cref{def:structured-ambiguity}. 
\end{theorem}
\begin{proof}
    The full proof deferred to Appendix~\ref{app:identification}. It combines 
    \cref{thm:error-bound-full} and \cref{lem:mean-dd} with some triangle inequalities. 
\end{proof}

\subsection{Control synthesis} \label{sec:str-synthesis}

Using the learned ambiguity set we can then proceed with control synthesis. 
Unlike in \cref{thm:drlqrcp} we cannot find the exact solution. Instead 
we solve a relaxed SDP by leveraging a result from robust optimization \cite{Ben-Tal2000}. 

In the non-structured setting, we used \cref{thm:ric-sdp-primal} to synthesize 
a controller by solving a SDP. There is however no guarantee that a feasible solution to that SDP
produces a stabilizing controller. This is only guaranteed for the optimum. 

So instead we will consider:
\begin{equation} \label{eq:sdp-stab}
    \begin{alignedat}{2}
        &\minimize_{(P, K)} &\qquad& \tr[P X_0] \\ 
        &\stt && P \sgeq 0, \, P - \adj{\op{E}_{K}}(P) \sgeq Q + \trans{K} R K.
    \end{alignedat}
\end{equation}
The second constraint is a Lyapunov inequality for $\op{E}_{K}(P)$,
which is defined as $\op{E}(\Pi_K(\cdot))$ as in \cref{eq:clops}. Hence any feasible $K$ is stabilizing by \cref{lem:stabcp}. 

\begin{theorem} \label{thm:sdp-stab}
    The problem \eqref{eq:sdp-stab} satisfies:
    \begin{enumerate}
        \item \label{thm:sdp-stab:a} $\rho(\op{E}_{K}) < 1$ for any feasible $K$;
        \item \label{thm:sdp-stab:b} $(P_\star, K_\star)$ is optimal and uniquely so if $X_0 \sgt 0$.
    \end{enumerate}
    Here $P_\star$ and $K_\star$ those of \cref{thm:cplqr}.
\end{theorem}
\begin{proof}
    Deferred to Appendix~\ref{app:lqr}. 
\end{proof}

We can now robustify \cref{eq:sdp-stab}. Specifically, we introduce the robust constraint 
\begin{equation} \label{eq:rob-lmi-stab}
    P - \adj{\op{E}_{K}}(W; P) \sgeq Q + \trans{K} R K, \quad \forall W \in \mmathcal{W}
\end{equation}
with $\mmathcal{W}$ as in \cref{def:structured-ambiguity}
and $\adj{\op{E}}_K(W; P) = \adj{\Pi_K}(\adj{\op{E}}(W; P))$ (similar to \cref{lem:opadj}). The non-linear SDP \eqref{eq:sdp-stab}
with this robust constraint is relaxed as stated below. 

\begin{theorem}
    Following the notation of \cref{def:structured-ambiguity} and letting $\{(A_i, B_i)\}_{i=1}^{n_w}$ 
    denote the modes in \eqref{eq:dyn-structured}, we introduce $\Theta \in \sym{n_x}$, 
    $\Gamma \in \Re^{n_u \times n_x}$, $\bar{\Sigma} \dfn \hat{\Sigma} + \beta_\Sigma I$ and
    $F_i = A_i \Theta + B_i \Gamma$ for $i \in \N_{1:n_w}$. Moreover, let $\opm{F} \dfn [F_2; \dots; F_{n_w}]$ 
    and $\hat{F} = F_1 + \ssum_{i=2}^{n_w} [\hat{\mean}]_i F_i$. Let $X_0 = C_0 \trans{C_0}$ denote the Cholesky 
    factorization of $X_0 \sgeq 0$. 

    Consider the SDP:
    \begin{equation*}
        \begin{alignedat}{2}
            &\minimize_{\Theta, \Gamma, \Phi, \Psi, \Lambda}&\qquad& \tr[\Phi] \\
            &\stt&& \Theta \sgt 0, \Phi \sgt 0, \Psi \sgt 0, \Lambda \sgt 0, \\
            &&& \begin{bmatrix}
                \Phi & \trans{U_0} \\ U_0 & \Theta
            \end{bmatrix} \sgeq 0, \, \begin{bmatrix}
                \Lambda & \beta_\mu \trans{\opm{F}} \\ \beta_\mu \opm{F} & I \kron \Psi
            \end{bmatrix} \sgeq 0, \\
            &&& T \sgeq \blkdiag(\Lambda, 0, \Psi, 0, 0),
        \end{alignedat}
    \end{equation*}
    where we use the linear map:
    \begin{equation*}
        T \dfn \begin{bmatrix}
            \Theta & \trans{\opm{F}} & \trans{\hat{F}} & \Theta & \trans{\Gamma} \\ 
            \opm{F} & \bar{\Sigma}^{-1} \kron \Theta \\
            \hat{F} && \Theta \\
            \Theta &&& Q^{-1} \\ 
            \Gamma &&&& R^{-1}
        \end{bmatrix}.
    \end{equation*}

    Assume \cref{asm:data}--\ref{asm:model-structured}, let $\op{E}_{\star}$ denote the 
    true second moment dynamics as in \cref{eq:dynsm-structured} and $\Pi_K$ the closed-loop policy as in \cref{eq:clops}.
    Then, given a feasible $(\Theta, \Gamma)$, the pair $(P, K)$ with $P = \Theta^{-1}$ 
    and $K = \Gamma \Theta^{-1}$ satisfies
    \begingroup
    \begin{enumerate}
        \item \label{thm:sdp-stab-reform:a} $\rho(\op{E}_{\star}(\Pi_K(\cdot))) < 1$; and
        \item \label{thm:sdp-stab-reform:b} $\tr[\Phi] \geq \tr[P X_0] \geq \mathrm{val}\eqref{eq:lqrcp} = \mathrm{val}\eqref{eq:slqr}$\todo{first ineq. is eq.},
    \end{enumerate}
    \endgroup
    with probability at least $1 - \delta_W - \delta_\mu$. 
\end{theorem}
\begin{proof}
    The proof is similar to that of \cite[Thm.~1]{Coppens2019}.  
    Note that $\adj{\op{E}}_K(P) = \adj{\Pi_K}(\adj{\op{E}}(P))$ by \cref{lem:opadj}. 
    So \eqref{eq:rob-lmi-stab} becomes
    \begin{equation} \label{eq:rob-lmi-0}
        P - \adj{\Pi_K}(\adj{\op{E}}(W; P)) \sgeq \adj{\Pi_K}(H),
    \end{equation}
    with $H = \blkdiag(Q, R)$. We then plug in the specific form of $W$ as described in \cref{def:structured-ambiguity}.
    That is $W = \blkdiag(0, \Sigma) + (1, \mu) \trans{(1, \mu)}$. Thus the left-hand side of \eqref{eq:rob-lmi-0} becomes
    \begin{equation} \label{eq:rob-lmi-1}
        P - \adj{\Pi_K}(\adj{\op{E}}(\blkdiag(0, \Sigma; P))) - \adj{\Pi_K}(\adj{\op{E}}(\mu_0 \trans{\mu_0}; P)),
    \end{equation}
    where we introduced $\mu_0 = (1, \mu)$. By \cref{eq:dynsm}, for all $Z$,
    \begin{equation*}
        \op{E}\left( \blkdiag(0, \Sigma); Z \right) = \ten{M}_{(3)} (\blkdiag(0, \Sigma), Z) \trans{\ten{M}}_{(3)}.
    \end{equation*}
    We have $\ten{M}_{(3)} = [[\ten{M}]_{::1}, \widetilde{\ten{M}}_{(3)}]$. Thus
    \begin{equation} \label{eq:struct-smop}
        \op{E}(\blkdiag(0, \Sigma); Z) = \widetilde{\op{E}}(\Sigma; Z) \dfn \widetilde{\ten{M}}_{(3)} (\Sigma \kron Z) \trans{\widetilde{\ten{M}}}_{(3)}.
    \end{equation}
    This enables us to further simplify \cref{eq:rob-lmi-1} to show that
    \begin{equation*}
        P - \adj{\Pi_K}(\adj{\widetilde{\op{E}}}(\Sigma; P)) - \adj{\Pi_K}(\adj{\op{E}}(\mu_0 \trans{\mu_0}; P)) \sgeq \adj{\Pi_K}(H)      
    \end{equation*}
    holding for all $\Sigma$ s.t. $\nrm{\Sigma - \hat{\Sigma}}_2 \leq \beta_\Sigma$ and for all $\mu \in (\{\hat{\mu}\} + r_{\mu} \ball{n_w})$,
    where $\ball{n_w}$ denotes the unit ball implies \eqref{eq:rob-lmi-stab}. 

    Clearly, $\adj{\widetilde{\op{E}}}(\bar{\Sigma}; P) \sgeq \adj{\widetilde{\op{E}}}(\Sigma; P)$ (cf. \cref{cor:cpmon})
    for all $\Sigma$ s.t. $\nrm{\Sigma - \hat{\Sigma}}_2 \leq \beta_\Sigma$. Hence \eqref{eq:rob-lmi-stab} holds iff:
    \begin{equation} \label{eq:rob-lmi-stab:b}
        P - \adj{\Pi_K}(\adj{\op{E}}(\mu_0 \trans{\mu_0}; P))  - \adj{\Pi_K}(\adj{\widetilde{\op{E}}}(\bar{\Sigma}; P)) \sgeq \adj{\Pi}(H),
    \end{equation}
    for all $\mu \in (\{\hat{\mu}\} + r_{\mu} \ball{n_w})$. By \cref{lem:opadj},
    \begin{equation*}
        \adj{\Pi_K}(\adj{\op{E}}(W; P)) = [I, \trans{K}] \ten{M}_{(2)} (W \kron P) \trans{\ten{M}_{(2)}} [I; K] .
    \end{equation*}
    We can expand the definition of the $2$-mode matricization \cref{eq:1modeflat}:
    \begin{equation*}
        \ten{M}_{(2)} = \begin{bmatrix}
            \trans{A_1} & \dots & \trans{A_{n_w}} \\ 
            \trans{B_1} & \dots & \trans{B_{n_w}}
        \end{bmatrix}.
    \end{equation*} 
    Let $\Theta = P^{-1}$ and $\Gamma = K \Theta$. Note that $\mu_0 \trans{\mu_0} \kron P = (\mu_0 \kron I) P (\trans{\mu_0} \kron I)$ \cite[Thm.~E.1.3]{DeKlerk2002}.
    Then, using these facts and the definition of $\opm{F}$ and $F_1$, after multiplying \cref{eq:rob-lmi-stab:b} on the left and right 
    side with $\Theta$ results in:
    \begin{align*}
        W   &- \trans{[F_1; \opm{F}]}(\mu_0 \kron I) P (\trans{\mu_0} \kron I) [F_1; \opm{F}] \\
            &-\trans{\opm{F}} (\bar{\Sigma} \otimes I) \opm{F} \sgeq \Theta Q \Theta + \trans{\Gamma} R \Gamma,
    \end{align*}
    for all $(\mu \in \{\hat{\mu}\} + r_\mu \ball{n_w})$. Applying Schur complements and introducing $\theta = (\hat{\mu} - \mu)/r_\mu$,
    allows us to rewrite the final display as an inequality of the type $\mathcal{U}[\beta_\mu] \sgeq 0$
    in \cref{lem:rob-sdp}, with $T$ as in the theorem above. Applying the Lemma gives the second and third LMI. 

    So by \cref{lem:rob-sdp}, a feasible pair $(\Theta, \Gamma)$ produces a 
    pair $(P, K)$ that is feasible for the robust constraint \eqref{eq:rob-lmi-stab}. Since \cref{eq:rob-lmi-stab} 
    holds for all $W \in \mmathcal{W}$ it holds for $\op{E}_{\star}$ provided that $\exists W \in \mmathcal{W}$
    such that $\op{E}_\star(Z) = \op{E}(W; Z)$ for all $Z$, 
    which is true with probability at least $1 - \delta_\mu - \delta_W$. Hence, by \cref{thm:sdp-stab:a} (i) follows.

    The first constraint in the final SDP 
    is equivalent to
    \begin{equation*}
        \Phi \sgeq \trans{U_0} \Theta^{-1} \trans{U_0},
    \end{equation*}
    by a Schur complement. So $\tr[\Phi] \geq \tr[\trans{U_0} \Theta^{-1} \trans{U_0}]= \tr[X_0 P]$. Combined 
    with the previous argument about feasibility and \cref{thm:sdp-stab:b}, we prove (ii). 
\end{proof}
\end{arxiv}

\section{Numerical experiments} \label{sec:numerical}
In this section we numerically investigate the methods developed in this paper. We begin with the simple case 
of \emph{repeated \update*{initialization}} (cf. \cref{rem:asm}), comparing both the case 
where the modes are known and where they are not, i.e. the model-free setting. We show how, in the first 
setting, our bounds are sufficiently tight to enable control synthesis. Next we use data generated using \emph{rollout}
to produce similar estimates and we highlight the differences and challenges associated with doing so. We also 
show that experimentally our method also works when using only a single trajectory, although theoretical guarantees
are not yet available in this setting. \ilarxiv{We also compared to the averaging rollout approach in \cite{Xing2020,Xing2021}.
Finally, we show how structural information can be exploited to get tighter estimates of the uncertainty.}\ilpub{\update*{
In the technical report \cite{Arxiv} we compare with the averaging rollout approach in \cite{Xing2020,Xing2021} and show 
how structural information can be exploited to get tighter estimates of the uncertainty.}}

\subsection{Accuracy of the estimate and bounds} \label{sec:toy}
We consider the case $n_x = 2$ and $n_u = 1$ with modes:
\begin{equation*}
    A_1 = \begin{bmatrix}
        1 & 0 \\ 0 & 0
    \end{bmatrix}, \, A_2 = \begin{bmatrix}
        0 & 1 \\ 0 & 0
    \end{bmatrix}, \, A_3 = \begin{bmatrix}
        0 & 0 \\ 0 & 1
    \end{bmatrix}, \, B_3 = \begin{bmatrix}
        0 \\ 1
    \end{bmatrix},
\end{equation*}
and $B_1 = B_2 = 0$. The disturbance is sampled uniformly from the ball $\{w \in \Re^3 \colon \nrm{w - \mu}_2 \leq 0.25\}$ 
with $\mu = (0, 0, 0.1)$. 

\paragraph*{Repeated initialization}
To generate measurements we sample a $z_0 = (x_0, u_0)$ uniformly from an Euclidean ball of radius one and then propagate
the dynamics by one step to get a measurement $x_1$, i.e. using \emph{repeated \update*{initialization}} as in \cref{rem:asm}. 
From this data, an estimate $\hat{W}$ and associated radius 
$\beta_W$ are determined as prescribed in \cref{thm:error-bound-full}. 

To empirically quantify the accuracy of the scheme, the procedure above is repeated $M = 100$\todo{correct?} times for 
each sample count $N$. For each estimate we compute $\nrm{\hat{W} - W_\star}_2$ and $\beta_W$. A confidence 
plot is provided in \cref{fig:toy}. The middle plot depicts the result when $\ten{M}$
is selected, based on true mode info (i.e. $[\ten{M}]_{::i} = [A_i, B_i]$ for $i=1,2,3$); and the right-most plot 
uses the model-free case (i.e. $\ten{M}_{(3)} = I$). 

We determined $W_\star$ in the first setting by noting that $\E[w\trans{w}] = r^2_w I_{n_w} / (n_w + 2)$ 
when $w$ is sampled uniformly from $\{w \in \Re^{n_w} \colon \nrm{w}_2 \leq r_w\}$ ---
which can be verified by symmetry and solving a simple integral --- and applying the appropriate
transformation\footnote{To shift the ball we shift $\widetilde{w} = w + \mu$. Hence $\E[\widetilde{w} \trans{\widetilde{w}}] = W_\star + \mu \trans{\mu}$.}. Similarly, $W_\star$ for the model-free case is then recovered by applying the transformation in 
\eqref{eq:true_moment}. 

The left-most plot depicts the error when $\hat{W}$ is estimated directly from measurements of $w$ 
and the bound is as in \cref{lem:dd-moment}. It is clear from the figure that the empirical error does not 
increase much when we do not directly observe the disturbance --- at least for this simple model. Instead 
the main loss is in the bound, which is about an order of magnitude looser than the direct sample case. This is to be 
expected however and, as we confirm later using DR synthesis, the bounds are still practical in low-dimensional settings.

\begin{figure}
    \includegraphics{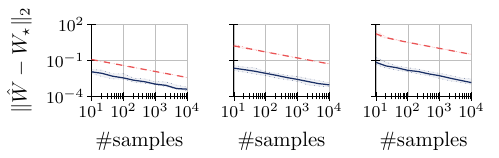}
    \caption{Error in moment estimation using repeated initialization: (left) when directly observing the disturbance; (mid) 
    when using the true modes in $\ten{M}$; (right) when using no mode information, i.e. model-free. 
    The dashed line depicts the radius predicted by \cref{thm:error-bound-full}. Each colored area is a $0.1$ confidence 
    interval surrounding the median.} \label{fig:toy}
\end{figure}

The middle and right-most plot depict the use of prior information and the model-free case respectively.
The fact that $W_\star$ is not the same for both cases implies that the absolute errors 
are dissimilar. For a fair comparison, we additionally use the estimated second moment dynamics 
$\op{E}$ as in \cref{eq:dynsm_}. Specifically we plot $\nrm{\hat{\opm{E}} - \opm{E}_\star}_2/\nrm{\opm{E}_\star}_2$ 
where $\hat{\opm{E}}$ is the matrix associated with the estimated dynamics and $\opm{E}_\star$ its true value. 
These are computed using \cref{cor:cpadj}. Note that $\opm{E}_{\star}$
is the same independent of the selected $\ten{M}$. \cref{fig:toy-operator} shows the result.

\begin{figure}
    \includegraphics{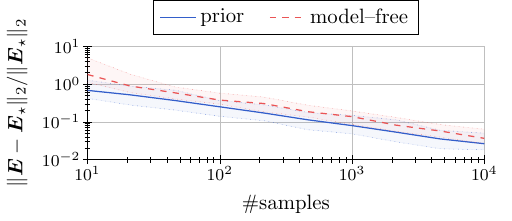}
    \caption{Estimation accuracy of matrix of second order dynamics and comparison between model-free and true mode based 
    estimation. Each colored area is a $0.1$ confidence interval surrounding the median.} \label{fig:toy-operator}
\end{figure}

It is clear that, for the low-dimensional example considered here, exploiting prior mode information has no significant 
advantage when it comes to estimation accuracy. Instead the main advantage is a smaller ambiguity 
set. When using the model-free approach, the control synthesis problem 
remains infeasible for any tested sample count. That is, the ambiguity set exceeds the minimum size that can be stabilized 
by a single controller. When exploiting the prior information encapsulated in the modes, this is not the case. 

\paragraph*{Control synthesis}
We can quantify the size of the ambiguity set through controller synthesis as in \cref{thm:drlqrcp}. 
The performance of a gain $\hat{K}$ is measured through the infinite horizon cost 
achieved on the true system. This cost equals the optimum of the following SDP, by \cref{prop:lyapcp} and \cref{lem:opadj},
\begin{equation*}
    \min_{P} \, \{\tr[P X_0] \colon P - \adj{\Pi_{\hat{K}}}(\adj{\op{E}_{\star}}(P)) \sgeq \blkdiag(Q, R), \, P \sgeq 0 \}.
\end{equation*}
We select $X_0 = I$ here and in \cref{thm:ric-sdp-primal} for control synthesis. Also let $Q = I$ and $R = 10$.
The relative error with the true optimum of \cref{eq:slqr} is then a metric 
for the accuracy of the ambiguity set. The result is depicted in \cref{fig:toy-cost}. Note that the 
rate is as predicted in \cref{thm:sample-complexity}.

\begin{figure}
    \includegraphics{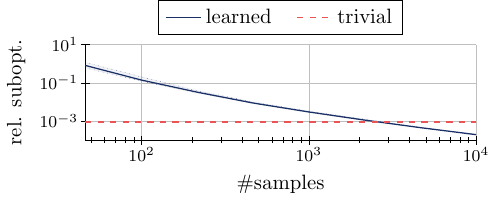}
    \caption{Evaluation of suboptimality of DR control synthesis. The colored area is a $0.1$ confidence interval surrounding the median.} \label{fig:toy-cost}
\end{figure}

The dashed line in \cref{fig:toy-cost} depicts the cost for the \emph{trivial ambiguity}, 
which only uses $\nrm{w}_2 \leq r_w$ implying $\nrm{W_\star}_2 \leq r_w^2$. 
This constraint is also used to synthesize a controller with \cref{thm:drlqrcp}. 
We say that the learned ambiguity set is \emph{informative}
when the performance of the associated controller improves upon the trivial one. 

\paragraph*{Rollout}
Repeated initialization is not realistic in practice, since it assumes we can directly select all states of 
the system. Rollout instead initializes the system at some --- easily realizable --- initial state and then 
applies a sequence of control actions to excite the system. We use the state $x_0 = (1, 1)$
and use the control law $u_t = [-0.5, -0.2] x_t + \delta_t$ with $\delta_t$ sampled uniformly from 
a Euclidian ball of radius $35$ at each time step to integrate the dynamics for $T = 25$ time steps. 
Then $z_i = (x_{T-1}, u_{T-1})$ and $x_{i+1} = x_T$ are used as data points (cf. \cref{rem:asm}) for $i=1, \dots, N$. 
Here $N$ is the amount of rollouts, which we also refer to as the sample count. 

Similarly to before we can evaluate the accuracy of the estimate empirically, by resampling data sets $M = 100$ times 
and computing the errors. The result is depicted in \cref{fig:toy-rollout}. Comparing with \cref{fig:toy} we see,
as expected, that estimation is more challenging when using rollout data. This is likely explained by the distribution of $z_i$
being less suitable for identification as discussed in \cref{rem:mixed}

\begin{figure}
    \includegraphics{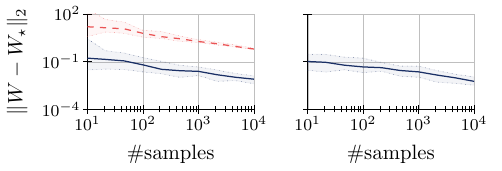}
    \caption{Empirical result for toy problem using rollout data: (left) using the tail of $\#$nsamples rollouts of length $25$,
    with predicted radius as a dashed line; (right) using one rollout of length $\#$nsamples.
    Each colored area is a $0.1$ confidence interval surrounding the median.} \label{fig:toy-rollout}
\end{figure}

\paragraph*{Single trajectory}
The right plot in \cref{fig:toy-rollout} also depicts the empirical error when we simply use one rollout of length $T = N$.
So the data is then $(x_i, u_i)$ and $x_{i+1}$ for $i\in\N_{0:N-1}$. It is clear that the estimate becomes more accurate 
when the trajectory length is increased, showing --- at least empirically --- that our approach can also work for single trajectory
identification. This was also observed for a similar setup in \cite{Di2021}, yet a convergence proof 
and associated conditions on the exciting inputs are still unavailable (cf. \cref{rem:asm}).

\ilarxiv{\subsection{Averaged rollouts approach}}

\ilarxiv{\begin{figure}
    \includegraphics{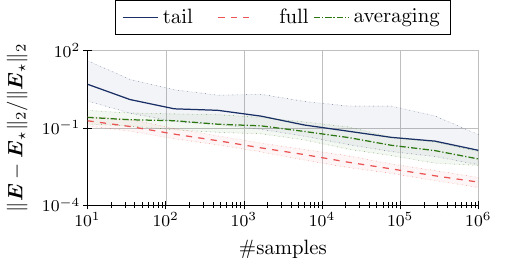}
    {\caption{Comparison of three system identification approaches:
    (tail) using only the tail of each rollout; (full) using the full rollout; (averaging) when trajectories are averaged over rollouts. 
    Each colored area is a $0.1$ confidence interval surrounding the median.} \label{fig:tsummers}}
\end{figure}}

\begin{arxiv}
In this part, we compare our approach with that of \cite{Xing2020,Xing2021}. We will refer to this approach 
as the \emph{averaging rollout approach} (contrasting our \emph{least-squares approach}). In essence, this approach
identifies the second moment dynamics \eqref{eq:dynsm_} directly by using a sample average to construct a single 
moment trajectory $(Z_t, X_t)$. The second moment dynamics are then estimated using standard identification techniques 
from LTI identification. 

We consider a dynamical system similar\footnote{The dynamics in \cite{Xing2021} are defined differently. Constructing an exact equivalent that also satisfies \cref{asm:model} 
is challenging. So instead the simpler system we consider is only approximately equivalent.} to the one considered in \cite[\S{}IV.A]{Xing2021}. Specifically
\begin{align*}
    A_1 &= \begin{bmatrix}
       -0.2 & 0.3 \\ -0.4 & 0.8 
    \end{bmatrix}, A_2 = \begin{bmatrix}
        0.2 & -0.15 \\ -1 & 0
    \end{bmatrix},  A_3 = \begin{bmatrix}
        1 & 0.1 \\ 0.2 & 0
    \end{bmatrix}, \\ A_4 &= \begin{bmatrix}
        -0.05 & 1 \\ -0.15 & 0
    \end{bmatrix},  A_5 = \begin{bmatrix}
        0 & 0 \\ 0 & 1
    \end{bmatrix}, \\ \trans{B_1} &= {\begin{bmatrix}
        -1.8 \, -0.8
    \end{bmatrix}}, \, \trans{B_6} = {\begin{bmatrix}
        -1 \,  -0.15
    \end{bmatrix}}, \, \trans{B_7} = {\begin{bmatrix}
        0.15 \, -1
    \end{bmatrix}}.
\end{align*}
The matrices $A_i$ and $B_i$ left unspecified are assumed zero. We then sample the disturbance uniformly
from a degenerate ellipse such that $W_\star = \diag(1, 0.15, 0.08, 0.02, 0.08, 0.05, 0.2)$
and $[w_t]_1 = 1$ for all $t \in \N$. 

We then generate rollouts as in \cite{Xing2021} and again evaluate the estimation accuracy in terms of $\opm{E}$.
\cref{fig:tsummers} shows the result. The rollout length was selected as $T = 12$ and the horizontal axis tracks
the amount of rollouts. In the averaging approach, the inputs should not only be random, but their distributions as well.
This explains the high variance in the estimates.

The worst error is achieved when we use only the tail of each rollout in our least-squares scheme, which makes
sure our convergence guarantees still hold. The averaging approach performs marginally better, and the 
asymptotic rates are the same (as confirmed by the theory in \cite{Xing2021}).
One can expect that accuracy improves further when the rollout length 
is increased\footnote{The analysis performed in \cite{Xing2021} however confirms an opposite result. 
This is likely due to conservativeness in their analysis.}. Note however that the approach of \cite{Xing2021} does not come with 
a tight error analysis. If theoretical guarantees are not critical, we can similarly to the single trajectory case considered above,
also add other states in the rollout and use the full data set. The resulting estimate clearly
outperforms the averaging result. Intuitively this is explained by the latter losing information by averaging.

It is also important to note that the scheme, as introduced in \cite{Xing2021} does not support correlation between 
$A(w)$ and $B(w)$ (as in our first example). Extending the averaging scheme to support this is relatively trivial. 
\end{arxiv}

\ilarxiv{\subsection{Exploiting prior structure}}
\begin{arxiv}
We consider the structured dynamics \eqref{eq:dyn-structured} with
\begin{align*}
    A_1 &= \begin{bmatrix}
        1 & 0.02 \\ 0 & 0.992
    \end{bmatrix}, \, B_1 = \begin{bmatrix}
        0 \\ 0.02
    \end{bmatrix} \\
    A_2 & = \begin{bmatrix}
        0 & 0 \\ 0 & -0.03
    \end{bmatrix}, \, A_3 = \begin{bmatrix}
        0 & -0.03 \\ 0 & 0
    \end{bmatrix}, \, B_4 = \begin{bmatrix}
        0 \\ 0.01
    \end{bmatrix}.
\end{align*}
The matrices $A_i$ and $B_i$ left unspecified are assumed zero. The disturbance is distributed 
uniformly on $\{w \in \Re^3 \colon \nrm{w}_2 \leq 0.05\}$ and data is generated using \emph{repeated initialization},
analogously to the procedure used in \cref{sec:toy}. We then estimate $\hat{W}$ and $\hat{\mu}$ using the procedure in 
\cref{sec:structural}. The empirical error and the predicted radii are then predicted in \cref{fig:struct-est}. 
\end{arxiv}

\begin{arxiv}
\begin{figure}
    \includegraphics{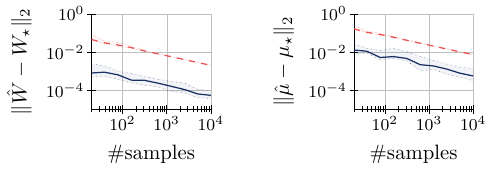}
    {\caption{Error in moment estimation using repeated initialization and prior structure: (left) the second moment estimation error; (right) the mean estimation error. 
    The dashed lines depict the predicted radius. Each colored area is a $0.1$ confidence interval surrounding the median.} \label{fig:struct-est}}
\end{figure}
\end{arxiv}

\begin{arxiv}
Finally, we also evaluate the size of the ambiguity set by using the control synthesis procedure of \cref{sec:str-synthesis}. 
The LQR problem is tuned identically to before with $X_0 = I$, $Q = I$ and $R = 10$. We again compare with the 
true optimum and the trivial controller recovered when using only $\nrm{\mu}_2 \leq r_w$ and $\nrm{W_\star}_2 \leq r_w^2 I$ 
(the transformation to an ambiguity set like \cref{def:structured-ambiguity} is analogous to the data-driven case). 
The resulting suboptimality plot is depicted in \cref{fig:struct-cost}. Note that, the horizontal axis starts at $N = 20$.
For fewer samples, the synthesis problem is infeasible. It is immediately clear how the structured prior information 
is exploited by our scheme by noting how little samples are required before the scheme improves upon the trivial controller.
Also note that, even though we only proved the $1/N$ rate for the non-structured case, the same decrease is observed 
here. The same observation was made in \cite{Coppens2019}. 
\end{arxiv}

\begin{arxiv}
\begin{figure}
    \includegraphics{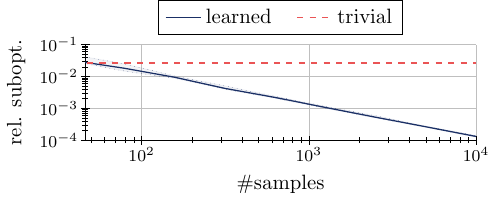}
    {\caption{Evaluation of suboptimality of DR control synthesis with structured prior information. 
    The colored area is a $0.1$ confidence interval surrounding the median.} \label{fig:struct-cost}}
\end{figure}
\end{arxiv}

\section{Conclusion} \label{sec:conclusion}
We developed a novel system identification scheme for linear dynamics with state- and input-multiplicative noise. 
The resulting estimators are shown to converge at a $1/\sqrt{N}$ rate, with $N$ the number of samples, when the 
data is generated either using rollout or through repeated initialization of the dynamics. We illustrated empirically 
that --- for simple dynamics --- the constants in our bounds are practical, i.e. they can be used for DR control synthesis. 
Moreover, the tightness of the bound is similar to that of a usual matrix Hoeffding bound. Also the DR control synthesis 
problem was shown to converge at a $1/N$ rate to the true optimum. 

For more complex dynamics it is likely that the bounds are not sufficiently tight to enable DR control synthesis. 
In that case either bootstrapping can be exploited as in \cite{Dean2019,Delage2010} or the DR scheme can be used 
to identify the robustness of the certainty equivalent controller as in \cite{Gravell2020}.

\begin{arxiv}
We additionally illustrated how knowledge about deterministic terms can be exploited to accelerate the identification
process and tighten the associated bounds (referred to as \emph{structured information}). An additional DR control 
synthesis scheme was also presented, which exploits the tighter bounds. The scheme was empirically shown to also 
converge to the true optimum at a $1/N$ rate.

Moreover experiments illustrated that the identification scheme functions well when data is gathered from a 
single trajectory and compared the results with the averaging rollout identification scheme of \cite{Xing2020,Xing2021}. 
\end{arxiv}

From these tests it is clear that correlation between errors does not accumulate when using single trajectory 
data. A formal proof of this fact and a sample complexity analysis however proved challenging and is considered an open
problem for future work. 

\printbibliography%

@article{Arxiv,
    title={{Provably stable learning control of linear dynamics with multiplicative noise}},
    author={Coppens, Peter and Patrinos, Panagiotis},
    year={2023},
    eprint={2207.06062},
    archivePrefix={arXiv},
    oprimaryClass={math.OC},
    ojournal = {\textnormal{arXiv}}
}

@article{Coppens2020-TR,
  title         = {{Sample complexity of data-driven stochastic LQR with multiplicative uncertainty}},
  author        = {Coppens, Peter and Patrinos, Panagiotis},
  year          = {2020},
  eprint        = {2005.12167},
  archiveprefix = {arXiv},
  oprimaryclass = {eess.SY},
  ojournal      = {\textnormal{arXiv}}
}

@book{Boyd2004,
  title={Convex optimization},
  author={Boyd, Stephen and Vandenberghe, Lieven},
  year={2004},
  publisher={Cambridge university press}
}

@inproceedings{Mania2019,
  author    = {Mania, Horia and Tu, Stephen and Recht, Benjamin},
  booktitle = {Advances in Neural Information Processing Systems},
  obooktitle = {NeurIPS},
  oeditor    = {H. Wallach and H. Larochelle and A. Beygelzimer and F. d\textquotesingle Alch\'{e}-Buc and E. Fox and R. Garnett},
  pages     = {},
  publisher = {Curran Associates, Inc.},
  title     = {Certainty Equivalence is Efficient for Linear Quadratic Control},
  url       = {https://proceedings.neurips.cc/paper/2019/file/5dbc8390f17e019d300d5a162c3ce3bc-Paper.pdf},
  volume    = {32},
  year      = {2019}
}

@article{Dean2019,
  oarchiveprefix   = {arXiv},
  oarxivid         = {1710.01688},
  author          = {Dean, Sarah and Mania, Horia and Matni, Nikolai and Recht, Benjamin and Tu, Stephen},
  doi             = {10.1007/s10208-019-09426-y},
  oeprint          = {1710.01688},
  issn            = {1615-3375},
  journal         = {Foundations of Computational Mathematics},
  ojournal         = {FoCM},
  omonth           = {aug},
  opages           = {1--43},
  title           = {{On the Sample Complexity of the Linear Quadratic Regulator}},
  url             = {http://link.springer.com/10.1007/s10208-019-09426-y},
  year            = {2019}
}

@article{Delage2010,
author = {Delage, Erick and Ye, Yinyu},
doi = {10.1287/opre.1090.0741},
issn = {0030-364X},
journal = {Operations Research},
omonth = {jun},
number = {3},
opages = {595--612},
title = {{Distributionally Robust Optimization Under Moment Uncertainty with Application to Data-Driven Problems}},
url = {http://pubsonline.informs.org/doi/abs/10.1287/opre.1090.0741},
volume = {58},
year = {2010}
}

@article{Rahimian2019,
    author = {Hamed Rahimian and Sanjay Mehrotra},
    title = {Frameworks and {Results} in {Distributionally} {Robust} {Optimization}},
    journal = {Open Journal of Mathematical Optimization},
    eid = {4},
    publisher = {Universit\'e de Montpellier},
    volume = {3},
    year = {2022},
    doi = {10.5802/ojmo.15},
    url = {https://ojmo.centre-mersenne.org/articles/10.5802/ojmo.15/}
}

@InProceedings{Coppens2019, 
  title={Data-driven distributionally robust {LQR} with multiplicative noise}, 
  author={Coppens, Peter and Schuurmans, Mathijs and Patrinos, Panagiotis},
  booktitle={L4DC}, 
  opages={521--530}, 
  year={2020}, 
  oeditor={Alexandre M. Bayen and Ali Jadbabaie and George Pappas and Pablo A. Parrilo and Benjamin Recht and Claire Tomlin and Melanie Zeilinger}, 
  volume={120}, 
  oseries={Proceedings of Machine Learning Research}, 
  oaddress={The Cloud}, 
  eprint={1912.09990},
  archivePrefix={arXiv},
  publisher={PMLR}, 
  url={http://proceedings.mlr.press/v120/coppens20a.html} 
}

@article{Tropp2015,
url = {http://dx.doi.org/10.1561/2200000048},
year = {2015},
volume = {8},
journal = {Foundations and Trends® in Machine Learning},
title = {An Introduction to Matrix Concentration Inequalities},
doi = {10.1561/2200000048},
issn = {1935-8237},
number = {1-2},
pages = {1-230},
author = {Joel A. Tropp}
}

@inproceedings{Coppens2020,
  author={Coppens, Peter and Patrinos, Panagiotis},
  booktitle={2020 59th IEEE Conference on Decision and Control (CDC)}, 
  obooktitle={CDC},
  publisher = {IEEE},
  title={{Sample Complexity of Data-Driven Stochastic LQR with Multiplicative Uncertainty}},
  year={2020},
  volume={},
  number={},
  opages={6210-6215},
  doi={10.1109/CDC42340.2020.9303905}
}

@inproceedings{Xing2020,
  author={Y. {Xing} and B. {Gravell} and X. {He} and K. H. {Johansson} and T. {Summers}},
  booktitle={2020 American Control Conference (ACC)}, 
  obooktitle={ACC},
  title={Linear System Identification Under Multiplicative Noise from Multiple Trajectory Data}, 
  year={2020},
  volume={},
  number={},
  opages={5157-5261},
  doi={10.23919/ACC45564.2020.9147756}
}

@book{Horn1991,
author = {Horn, Roger A. and Johnson, Charles R.},
doi = {10.1017/CBO9780511840371},
isbn = {9780521305877},
omonth = {apr},
publisher = {Cambridge University Press},
title = {{Topics in Matrix Analysis}},
url = {https://www.cambridge.org/core/product/identifier/9780511840371/type/book},
year = {1991}
}

@book{Golub2013,
author = {Golub, Gene H. and Loan, Charles F. Van},
edition = {4th},
isbn = {978-0-8018-5414-9},
opages = {784},
publisher = {JHU Press},
title = {{Matrix Computations}},
year = {2013}
}

@book{DeKlerk2002,
oaddress = {Boston, MA},
author = {de Klerk, Etienne},
doi = {10.1007/b105286},
isbn = {978-1-4020-0547-3},
publisher = {Springer US},
oseries = {Applied Optimization},
title = {{Aspects of Semidefinite Programming}},
url = {http://link.springer.com/10.1007/b105286},
volume = {65},
year = {2002}
}

@article{Coppens2021,
  author  = {Coppens, Peter and Patrinos, Panagiotis},
  journal = {Control Syst. Lett.},
  title   = {Data-Driven Distributionally Robust {MPC} for Constrained Stochastic Systems},
  year    = {2022},
  volume  = {6},
  number  = {},
  publisher = {IEEE},
  opages   = {1274-1279},
  doi     = {10.1109/LCSYS.2021.3091628}
}

@article{Kolda2009,
author = {Kolda, Tamara G. and Bader, Brett W.},
doi = {10.1137/07070111X},
issn = {0036-1445},
journal = {SIAM Review},
omonth = {aug},
number = {3},
opages = {455--500},
title = {{Tensor Decompositions and Applications}},
url = {http://epubs.siam.org/doi/10.1137/07070111X},
volume = {51},
year = {2009}
}

@techreport{Kolda2006,
oaddress = {Albuquerque, NM, and Livermore, CA},
author = {Kolda, Tamara Gibson},
doi = {10.2172/923081},
institution = {Sandia National Laboratories (SNL)},
omonth = {apr},
number = {No. SAND2006-2081},
title = {{Multilinear operators for higher-order decompositions.}},
url = {http://www.osti.gov/servlets/purl/923081-u0xXJa/},
year = {2006}
}

@article{Morozan1983,
author = {Morozan, Toader},
odoi = {10.1080/07362998308809005},
oissn = {15329356},
journal = {Stochastic Analysis and Applications},
ojournal = {Stoch. Anal. Appl.},
omonth = {jan},
number = {1},
opages = {89--116},
title = {{Stabilization of some stochastic discrete-time control systems}},
volume = {1},
year = {1983}
}

@article{Gravell2020,
archivePrefix = {arXiv},
arxivId = {2004.08019},
author = {Gravell, Benjamin and Esfahani, Peyman Mohajerin and Summers, Tyler},
eprint = {2004.08019},
ojournal = {arXiv preprint arXiv:2004.08019},
opages = {1--14},
title = {{Robust Control Design for Linear Systems via Multiplicative Noise}},
url = {http://arxiv.org/abs/2004.08019},
year = {2020}
}

@inproceedings{Gravell2020a,
  title     = {Robust Learning-Based Control via Bootstrapped Multiplicative Noise},
  author    = {Gravell, Benjamin and Summers, Tyler},
  obooktitle={L4DC}, 
  booktitle = {Proceedings of the 2nd Conference on Learning for Dynamics and Control},
  opages     = {599--607},
  year      = {2020},
  editor    = {Bayen, Alexandre M. and Jadbabaie, Ali and Pappas, George and Parrilo, Pablo A. and Recht, Benjamin and Tomlin, Claire and Zeilinger, Melanie},
  volume    = {120},
  oseries    = {Proceedings of Machine Learning Research},
  omonth     = {10--11 Jun},
  publisher = {PMLR},
  url       = {https://proceedings.mlr.press/v120/gravell20a.html}
}

@article{Bernstein1987,
author = {Bernstein, D.},
doi = {10.1109/TAC.1987.1104517},
issn = {0018-9286},
journal = {IEEE Transactions on Automatic Control},
ojournal = {IEEE Trans. Autom. Control},
omonth = {dec},
number = {12},
opages = {1076--1084},
title = {{Robust static and dynamic output-feedback stabilization: Deterministic and stochastic perspectives}},
url = {http://ieeexplore.ieee.org/document/1104517/},
volume = {32},
year = {1987}
}

@article{Romero2013,
author = {Romero, Daniel and Leus, Geert},
doi = {10.1109/TSP.2013.2283473},
issn = {1053-587X},
journal = {IEEE Transactions on Signal Processing},
ojournal = {IEEE Trans. Signal Process.},
omonth = {dec},
number = {24},
opages = {6232--6246},
publisher = {IEEE},
title = {{Wideband Spectrum Sensing From Compressed Measurements Using Spectral Prior Information}},
url = {http://ieeexplore.ieee.org/document/6609101/},
volume = {61},
year = {2013}
}

@article{Lovasz1989,
author = {Lov{\'{a}}sz, L{\'{a}}szl{\'{o}}},
doi = {10.1007/BF02585470},
issn = {01003569},
journal = {Bol. Soc. Bras. Mat.},
number = {1},
opages = {87--99},
title = {{Singular spaces of matrices and their application in combinatorics}},
volume = {20},
year = {1989}
}

@article{Testa2018,
archivePrefix = {arXiv},
arxivId = {1701.06864v1},
author = {Testa, Damiano},
eprint = {1701.06864v1},
opages = {1--9},
title = {3 × 3 singular matrices of linear forms},
year = {2018}
}

@book{Ledoux1991,
oaddress = {Berlin, Heidelberg},
author = {Ledoux, Michel and Talagrand, Michel},
doi = {10.1007/978-3-642-20212-4},
isbn = {978-3-642-20211-7},
publisher = {Springer Berlin Heidelberg},
title = {{Probability in Banach Spaces}},
url = {http://link.springer.com/10.1007/978-3-642-20212-4},
year = {1991}
}

@book{Boucheron2013,
author = {Boucheron, St{\'{e}}phane and Lugosi, G{\'{a}}bor and Massart, Pascal},
booktitle = {Concentration Inequalities},
doi = {10.1093/acprof:oso/9780199535255.001.0001},
isbn = {9780199535255},
omonth = {feb},
publisher = {Oxford University Press},
title = {{Concentration Inequalities}},
url = {https://oxford.universitypressscholarship.com/view/10.1093/acprof:oso/9780199535255.001.0001/acprof-9780199535255},
year = {2013}
}

@book{Grimmett2001d,
author = {Grimmett, Geoffrey R and Stirzaker, David},
title = {{Probability and random processes}},
year = {2001}
}

@book{Horn2012,
author = {Horn, Roger A. and Johnson, Charles R.},
doi = {10.5555/2422911},
edition = {2},
isbn = {978-0-521-54823-6},
omonth = {dec},
publisher = {Cambridge University Press},
title = {{Matrix Analysis}},
url = {https://dl.acm.org/doi/book/10.5555/2422911},
year = {2012}
}

@article{Kubrusly1985,
author = {Kubrusly, C. and Costa, O.},
doi = {10.1109/TAC.1985.1103840},
issn = {0018-9286},
journal = {IEEE Transactions on Automatic Control},
ojournal = {IEEE Trans. Autom. Control},
omonth = {nov},
number = {11},
opages = {1082--1087},
title = {{Mean square stability conditions for discrete stochastic bilinear systems}},
url = {http://ieeexplore.ieee.org/document/1103840/},
volume = {30},
year = {1985}
}

@incollection{Ben-Tal2000,
author = {Ben-Tal, Aharon and {El Ghaoui}, Laurent and Nemirovski, Arkadi},
booktitle = {Handbook of Semidefinite Programming},
ochapter = {6},
doi = {10.1007/978-1-4615-4381-7_6},
opages = {139--162},
title = {{Robustness}},
url = {http://link.springer.com/10.1007/978-1-4615-4381-7_6},
year = {2000}
}

@article{Kukush2003,
author = {Kukush, A and Markovsky, I and Huffel, S. Van},
doi = {10.1007/s001840200217},
issn = {0026-1335},
journal = {Metrika},
omonth = {jul},
number = {3},
opages = {253--285},
title = {{Consistent estimation in the bilinear multivariate errors-in-variables model}},
url = {http://link.springer.com/10.1007/s001840200217},
volume = {57},
year = {2003}
}

@article{Eaton1973,
author = {Eaton, Morris L. and Perlman, Michael D.},
journal = {The Annals of Statistics},
number = {4},
opages = {710--717},
title = {{The Non-Singularity of Generalized Sample Covariance Matrices}},
url = {https://www.jstor.org/stable/2958314},
volume = {1},
year = {1973}
}

@article{Choi1975,
author = {Choi, Man-Duen},
isbn = {0024-3795},
issn = {00243795},
journal = {Linear Algebra and its Applications},
number = {3},
pages = {285--290},
title = {{Positive Linear Maps on Complex Matrices}},
url = {http://www.sciencedirect.com/science/article/pii/0024379575900750#},
volume = {10},
year = {1975}
}

@article{Pascoe2019,
archivePrefix = {arXiv},
arxivId = {1905.09895},
author = {Pascoe, J. E.},
eprint = {1905.09895},
opages = {1--22},
title = {{The outer spectral radius and dynamics of completely positive maps}},
url = {http://arxiv.org/abs/1905.09895},
year = {2019}
}

@book{Damm2004,
oaddress = {Berlin, Heidelberg},
author = {Damm, Tobias},
doi = {10.1007/b10906},
edition = {1},
isbn = {978-3-540-20516-6},
publisher = {Springer Berlin Heidelberg},
series = {Lecture Notes in Control and Information Sciences},
title = {{Rational Matrix Equations in Stochastic Control}},
url = {http://link.springer.com/10.1007/b10906},
year = {2004}
}

@InProceedings{Bhatia2002,
author={Bhatia, Rajendra and Elsner, Ludwig},
oeditor={Gohberg, I. and Langer, H.},
title={Positive Linear Maps and the Lyapunov Equation},
booktitle={Linear Operators and Matrices},
year={2002},
publisher={Birkh{\"a}user Basel},
oaddress={Basel},
opages={107--120},
isbn={978-3-0348-8181-4}
}

@InProceedings{Damm2003,
author={Damm, Tobias},
oeditor={Benvenuti, Luca and De Santis, Alberto and Farina, Lorenzo},
title={Stability of Linear Systems and Positive Semigroups of Symmetric Matrices},
booktitle={Positive Systems},
year={2003},
publisher={Springer Berlin Heidelberg},
oaddress={Berlin, Heidelberg},
opages={207--214},
isbn={978-3-540-44928-7}
}

@article{Lindblad1976,
author = {Lindblad, G.},
doi = {10.1007/BF01608499},
issn = {0010-3616},
journal = {Communications in Mathematical Physics},
ojournal = {Commun. Math. Phys.},
omonth = {jun},
number = {2},
opages = {119--130},
title = {{On the generators of quantum dynamical semigroups}},
url = {http://link.springer.com/10.1007/BF01608499},
volume = {48},
year = {1976}
}

@book{Bertsekas2005V1,
author = {Bertsekas, Dimitri P.},
edition = {3rd},
publisher = {Athena scientific},
title = {{Dynamic Programming and Optimal Control -- Vol I}},
year = {2005}
}

@book{Bertsekas2005V2,
author = {Bertsekas, Dimitri P.},
edition = {3rd},
publisher = {Athena scientific},
title = {{Dynamic Programming and Optimal Control -- Vol II}},
year = {2005}
}

@article{Balakrishnan2003,
author = {Balakrishnan, Venkataramanan and Vandenberghe, Lieven},
doi = {10.1109/TAC.2002.806652},
issn = {00189286},
journal = {IEEE Transactions on Automatic Control},
ojournal = {IEEE Trans. Autom. Control},
omonth = {jan},
number = {1},
opages = {30--41},
title = {{Semidefinite programming duality and linear time-invariant systems}},
volume = {48},
year = {2003}
}

@article{Chen2020a,
author = {Chen, Bilian and Li, Zhening},
doi = {10.1007/s10589-020-00177-z},
isbn = {1058902000177},
issn = {0926-6003},
journal = {Computational Optimization and Applications},
ojournal = {Comput Optim Appl},
omonth = {apr},
number = {3},
opages = {609--628},
publisher = {Springer US},
title = {{On the tensor spectral p-norm and its dual norm via partitions}},
url = {https://doi.org/10.1007/s10589-020-00177-z http://link.springer.com/10.1007/s10589-020-00177-z},
volume = {75},
year = {2020}
}

@article{Lee2014,
  archivePrefix = {arXiv},
  arxivId = {1405.7786},
  author = {Lee, Namgil and Cichocki, Andrzej},
  eprint = {1405.7786},
  opages = {1--36},
  title = {{Fundamental Tensor Operations for Large-Scale Data Analysis in Tensor Train Formats}},
  url = {http://arxiv.org/abs/1405.7786},
  year = {2014}
}

@article{Xing2021,
archivePrefix = {arXiv},
arxivId = {2106.16078},
author = {Xing, Yu and Gravell, Benjamin and He, Xingkang and Johansson, Karl Henrik and Summers, Tyler},
eprint = {2106.16078},
omonth = {jun},
opages = {1--50},
title = {{Identification of Linear Systems with Multiplicative Noise from Multiple Trajectory Data}},
url = {http://arxiv.org/abs/2106.16078},
year = {2021}
}

@article{Recht2018,
  oarchiveprefix   = {arXiv},
  oarxivid         = {1806.09460},
  author          = {Recht, Benjamin},
  doi             = {10.1146/annurev-control-053018-023825},
  oeprint          = {1806.09460},
  issn            = {2573-5144},
  journal         = {Annual Review of Control, Robotics, and Autonomous Systems},
  ojournal         = {Annu. rev. control robot. auton. syst.},
  omonth           = {may},
  number          = {1},
  opages           = {253--279},
  title           = {{A Tour of Reinforcement Learning: The View from Continuous Control}},
  url             = {https://www.annualreviews.org/doi/10.1146/annurev-control-053018-023825},
  volume          = {2},
  year            = {2019}
}

@article{Hewing2020d,
  author          = {Hewing, Lukas and Wabersich, Kim P. and Menner, Marcel and Zeilinger, Melanie N.},
  doi             = {10.1146/annurev-control-090419-075625},
  issn            = {2573-5144},
  journal         = {Annual Review of Control, Robotics, and Autonomous Systems},
  ojournal         = {Annu. rev. control robot. auton. syst.},
  number          = {1},
  opages           = {269--296},
  title           = {{Learning-Based Model Predictive Control: Toward Safe Learning in Control}},
  volume          = {3},
  year            = {2020}
}

@article{Fazel2018,
  oarchiveprefix   = {arXiv},
  oarxivid         = {1801.05039},
  author          = {Fazel, Maryam and Ge, Rong and Kakade, Sham M. and Mesbahi, Mehran},
  oeprint          = {1801.05039},
  isbn            = {9781510867963},
  journal         = {35th International Conference on Machine Learning, ICML},
  ojournal        = {35th ICML},
  omonth           = {jan},
  opages           = {2385--2413},
  title           = {{Global Convergence of Policy Gradient Methods for the Linear Quadratic Regulator}},
  url             = {http://arxiv.org/abs/1801.05039},
  volume          = {4},
  year            = {2018}
}

@inproceedings{Bradtke1994,
  author          = {Bradtke, S.J. and Ydstie, B.E. and Barto, A.G.},
  booktitle={Proceedings of 1994 American Control Conference, ACC}, 
  doi             = {10.1109/ACC.1994.735224},
  isbn            = {0-7803-1783-1},
  issn            = {07431619},
  number          = {2},
  opages           = {3475--3479},
  publisher       = {IEEE},
  title           = {{Adaptive linear quadratic control using policy iteration}},
  url             = {http://ieeexplore.ieee.org/document/735224/},
  volume          = {3},
  year            = {1994}
}

@article{Lewis2009,
  author          = {Lewis, Frank L. and Vrabie, Draguna},
  doi             = {10.1109/MCAS.2009.933854},
  isbn            = {9788995605691},
  issn            = {1531-636X},
  journal         = {IEEE Circuits and Systems Magazine},
  number          = {3},
  opages           = {32--50},
  title           = {{Reinforcement learning and adaptive dynamic programming for feedback control}},
  url             = {https://ieeexplore.ieee.org/document/5227780/},
  volume          = {9},
  year            = {2009}
}

@article{Abeille2020,
  oarchiveprefix   = {arXiv},
  oarxivid         = {2007.06482},
  author          = {Abeille, Marc and Lazaric, Alessandro},
  oeprint          = {2007.06482},
  isbn            = {9781713821120},
  journal         = {37th International Conference on Machine Learning, ICML},
  ojournal         = {ICML},
  opages           = {11--19},
  title           = {{Efficient optimistic exploration in linear-quadratic regulators via lagrangian relaxation}},
  pvolume          = {PartF16814},
  year            = {2020}
}

@article{Wonham1967,
  author          = {Wonham, W. M.},
  doi             = {10.1137/0305028},
  issn            = {0036-1402},
  journal         = {SIAM Journal on Control},
  omonth           = {aug},
  number          = {3},
  opages           = {486--500},
  publisher       = {Society for Industrial & Applied Mathematics (SIAM)},
  title           = {{Optimal Stationary Control of a Linear System with State-Dependent Noise}},
  volume          = {5},
  year            = {1967}
}

@article{Gravell2019,
  archiveprefix   = {arXiv},
  arxivid         = {1905.13547},
  author          = {Gravell, Benjamin and Esfahani, Peyman Mohajerin and Summers, Tyler},
  eprint          = {1905.13547},
  omonth           = {may},
  title           = {{Learning robust control for LQR systems with multiplicative noise via policy gradient}},
  url             = {http://arxiv.org/abs/1905.13547},
  year            = {2019}
}

@article{Wang2018,
  author          = {Wang, Tao and Zhang, Huaguang and Luo, Yanhong},
  doi             = {10.1016/j.neucom.2018.04.018},
  issn            = {09252312},
  journal         = {Neurocomputing},
  omonth           = {oct},
  opages           = {1--8},
  publisher       = {Elsevier B.V.},
  title           = {{Stochastic linear quadratic optimal control for model-free discrete-time systems based on Q-learning algorithm}},
  url             = {https://linkinghub.elsevier.com/retrieve/pii/S0925231218304387},
  volume          = {312},
  year            = {2018}
}

@article{Pang2021,
  archiveprefix   = {arXiv},
  arxivid         = {2107.07788},
  author          = {Pang, Bo and Jiang, Zhong-Ping},
  eprint          = {2107.07788},
  omonth           = {jul},
  opages           = {1--9},
  title           = {{Reinforcement Learning for Adaptive Optimal Stationary Control of Linear Stochastic Systems}},
  url             = {http://arxiv.org/abs/2107.07788},
  year            = {2021}
}

@inproceedings{Di2021,
  author    = {Di, Bolei and Lamperski, Andrew},
  booktitle={2021 American Control Conference (ACC)}, 
  title     = {Confidence Bounds on Identification of Linear Systems with Multiplicative Noise},
  year      = {2021},
  volume    = {},
  number    = {},
  opages     = {2212-2217},
  doi       = {10.23919/ACC50511.2021.9482968}
}

@article{Todorov2002,
  author          = {Todorov, Emanuel and Jordan, Michael},
  doi             = {10.1038/nn963},
  issn            = {1097-6256},
  journal         = {Nature Neuroscience},
  ijournal         = {Nat. Neurosci},
  omonth           = {nov},
  number          = {11},
  opages           = {1226--1235},
  title           = {{Optimal feedback control as a theory of motor coordination}},
  url             = {http://www.nature.com/articles/nn963},
  volume          = {5},
  year            = {2002}
}

@article{Todorov2005,
  author          = {Todorov, Emanuel},
  doi             = {10.1162/0899766053491887},
  issn            = {0899-7667},
  journal         = {Neural Computation},
  omonth           = {may},
  number          = {5},
  opages           = {1084--1108},
  title           = {{Stochastic Optimal Control and Estimation Methods Adapted to the Noise Characteristics of the Sensorimotor System}},
  url             = {https://direct.mit.edu/neco/article/17/5/1084-1108/6949},
  volume          = {17},
  year            = {2005}
}

@article{Russo2018,
  author          = {Russo, Giovanni and Shorten, Robert},
  doi             = {10.1016/j.physd.2018.01.003},
  issn            = {01672789},
  journal         = {Physica D: Nonlinear Phenomena},
  opages           = {47--54},
  publisher       = {Elsevier B.V.},
  title           = {{On common noise-induced synchronization in complex networks with state-dependent noise diffusion processes}},
  url             = {https://doi.org/10.1016/j.physd.2018.01.003},
  volume          = {369},
  year            = {2018}
}

@article{Mohler1980a,
  author          = {Mohler, R. R. and Kolodziej, W. J.},
  doi             = {10.1109/TSMC.1980.4308421},
  issn            = {0018-9472},
  journal         = {IEEE Transactions on Systems, Man, and Cybernetics},
  number          = {12},
  opages           = {913--918},
  title           = {{An Overview of Stochastic Bilinear Control Processes}},
  url             = {http://ieeexplore.ieee.org/document/4308421/},
  volume          = {10},
  year            = {1980}
}

@article{Mohler1980b,
  author          = {Mohler, R.R. and Bruni, Carlo and Gandolfi, Alberto},
  doi             = {10.1109/PROC.1980.11775},
  issn            = {0018-9219},
  journal         = {Proceedings of the IEEE},
  number          = {8},
  opages           = {964--990},
  title           = {{A systems approach to immunology}},
  url             = {http://ieeexplore.ieee.org/document/1456044/},
  volume          = {68},
  year            = {1980}
}

@article{Afshari2020,
  author          = {Afshari, Amir and Karrari, Mehdi and Baghaee, Hamid Reza and Gharehpetian, Gevork B.},
  doi             = {10.1049/iet-rpg.2019.1180},
  issn            = {1752-1416},
  journal         = {IET Renewable Power Generation},
  ojournal         = {IET Renew. Power Gener.},
  omonth           = {jun},
  number          = {8},
  opages           = {1321--1331},
  title           = {{Resilient cooperative control of AC microgrids considering relative state‐dependent noises and communication time‐delays}},
  url             = {https://onlinelibrary.wiley.com/doi/10.1049/iet-rpg.2019.1180},
  volume          = {14},
  year            = {2020}
}

@article{Wang2002,
  author          = {Wang, Fan and Balakrishnan, Venkataramanan},
  doi             = {10.1109/78.992124},
  issn            = {1053587X},
  journal         = {IEEE Transactions on Signal Processing},
  ojournal         = {IEEE Trans. Signal Process.},
  omonth           = {apr},
  number          = {4},
  opages           = {803--813},
  title           = {{Robust Kalman filters for linear time-varying systems with stochastic parametric uncertainties}},
  url             = {http://ieeexplore.ieee.org/document/992124/},
  volume          = {50},
  year            = {2002}
}

@article{Sura2005,
  author          = {Sura, Philip and Newman, Matthew and Penland, C{\'{e}}cile and Sardeshmukh, Prashant},
  doi             = {10.1175/JAS3408.1},
  issn            = {00224928},
  journal         = {Journal of the Atmospheric Sciences},
  ojournal         = {J Atmos Sci},
  number          = {5},
  opages           = {1391--1409},
  title           = {{Multiplicative noise and non-Gaussianity: A paradigm for atmospheric regimes?}},
  volume          = {62},
  year            = {2005}
}

@article{Doyle1978,
  author          = {Doyle, J.},
  doi             = {10.1109/TAC.1978.1101812},
  issn            = {0018-9286},
  journal        = {IEEE Transactions on Automatic Control},
  ojournal         = {IEEE Trans. Autom. Control},
  omonth           = {aug},
  number          = {4},
  opages           = {756--757},
  title           = {{Guaranteed margins for LQG regulators}},
  url             = {http://ieeexplore.ieee.org/document/1101812/},
  volume          = {23},
  year            = {1978}
}

@inproceedings{Doyle1996,
  author          = {Doyle, J.},
  booktitle      = {Proceedings of 35th IEEE Conference on Decision and Control},
  obooktitle       = {CDC},
  doi             = {10.1109/CDC.1996.572756},
  isbn            = {0-7803-3590-2},
  opages           = {1595--1598},
  publisher       = {IEEE},
  title           = {{Robust and optimal control}},
  url             = {http://ieeexplore.ieee.org/document/572756/},
  volume          = {2},
  year            = {1996}
}

@article{Rami2002d,
  author          = {Rami, M. Ait and Chen, X. and Zhou, X.Y.},
  doi             = {10.1023/A:1016578629272},
  issn            = {09255001},
  journal         = {Journal of Global Optimization},
  ojournal         = {J Glob Optim},
  number          = {3/4},
  opages           = {245--265},
  publisher       = {Springer},
  title           = {{Discrete-time Indefinite LQ Control with State and Control Dependent Noises}},
  url             = {http://link.springer.com/10.1023/A:1016578629272},
  volume          = {23},
  year            = {2002}
}

@article{ElGhaoui1995,
  author          = {{El Ghaoui}, Laurent},
  doi             = {10.1016/0167-6911(94)00045-W},
  issn            = {01676911},
  journal         = {Systems and Control Letters},
  ojournal         = {Control Syst. Lett.},
  publisher       = {IEEE},
  omonth           = {feb},
  number          = {3},
  opages           = {223--228},
  title           = {{State-feedback control of systems with multiplicative noise via linear matrix inequalities}},
  volume          = {24},
  year            = {1995}
}

@article{Kubrusly1981c,
  author          = {Kubrusly, C. S.},
  doi             = {10.1080/00207178108922924},
  issn            = {0020-7179},
  journal         = {International Journal of Control},
  ojournal         = {Int. J. Control},
  omonth           = {feb},
  number          = {2},
  opages           = {291--309},
  title           = {{Identification of discrete-time stochastic bilinear systems}},
  url             = {http://www.tandfonline.com/doi/abs/10.1080/00207178108922924},
  volume          = {33},
  year            = {1981}
}

@article{Greville1966,
  author          = {Greville, T. N. E.},
  doi             = {10.1137/1008107},
  issn            = {0036-1445},
  journal         = {SIAM Review},
  omonth           = {oct},
  number          = {4},
  opages           = {518--521},
  title           = {{Note on the Generalized Inverse of a Matrix Product}},
  url             = {http://epubs.siam.org/doi/10.1137/1008107},
  volume          = {8},
  year            = {1966}
}

@book{Astrom2008,
  author          = {Astrom, Karl J. and Wittenmark, Bjorn},
  edition         = {2nd},
  publisher       = {Dover Publications, Inc.},
  title           = {{Adaptive Control}},
  year            = {2008}
}

@book{Ljung1999,
  author          = {Ljung, Lennart},
  publisher       = {Prentice Hall PTR},
  title           = {{System Identification: Theory for the User}},
  year            = {1999}
}

@book{Jungers2009,
  address         = {Berlin, Heidelberg},
  author          = {Jungers, Rapha{\"{e}}l},
  doi             = {10.1007/978-3-540-95980-9},
  isbn            = {978-3-540-95979-3},
  publisher       = {Springer Berlin Heidelberg},
  series          = {Lecture Notes in Control and Information Sciences},
  title           = {{The Joint Spectral Radius}},
  url             = {http://link.springer.com/10.1007/978-3-540-95980-9},
  volume          = {385},
  year            = {2009}
}

@book{Magnus2019,
  author          = {Magnus, Jan R. and Neudecker, Heinz},
  doi             = {10.1002/9781119541219},
  isbn            = {9781119541202},
  omonth           = {feb},
  publisher       = {Wiley},
  series          = {Wiley Series in Probability and Statistics},
  title           = {{Matrix Differential Calculus with Applications in Statistics and Econometrics}},
  url             = {https://onlinelibrary.wiley.com/doi/book/10.1002/9781119541219},
  year            = {2019}
}

@article{Brunke2022,
  title     = {Safe learning in robotics: From learning-based control to safe reinforcement learning},
  author    = {Brunke, Lukas and Greeff, Melissa and Hall, Adam W and Yuan, Zhaocong and Zhou, Siqi and Panerati, Jacopo and Schoellig, Angela P},
  journal   = {Annual Review of Control, Robotics, and Autonomous Systems},
  volume    = {5},
  pages     = {411--444},
  year      = {2022},
  publisher = {Annual Reviews}
}

@inproceedings{Simchowitz2020,
  title     = {Naive Exploration is Optimal for Online {LQR}},
  author    = {Simchowitz, Max and Foster, Dylan},
  booktitle = {Proceedings of the 37th International Conference on Machine Learning},
  pages     = {8937--8948},
  year      = {2020},
  editor    = {III, Hal Daumé and Singh, Aarti},
  volume    = {119},
  oseries    = {Proceedings of Machine Learning Research},
  month     = {13--18 Jul},
  publisher = {PMLR},
  url       = {https://proceedings.mlr.press/v119/simchowitz20a.html}
}

@article{Gershon2006,
  author          = {Gershon, E. and Shaked, U.},
  doi             = {10.1016/j.sysconle.2005.07.010},
  issn            = {01676911},
  journal         = {Systems and Control Letters},
  number          = {3},
  pages           = {232--239},
  title           = {{Static $H_2$ and $H_\infty$ output-feedback of discrete-time LTI systems with state multiplicative noise}},
  volume          = {55},
  year            = {2006}
}

@article{Willsky1976b,
  author          = {Willsky, Alan S. and Marcus, Steven I.},
  doi             = {10.1016/0016-0032(76)90135-6},
  issn            = {00160032},
  journal         = {Journal of the Franklin Institute},
  month           = {jan},
  number          = {1-2},
  pages           = {103--122},
  title           = {{Analysis of bilinear noise models in circuits and devices}},
  url             = {https://linkinghub.elsevier.com/retrieve/pii/0016003276901356},
  volume          = {301},
  year            = {1976}
}

@article{Athans1977,
  author          = {Athans, M. and Ku, R. and Gershwin, S.},
  doi             = {10.1109/TAC.1977.1101526},
  issn            = {0018-9286},
  journal         = {IEEE Transactions on Automatic Control},
  month           = {jun},
  number          = {3},
  pages           = {491--495},
  title           = {{The uncertainty threshold principle: Some fundamental limitations of optimal decision making under dynamic uncertainty}},
  url             = {http://ieeexplore.ieee.org/document/1101526/},
  volume          = {22},
  year            = {1977}
}

@article{Ding2019,
  author          = {Ding, Jian and Peres, Yuval and Ranade, Gireeja and Zhai, Alex},
  journal         = {The Annals of Applied Probability},
  number          = {4},
  pages           = {1963--1992},
  title           = {When multiplicative noise stymies control},
  url             = {https://www.jstor.org/stable/26754147},
  volume          = {29},
  year            = {2019}
}

@misc{Coppens2023,
  title         = {Policy iteration using {Q-functions}: Linear dynamics with multiplicative noise},
  author        = {Peter Coppens and Panagiotis Patrinos},
  year          = {2022},
  eprint        = {2212.01192},
  archiveprefix = {arXiv},
  primaryclass  = {math.OC}
}
%
\begin{pub}%
%
\begin{IEEEbiography}[{\includegraphics[width=1in,height=1.25in,clip,keepaspectratio]{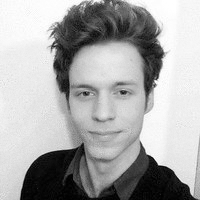}}]{Peter Coppens}%
    is a PhD researcher at the Department of Electrical Engineering (ESAT) of KU Leuven, Belgium and a recipient 
    of a \emph{FWO fundamental research fellowship}. He received his M.Eng. in Mathematical Engineering from the KU Leuven
    in 2019. His current research interests lie in learning control, enabled by distributionally robust optimization. 
    Specifically its use in model predictive control an safety critical applications. 
\end{IEEEbiography}
%
\begin{IEEEbiography}[{\includegraphics[width=1in,height=1.25in,clip,keepaspectratio]{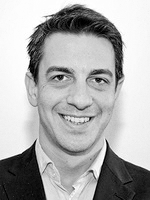}}]{Panagiotis (Panos) Patrinos}%
    is associate professor at the Department of Electrical Engineering (ESAT) of KU Leuven, Belgium. 
    In 2014 he was a visiting professor at Stanford University. He received his PhD in Control and Optimization, 
    M.S. in Applied Mathematics and M.Eng. in Chemical Engineering from the National Technical University of Athens 
    in 2010, 2005 and 2003, respectively. After his PhD he held postdoc positions at the University of Trento and 
    IMT Lucca, Italy, where he became an assistant professor in 2012. His current research interests lie in the 
    intersection of optimization, control and learning. In particular he is interested in the theory and algorithms 
    for structured nonconvex optimization as well as learning-based, model predictive control with a wide range of 
    applications including autonomous vehicles, machine learning and signal processing.  
\end{IEEEbiography}%
\end{pub}

\appendices\begin{arxiv}\section{Proofs for preliminaries} \label{app:tensors}
\paragraph{Proof of \cref{lem:sopbnd}}    
We begin by noting that
\begin{align*}
    \nrm{\op{S}}_2^2 &\leq \max_{X \in \sym{m}} \{\nrm{\op{S}(X)}_F^2 \colon \nrm{X}_F \leq \sqrt{m}\} \\
                          &\leq m \max_{x \in \Re^{\sd{m}}} \{ \trans{x} \trans{\opm{S}} \opm{S} x \colon \nrm{x}_2 \leq 1\},
\end{align*}
where we used $\nrm{A}_2 \leq \nrm{A}_F \leq \sqrt{\rk(A)} \nrm{A}_2$ for any $A \in \sym{d}$
for the first inequality and $\nrm{X}_F = \nrm{\svec(X)}_2$ for the second. The proof is 
then concluded by noting that the final expression equals $m \lambda_1(\trans{\opm{S}} \opm{S}) = m \nrm{\opm{S}}_2^2$. \qed{}

\paragraph{Proof of \cref{lem:tencp}}
Since $W \sgeq 0$, there is some $C$ s.t. $W = C \trans{C}$. Then,
\begin{equation*}
    \op{S}(X) = \ten{A}_{(1)} (C \kron I)(I \kron X) (\trans{C} \kron I) \trans{\ten{A}_{(1)}}.
\end{equation*}
Letting $\ten{V} = \ten{A} \ttimes{1} \trans{C}$ and applying \cref{prop:kronunfold},
implies $I \ten{A}_{(1)} (C \kron I) = (\tucker{\ten{A}; I, I, \trans{C}})_{(3)} = \ten{V}_{(3)}$. 
Therefore 
\begin{align*}
    \op{S}(X) &= \ten{V}_{(1)} (I \kron X) \trans{\ten{V}_{(1)}} \\
              &= \begin{bmatrix} \ten{V}_{::1} & \dots & \ten{V}_{::r} \end{bmatrix}
              (I \kron X) \trans{\begin{bmatrix}
                \ten{V}_{::1} & \dots & \ten{V}_{::r}
              \end{bmatrix}} \\
              &=\sum_{i=1}^r [\ten{V}]_{::i} X \trans{[\ten{V}]_{::i}},
\end{align*}
as in \eqref{eq:cp-def}. \qed{}

\paragraph{Proof of \cref{cor:cpmon}}
Let $P$ denote the \emph{perfect shuffle} \cite[Thm.~4.3.8]{Horn1991} with $\trans{P} P = I$. Then 
\begin{align*}
    \op{S}(W; X) &= \ten{A}_{(1)} \trans{P} P (W \kron X) \trans{P} P \trans{\ten{A}_{(1)}} \\
                 &= \ten{A}_{(1)} \trans{P} (X \kron W) P \trans{\ten{A}_{(1)}},
\end{align*}
by \cite[Cor.~4.3.10]{Horn1991}, which is of the form \eqref{eq:cp-tensor-def}. So the conclusion follows by \cref{lem:tencp}. 
The second result follows from $\op{S}(W; X)$ being CP and linearity. \qed{}

\paragraph{Proof of \cref{lem:tensor-kron}}
    From \cite[Prop.~2.2]{Lee2014} we have 
    \begin{equation} \label{eq:tensor-kron:a}
        \ten{Y} \kron \ten{Y} = \tucker{\ten{T} \kron \ten{T}; X_1 \kron X_1, X_2 \kron X_2, X_3 \kron X_3}.
    \end{equation}
    For the proof we focus on one argument 
    for conciseness. So let $\ten{Y} = \tucker{\ten{T}; X_1; \dots}$. Then
    \begin{align*}
        \ten{Y} \skron \ten{Y}  &\labelrel={step:skron} \tucker{\ten{Y} \kron \ten{Y}; Q_{p_1}, \dots} \\
                                &\labelrel={step:resi} \tucker{\tucker{\ten{T} \kron \ten{T}; X_1 \kron X_1, \dots}; Q_{p_1}, \dots} \\
                                &\labelrel={step:assoc} \tucker{\ten{T} \kron \ten{T}; Q_{p_1} (X_1 \kron X_1), \dots}. 
    \end{align*}
    Here \ref{step:skron} follows by definition of $\skron$, \ref{step:resi} from \cref{eq:tensor-kron:a} and 
    \ref{step:assoc} follows by noting $(\ten{Y} \ttimes{1} A) \ttimes{1} B = \ten{Y} \ttimes{1} BA$ \cite[\S2.5]{Kolda2009}.
    Next for every $H \in \sym{p_1}$ we have $\trans{\svec}(H) Q_{p_1} (X_1 \kron X_1) = \trans{\vec}(X_1 H \trans{X_1})$,
    by \cref{def:svec} and \cref{eq:kronfund}. The result is a vectorization of a symmetric matrix in $\sym{q_1}$. Hence,
    since $\trans{Q}_{q_1} Q_{q_1}$ acts like the identity on such vectors, we have 
    $\trans{h} Q_{p_1} (X_1 \kron X_1) = \trans{h} Q_{p_1} (X_1 \kron X_1) \trans{Q}_{q_1} Q_{q_1}$
    for all $h \in \Re^{\sd{{q_1}}}$. Hence, reversing the steps of the earlier display, gives
    \begin{align*}
        \ten{Y} \skron \ten{Y}  &= \tucker{\ten{T} \kron \ten{T}; Q_{p_1} (X_1 \kron X_1) \trans{Q_{q_1}} Q_{q_1}, \dots} \\
                                &= \tucker{\tucker{\ten{T} \kron \ten{T}; Q_{q_1}, \dots}; Q_{p_1} (X_1 \kron X_1) \trans{Q_{q_1}}, \dots} \\
                                &= \tucker{\ten{T} \skron \ten{T}; X_1 \skron X_1, \dots}.
    \end{align*}
    Thus we have shown the required result. \qed{}

We can use \cref{lem:tensor-kron} to show 
\begin{proposition} \label{prop:cpunfold}
    Let $\op{S}(W; X)$ be a CP operator parameterized as in \cref{eq:cp-tensor-def} with some $\ten{V}$. 
    Then $\tr[P \op{S}(W; X)]$ equals:
    \begin{enumerate}
        \item \label{prop:cpunfold:a} $\begin{aligned}[t]&\tr[P \ten{V}_{(1)} (W \kron X) \trans{\ten{V}_{(1)}}] = \tr[X \ten{V}_{(2)} (W \kron P) \trans{\ten{V}_{(2)}}] \\&\quad= \tr[W \ten{V}_{(3)} (P \kron X) \trans{\ten{V}_{(3)}}]\end{aligned}$;
        \item \label{prop:cpunfold:b} $\tucker{\ten{V} \kron \ten{V}; \vec(P), \vec(X), \vec(W)}$;
        \item \label{prop:cpunfold:c} $\tucker{\ten{V} \skron \ten{V}; \svec(P), \svec(X), \svec(W)}$;
    \end{enumerate}
    for all $P \in \sym{n_p}, X \in \sym{n_x}$ and $W \in \sym{n_w}$.
\end{proposition}
\begin{proof}
    By the eigenvalue decomposition, $P = \sum_{i=1}^{n_p} \alpha_i p_i \trans{p_i}$,
    $X = \sum_{j=1}^{n_x} \beta_j x_j \trans{x_j}$ and $W = \sum_{k=1}^{n_w} \gamma_k w_k \trans{w_k}$. 
    
    Writing out $\tr[P\op{S}(W; X)]$ with \cref{lem:tencp} gives
    \begin{align*}
        &\ssum_{i,j,k} \alpha_i \beta_j \gamma_k \tr[p_i \trans{p_i} \ten{V}_{(1)} (w_k \trans{w_k} \kron x_j \trans{x_j}) \trans{\ten{V}_{(1)}}] \\
        &\qquad = \ssum_{i,j,k} \alpha_i \beta_j \gamma_k [(\trans{p_i} \ten{V}_{(1)} (w_k \kron x_j))^2],
    \end{align*}
    where we used \cite[Thm.~E.1.3]{DeKlerk2002} for $(w_k \trans{w_k} \kron x_j \trans{x_j}) = (w_k \kron x_j) \trans{(w_k \kron x_j)}$. 
    We show how the expression in square brackets can be rewritten. This, through reversal of the
    steps above enables proving the results.

    By \cref{prop:kronunfold} $(\trans{p} \ten{V}_{(1)} (w \kron x))^2 = \tucker{\ten{V}; p, x, w}^2$. 
    Unfolding along $n = 2$ and $3$ shows the other equalities in \cref{prop:cpunfold:a}. 
    Moreover since $\tucker{\ten{V}; p, x, w}^2 = \tucker{\ten{V}; p, x, w} \kron \tucker{\ten{V}; p, x, w}$
    we can apply \cref{lem:tensor-kron} to analogously get \cref{prop:cpunfold:b}. Finally note that 
    $\trans{Q}_{q_1} Q_{n_p} \vec{(P)} = \trans{Q}_{n_p} \svec{(P)}$ since $P \in \sym{n_p}$ 
    and similarly for $X$ and $W$. From \cite[\S2.5]{Kolda2009} $\ten{T} \btimes{1} \trans{Q} x = (\ten{T} \ttimes{1} Q) \btimes{1} x$,
    which combined with the definition of $\ten{V} \skron \ten{V}$ shows \cref{prop:cpunfold:c}.
\end{proof}

\paragraph{Proof of \cref{cor:cpadj}}
We can directly apply \cref{prop:cpunfold:b} to show the expression for $\adj{\op{S}}(W; X)$. For $\opm{S}$ 
note that $\tr[P \op{S}(W; X)] = \trans{\svec}(P) \opm{S} \svec(X)$. 
We can rewrite the Tucker operator in \cref{prop:cpunfold:c} as
\begin{align*}
    &\tucker{\ten{V} \skron \ten{V}; \svec(P), \svec(X), \svec(W)} \\
    &\qquad = (((\ten{V} \skron \ten{V}) \btimes{3} \svec(W)) \btimes{1} \svec(P)) \svec(X) \\
    &\qquad = \trans{\svec(P)} ((\ten{V} \skron \ten{V}) \btimes{3} \svec(W)) \svec(X).
\end{align*}
Since this holds for all $P \in \sym{m}$, $X \in \sym{n}$ we have shown $(\ten{V} \skron \ten{V}) \btimes{3} \svec(W) = \opm{S}$.
\end{arxiv}

\section{Linear Quadratic Regulation} \label{app:lqr}
This section gives proofs related to LQR. We begin by introducing some auxiliary operators and their adjoints:
\begin{lemma} \label{lem:opadj}
    Let $\op{E}$ be as in \cref{eq:dynsm_}, ${\Pi}_K(X) = [I; K] X [I, \trans{K}]$,
     $\op{E}_K(X) = \op{E}(\Pi_K(X))$. \begin{arxiv}Consider the Lyapunov operator $\op{L}(Z) = X - \op{E}(Z)$ and $\op{L}_K(X) = X - \op{E}_K(X)$. 
    Here $X$ is the top-left $n_x \times n_x$ block of $Z$.

    \end{arxiv}Then we have the adjoints:
    \begin{enumerate}
        \item $\adj{\Pi}_K(H) = [I, \trans{K}] H [I; K]$;
        \item $ \adj{\op{E}}(P) = \begin{bmatrix}
                    \adj{\op{F}}(P) & \trans{(\adj{\op{H}})}(P) \\ \adj{\op{H}}(P) & \adj{\op{G}}(P)
                \end{bmatrix}$;
        \item $\adj{\op{E}}_K(P) = \adj{\Pi}_K(\adj{\op{E}}(P))$;
        \ilarxiv{\item $ \adj{\op{L}}(P) = \blkdiag(P, 0) - \adj{\op{E}}(P)$;}
        \ilarxiv{\item $\adj{\op{L}}_K(P) = \adj{\Pi}_K(\adj{\op{L}}(P))$.}
    \end{enumerate}
    
    The elementary adjoints are $\adj{\op{F}}(P) = \ten{A}_{(2)} (\smoment \kron P) \trans{\ten{A}_{(2)}}$, 
    $\adj{\op{G}}(P) = \ten{B}_{(2)} (\smoment \kron P) \trans{\ten{B}_{(2)}}$ and $\adj{\op{H}}(P) = \ten{B}_{(2)}(\smoment \otimes P) \trans{\ten{A}_{(2)}}$.
\end{lemma}
\begin{proof}
    Only elementary algebra, \cref{lem:tencp} and \cref{cor:cpadj} are used.
\end{proof}

\begin{lemma} \label{lem:z-partition}
    Any $Z \sgeq 0$ can be partitioned as:
    \begin{equation*}
        \begin{bmatrix}
            X & X \trans{K} \\ K X & K X \trans{K} + \Delta            
        \end{bmatrix}
    \end{equation*}
    for $\Delta \sgeq 0$ and $X \sgeq 0$.
\end{lemma}
\begin{proof}
    Let $Z = [X, V; \trans{V}, U]$. By \cite[\S{}A.5.5]{Boyd2004} $Z \sgeq 0$ iff
    \begin{equation*}
        X \sgeq 0, \quad \img{V} \subseteq \img{X}, \quad U - \trans{V} X^\dagger V \sgeq 0. 
    \end{equation*}
    So there exists a $K$ s.t. $V = X \trans{K}$ and $\Delta = U - \trans{V} \pinv{X} V = U -\trans{K} X K \sgeq 0$
\end{proof}

\paragraph*{Proof of \cref{thm:cpmulteq}}
We use \cref{asm:model} and \cref{lem:model-eq} to consider \eqref{eq:dynten} instead of \eqref{eq:dyn} without loss of generality. 
We start with a feasible sequence for \cref{eq:slqr} and 
construct an (equivalent) feasible sequence for \cref{eq:lqrcp}. Let $H = \blkdiag(Q, R)$, $Z_t = \E[z_t \trans{z_t}] \sgeq 0$ with $z_t = (x_t, u_t)$ 
and $x_t, u_t$ satisfying \eqref{eq:dynten}. 
Note that, by \cref{eq:dynsm_}, any sequence $Z_t$ constructed as such satisfies the constraints of \eqref{eq:lqrcp}.
Then
\begin{equation} \label{eq:cost-eq-cp-slqr}
    \E\left[ \ssum_{t=0}^\infty \trans{x}_t Q x_t + \trans{u}_t R u_t \right] = \ssum_{t=0}^\infty \tr[H Z_t]. 
\end{equation}
So the cost of $(x_t, u_t)$ for \cref{eq:slqr} equals that of the constructed $Z_t$ in \cref{eq:lqrcp}. 
So $\mathrm{val}\eqref{eq:lqrcp} \leq \mathrm{val}\eqref{eq:slqr}$. 

Next, we argue that for any sequence $Z_t$ 
feasible for \cref{eq:lqrcp}, we can construct a sequence $z_t = (x_t, u_t)$ feasible for \cref{eq:slqr}. 
Note that for any $X \sgeq 0$ we can easily construct a random vector $x$ with $\E[x \trans{x}] = X$. 
So let $x_0$ be such a vector for $X_0$. Then, by \cref{lem:z-partition}, $Z_0$ can be partitioned as 
$[X_0, X_0\trans{K}_0; K_0 X_0, K_0 X_0 \trans{K}_0 + \Delta_0]$ with $\Delta_0 \sgeq 0$ for 
which we pick a random vector $\delta_0$. So if we take $u_0 = K_0 x_0 + \delta_0$ and $z_0 = (x_0, u_0)$ 
then $\E[z_0 \trans{z_0}] = Z_0$. Repeating the same argument starting at $x_1$ and continuing for all time steps 
allows us to construct a feasible trajectory for \cref{eq:slqr}. Again 
by $\E[z_t \trans{z_t}] = Z_t$ and \cref{eq:cost-eq-cp-slqr} the cost for 
both trajectories is the same. Thus $\mathrm{val}\eqref{eq:lqrcp} \geq \mathrm{val}\eqref{eq:slqr}$. \qed{}

\begin{arxiv}
We continue with proving auxiliary results for \cref{thm:cplqr}. 
We first consider the finite horizon version of \cref{eq:lqrcp}:
\begin{equation} \label{eq:lqr-finite}
    \begin{alignedat}{2}
        &\minimize_{Z_t} &\quad& \sum_{t=0}^{N-1} \tr[Z_t H] + \tr[P X_N] \\
                             &\stt && Z_t = \begin{bmatrix}
                                 X_t & V_t \\ \trans{V_t} & U_t
                             \end{bmatrix} \sgeq 0, \\&&&X_{t+1} = \op{E}(Z_t), \quad \forall t \in \N_{0:N-1}. 
    \end{alignedat}
\end{equation}
We assume $H = \blkdiag(Q, R) \sgt 0$ and denote the optimal value as $\op{J}^N(P; X_0)$ and $\op{J}^{\star}(X_0) \dfn \lim_{N \to \infty} \op{J}^{N}(0; X_0)$.

We prove the fundamental property linking $\op{R}$ to \eqref{eq:lqr-finite}:
\begin{proposition} \label{prop:lqr-finite}
    Consider $\op{J}^N$ and $\op{R}$ defined in \cref{sec:lqr}. Let
    \begin{equation*}
        P_k \dfn \op{R}^{k}(P), \, K_k = -(R + \adj{\op{G}}(P_k))^{-1} \adj{\op{H}}(P_k), \, k \in \N_{0:N}.
    \end{equation*}
    Then $Z_t = \Pi_{K_{N-t-1}}(X_t)$ for $t \in \N_{0:N-1}$ achieves the minimum:
    \begin{equation*}
        \op{J}^N(P; X_0) = \tr[P_N X_0].
    \end{equation*}
\end{proposition}
\begin{proof}
    The result follows from \emph{dynamic programming (DP)} \cite[\S4.1]{Bertsekas2005V1}. Starting from $\op{J}^0(P; X_0) = \tr[P_0X_0]$, with $P_0 = P$, 
    DP constructs $\op{J}^k(P; X_0) = \tr[P_k X_0]$ using:
    \begin{equation} \label{eq:bellman}
        \tr[P_{k+1} X] = \min_{Z \sgeq 0} \{
            \tr[Z H] + \tr[P_{k} \op{E}(Z)]\},
    \end{equation}
    Note that $\tr[P_{k} \op{E}(Z)] = \tr[\adj{\op{E}}(P_k) Z]$ by definition of the adjoint.
    We use \cref{lem:opadj} to write $\adj{\op{E}}$ in terms of $\adj{\op{F}}$, $\adj{\op{G}}$ and $\adj{\op{H}}$. 
    Introducing $X, K$ and $\Delta$ to partition $Z$ as in \cref{lem:z-partition} allows rewriting the cost of \eqref{eq:bellman} (omitting constant terms) as:
    \begin{align*}
        &\tr[\trans{(\adj{\op{H}}(P_{k}))} K X] + \tr[\adj{\op{H}}(P_k) X \trans{K}]\\
        &\quad + \tr[(R + \adj{\op{G}}(P_{k}))K X \trans{K}] + \tr[(R + \adj{\op{G}}(P_k)) \Delta],
    \end{align*}
    which is minimized uniquely at $\Delta = 0$, by noting that $R \sgt 0$ in conjunction with
    \cref{lem:trivialinequality}.
    We can evaluate the gradient of the remaining terms with respect to $K$ using the identities in \cite[Table.~9.4]{Magnus2019}. 
    This gives the following first-order optimality conditions:
    \begin{equation} \label{eq:first-order}
        2 \adj{\op{H}}(P_k) X + 2 (\adj{\op{G}}(P_k) + R) K X = 0.
    \end{equation}
    \update*{This should hold for all $X \in \psd{n_x}$}. Since $R \sgt 0$ this is the case iff
    \begin{equation*}
        K = -(R + \adj{\op{G}}(P_k))^{-1}\adj{\op{H}}(P_k) = K_k.
    \end{equation*}
    For these values of $K$ and $\Delta$ we have $Z = \Pi_{K_k}(X)$, 
    with $\Pi_{K_k}$ as in \cref{lem:opadj}. Plugging into \eqref{eq:bellman}:
    \begin{align}
        &\tr[P_{k+1} X] = \tr[\Pi_{K_k}(X) H] + \tr[P_k \op{E}(\Pi_{K_k}(X))] \nonumber \\
                       &\quad = \tr[X \adj{\Pi}_{K_k}(H + \adj{\op{E}}(P_k))] = \tr[X \op{R}(P_k)], \label{eq:ric-alt}
    \end{align}
    where the second equality follows from linearity and adjoints, while the third follows from \cref{lem:opadj} 
    and cancellation of the inverses in $K_k$. Iterating over $k$, repeatedly 
    computing the optimal policy $K_k$ by the same procedure as usual in dynamic programming (cf. \cite[Prop.~1.3.1]{Bertsekas2005V1}), 
    completes the proof. Note that $k$ counts backwards in time, hence $K_{N-t-1}$ in $Z_t$. 
\end{proof}
The steps in the proof of \cref{prop:lqr-finite} can be used to prove some additional properties of $\op{R}$:
\begin{corollary} \label{cor:ric-prop}
    For all $P \sgeq 0$, $K' \neq -(R + \adj{\op{G}}(P))^{-1} \adj{\op{H}}(P)$ and $H = \blkdiag(Q, R) \sgt 0$:
    \begin{enumerate}
        \item \label{cor:ric-prop:a} $\op{R}(P') \sleq \op{R}(P)$ if $P' \sleq P$;
        \item \label{cor:ric-prop:b} $\op{R}(P) \nsleq \update*{\adj{\Pi}_{K'}(H + \adj{\op{E}}(P))}$.
    \end{enumerate}
\end{corollary}
\begin{proof}
    Result (i) is simply monotonicity of the Bellman operator \cite[Lem.~1.1.1]{Bertsekas2005V2}.
    Specifically, by observing \eqref{eq:bellman} it is clear that the cost is larger for $P$ than for $P'$ 
    since $\tr[P \op{E}(Z)] \geq \tr[P' \op{E}(Z)]$ for all $Z$. Hence the minimum is also larger. 
    Moreover using \cref{eq:ric-alt} gives $\tr[X \op{R}(P)] \geq \tr[X \op{R}(P')]$. Taking $X = x\trans{x}$ for any $x$
    gives $\trans{x} \op{R}(P) x \geq \trans{x} \op{R}(P') x$ thus showing (i). 
    Result (ii) follows from $H = \blkdiag(Q, R) \sgt 0$, causing the minimizer we computed for \eqref{eq:bellman} 
    to be unique since it should satisfy \cref{eq:first-order}. \update*{So we can claim, using \cref{eq:ric-alt}, 
    \begin{equation*}
        \exists X \in \psd{n_x} \colon \tr[X \op{R}(P)] < \tr[X \adj{\Pi}_{K'}(H + \adj{\op{E}}(P))].
    \end{equation*}
    We take the rank one decomposition $X = \sum_{i=1}^{r} x_i \trans{x_i}$. Then $\sum_{i=1}^{r} \trans{x_i} \op{R}(P) x_i < \sum_{i=1}^{r}\trans{x_i} \adj{\Pi}_{K'}(H + \adj{\op{E}}(P)) x_i$. 
    So we conclude that $\op{R}(P) \nsleq \adj{\Pi}_{K'}(H + \adj{\op{E}}(P))$. 
    }
\end{proof}

\begin{prepupdate}
Next, we study the fixed point of the Riccati equation.
\begin{proposition} \label{prop:lqr-infinite}
    Assume $H = \blkdiag(Q, R) \sgt 0$. Then, $\exists K \colon \rho(\op{E}_K) < 1$ implies
    \begin{enumerate}
        \item For all $P \sgeq 0$, $\lim_{t \to \infty} \op{R}^t(P) = P_{\star} \sgt 0$;
        \item $P_{\star}$ uniquely solves $\op{R}(P_{\star}) = P_{\star}$;
        \item For $K_{\star} = -(R + \adj{\op{G}}(P_{\star}))^{-1} \adj{\op{H}}(P_{\star})$,  $\rho(\op{E}(K_{\star})) < 1$.
    \end{enumerate}
\end{proposition}
\begin{proof}
    The proof follows that of \cite[Prop.~4.4.1]{Bertsekas2005V1}. 

    \paragraph*{\emph{(i)} and $P = 0$} Let $K$ be a controller such that $\rho(\op{E}_K) < 1$. 
    We will prove that $\lim_{k \to \infty} \op{R}^k(P) = P_{\star}$ is well defined by applying \cref{lem:monotone-convergence} for which we should show 
    for all $X$ that 
    \begin{align*}
        \tr[\op{R}^k(0) X] &\leq \tr[\op{R}^{k+1}(0) X], \, \forall k \in \N \text{ and } \\
        \lim_{k \to \infty} \tr[\op{R}^k(0) X] &< +\infty.
    \end{align*}
    The first inequality follows by induction. To start the induction note that $\op{R}(0) = \adj{\Pi}_{K_{0}}(H + \adj{\op{E}}(0)) = \adj{\Pi}_{K_{0}}(H) \sgt 0$,
    where the first equality follows from \cref{eq:ric-alt}. Then assume $\op{R}^{k+1}(0) \sgeq \op{R}^{k}(0)$. Taking $\op{R}$ on both sides and then
    using \cref{cor:ric-prop:a} then proves the induction step. To show that $\lim_{k \to \infty} \tr[\op{R}^{k}(0) X] < +\infty$ note that 
    $\tr[\op{R}^{k}(0) X] = \op{J}^k(0; X)$, which is upper bounded by $\sum_{t=0}^{k} \tr[\op{E}_K^t(X) \adj{\Pi}_K(H)]$. 
    By \cref{prop:lyapcp}, the limit of this sum for $k \to \infty$ is bounded. Therefore $\lim_{k \to \infty} \tr[\op{R}^{k}(0) X] < +\infty$. 
    From \cref{lem:monotone-convergence} we thus get $\lim_{k \to \infty} \op{R}^k(P) = P_{\star}$. Finally $P_\star \sgt 0$ 
    follows from $\op{R}(0) \sgt 0$.

    \paragraph*{\emph{(ii)}} Clearly $P_{\infty}$ derived in the previous step satisfies $\op{R}(P_{\infty}) = P_{\infty}$. 
    This equality can be rewritten using \cref{eq:ric-alt} as:
    \begin{equation} \label{eq:ric-lyap}
        P_\infty - \adj{\op{E}_{K_\star}}(P_{\infty}) = \adj{\op{E}_{K_\star}}(H) \sgt 0. 
    \end{equation}
    This is a Lyapunov equation. Therefore, by \cref{prop:lyapcp}, $P_\infty$ is the unique solution to $\op{R}(P_{\infty}) = P_{\infty}$. 

    \paragraph*{\emph{(iii)}} Stability follows directly from \cref{eq:ric-lyap} and \cref{prop:lyapcp}. 

    \paragraph*{\emph{(i)} for $P \neq 0$} From the definition of $\op{J}^N$ we get 
    \begin{align} \label{eq:squeezed-costs}
        &\op{J}^N(0; X_0) \leq \op{J}^N(P; X_0) \nonumber \\
                         &\quad\leq \sum_{t=0}^{N-1} \tr[\op{E}_{K_\star}^t(X) \adj{\Pi}_{K_\star}(H)] + \tr[\op{E}_{K_\star}^N(X_0) P], 
    \end{align}
    where the first inequality follows from $P \nsgeq 0$, so the cost that is minimized in $\op{J}^N(P; X_0)$ is larger than 
    the one of $\op{J}^N(0; X_0)$. The second inequality follows from the fact that $Z_t = \Pi_{K_{\star}}(X_t) = \Pi_{K_\star}(\op{E}_{K_\star}^t(X))$ is feasible 
    for the problem in $\op{J}^N(P; X_0)$. From \cref{eq:ric-lyap} and \cref{prop:lyapcp} we get that the right-hand side of \cref{eq:squeezed-costs} converges to $\tr[X_0 P_{\star}]$. 
    This holds trivially for the sum. For the final term we use the fact that $\rho(\op{E}_{K_\star}) < 1$. 
    Similarly, from \cref{prop:lqr-finite}, $\lim_{N\to \infty} \op{J}^N(0; X_0) = \tr[X \lim_{t\to\infty} \op{R}^t(0)] = \tr[X_0 P_\star]$
    and $\lim_{N\to \infty} \op{J}^N(P; X_0) = \tr[X \lim_{t\to\infty} \op{R}^t(P)]$. Thus, by the sandwich theorem, $\tr[X (\op{R}^\infty(P) - \op{R}^\infty(0))] = 0$ for 
    any $X \in \sym{d}$. Note that $\op{R}^\infty(P) \sgeq \op{R}^\infty(0)$ since $P \sgeq 0$ and by repeated application of \cref{cor:ric-prop:a}. 
    We proceed by contradiction, assuming that $\op{R}^\infty(P) \nsgeq \op{R}^\infty(0)$. Then, taking $X \sgt 0$, \cref{lem:trivialinequality} implies $\tr[X (\op{R}^\infty(P) - \op{R}^\infty(0))] > 0$. 
    This is a contradiction from what we showed earlier. Thus $\op{R}^\infty(P) = \op{R}^\infty(0)$.
\end{proof}
\end{prepupdate}

We are now ready to prove the main LQR result:
\paragraph*{Proof of \cref{thm:cplqr}}
Stability follows directly from \cref{prop:lqr-infinite}.
Taking $N \to \infty$ in \cref{prop:lqr-finite} shows (ii). Using \eqref{eq:ric-lyap}
and \cref{prop:lyapcp} gives $\tr[P_{\star} X_0] = \tr[H \ssum_{t=0}^{\infty} \op{E}^t_{K_{\star}}(X_0)]$.
That is $Z_t = \Pi_{K_{\star}}(X_t)$ achieves the optimal cost. 
\qed{}

We conclude this section with the derivation of the SDP reformulations of the Riccati equation \eqref{eq:ric-def}.

\begin{prepupdate}
\paragraph{Proof of \cref{thm:ric-sdp-primal}}
We follow \cite{Balakrishnan2003}. 
Introducing $\op{L}(Z) = X - \op{E}(Z)$ as in \cref{lem:opadj} and its adjoint 
$\adj{\op{L}}$. It becomes clear from \cref{lem:opadj} that the SDP in \cref{thm:ric-sdp-primal} is equivalent to 
\begin{equation*}
    \minimize \, \{\tr[P (-X_0)] \colon \adj{\op{L}}(P) \sleq H, P \sgeq 0\}. 
\end{equation*}
The Lagrangian dual of this SDP is \cite[Eq.~9]{Balakrishnan2003}
\begin{equation} \label{eq:dual}
    \begin{alignedat}{2}
        &\maximize_{Z} & \quad & - \tr[H Z] \\
        &\stt && Z_{22} - \op{L}(Z_{11}) = - X_0 \\
        &&& Z = \begin{bmatrix}
            Z_{11} & Z_{12} \\ Z_{21} & Z_{22}
        \end{bmatrix} \sgeq 0.
    \end{alignedat}
\end{equation}

The proof proceeds as follows: first strong duality is established; then feasibility is examined, proving \emph{(i)}; next optimality conditions are given for $X_0 \sgeq 0$; and 
finally the uniqueness of the solution is established for $X_0 \sgt 0$, proving \emph{(ii)}. 

\paragraph*{Strong duality} Note that (i) $\exists Z \colon Z \nsgeq 0$, $Z_{22} - \op{L}(Z_11) = 0$ and $\tr[HZ_{11}] \sleq 0$,
with $Z$ partitioned as in \cref{eq:dual}; and (ii) $\exists P \colon P \sgt 0$ and $\adj{\op{L}}(P) \slt H$ are strong alternatives \cite[Thm.~1]{Balakrishnan2003}. 
We can leverage \cref{lem:trivialinequality} and $H \sgt 0$ to show that (i) holds only if $Z_{11} = 0$. Since $\op{L}$ is linear, this implies $Z_{22} = 0$. 
Therefore $Z = 0$ as well, which means that (i) can never hold. Therefore (ii) must hold and the primal problem is strictly feasible. Strong duality then follows by \cite[Thm.~4]{Balakrishnan2003}.

\paragraph*{Feasibility} Note that \cref{eq:dual} is feasible iff there is some $Z_{11} \sgeq 0$ such that 
$\op{L}(Z_{11}) \sgeq X_0$. When the dynamics are stabilizable, then there is some $K$ such that $\rho(\op{E}_K) = \rho(\adj{\op{E}}_K) < 1$. 
Hence there exists an $X \sgt 0$ such that $\op{L}(\Pi_K(X)) = X - \op{E}_{K}(X) = X_0$ by \cref{prop:lyapcp}, implying feasibility of $Z_{11} = \Pi_K(X)$. 
Therefore, by \cite[\S{}II.B]{Balakrishnan2003}, the primal problem is bounded.

Conversely by \cref{prop:lyapcp}, there is no $K, X$ such that $\op{L}(\Pi_K(X)) = X - \op{E}_{K}(X) = X_0 + Z_{22}$ for any $Z_{22} \sgeq 0$. Since, by \cref{lem:z-partition}, any $Z_{11} \sgeq 0$ 
can be written as $Z_{11} = \Pi_K(X) + \blkdiag(\Delta, 0)$ for $\Delta \sgeq 0$ and $\op{L}(\Pi_K(X) + \blkdiag(\Delta, 0)) = \op{L}(\Pi_K(X)) - \op{E}(\blkdiag(0, \Delta)) \sleq \op{L}(\Pi_K(X))$. Thus there is no $Z_{11}$
such that $\op{L}(Z_{11}) = X_0 + Z_{22}$ and \cref{eq:dual} is infeasible. This implies the primal problem is unbounded by \cite[\S{}II.B]{Balakrishnan2003}.

\paragraph*{Optimality} From earlier, stabilizability results in feasibility of the dual. 
Note that $Z_{22} = \op{L}(Z_{11}) - X_0$ when $Z$ is feasible. So we only consider constraints on
$Z_{11}$. By \cref{lem:z-partition}, we can partition $Z_{11} \sgeq 0$ as
$Z_{11} = \Pi_K(X) + \blkdiag(0, \Delta)$
with $X \sgeq 0$ and $\Delta \sgeq 0$. By complementary slack \cite[\S{}II.C]{Balakrishnan2003}:
\begin{align*}
    &\begin{bmatrix}
        I \\ K
    \end{bmatrix} X \begin{bmatrix}
        I & \trans{K}
    \end{bmatrix} \begin{bmatrix}
        Q - P + \adj{F}(P) & \trans{(\adj{\op{H}})}(P) \\ \adj{\op{H}}(P) & R + \adj{\op{G}}(P) 
    \end{bmatrix}\\
    & \quad + (R + \adj{\op{G}}(P)) \Delta = 0
\end{align*}
for the optimum. Since $R \sgt 0$, the second term is zero iff $\Delta = 0$. Thus $Z_{11}$ solves the dual iff $Z_{11} = \Pi_K(X)$. 

A candidate solution $(Z_{11}, P) = (\Pi_K(X), P)$ with $X \sgt 0$ is then optimal iff 
\begin{equation} \label{eq:complementary-slack}
    \begin{bmatrix}
        I&K
    \end{bmatrix} \begin{bmatrix}
        Q - P + \adj{F}(P) & \trans{(\adj{\op{H}})}(P) \\ \adj{\op{H}}(P) & R + \adj{\op{G}}(P)
    \end{bmatrix} = 0,
\end{equation}
which holds iff $P = \op{R}(P)$ (i.e. $P = P_{\star}$) and $K = K_{\star}$. So $P_{\star}$ is an optimal solution to the SDP, when it is bounded (i.e. the dual problem is feasible). 

\paragraph*{Uniqueness} Assume $X_0 \sgt 0$. As shown before, $Z_{11} = \Pi_K(X)$ for some $X$. From feasibility we require that 
$X - \op{E}_K(X) \sgeq X_0 \sgt 0$. So $X \sgt 0$ holds. Therefore any pair of solutions $(Z_{11}, P)$ with $Z_{11} = \Pi_K(X)$ must satisfy \cref{eq:complementary-slack}
and, by earlier arguments $P = \op{R}(P)$ and $K = K_{\star}$.
\qed{}
\end{prepupdate}

Next we prove the properties of \eqref{eq:sdp-stab}.

\paragraph*{Proof of \cref{thm:sdp-stab}} 
Note that, by $H \sgt 0$, $\adj{\Pi}_K(H) = Q + \trans{K} R K \sgt 0$. Hence any feasible pair $(P, K)$ satisfies a Lyapunov
equation for some $Q' \sgeq \adj{\Pi}_K(H)$. From \cref{lem:stabcp} stability then follows and
\begin{align*}
    \forall X \colon \tr[PX]    &= \tr[\ssum_{t=0}^{\infty} \op{E}^t_K(X) Q']  \\
                                &\geq \tr[\ssum_{t=0}^{\infty} \op{E}^t_K(X) \adj{\Pi}_K(H)] \geq \op{J}^\star(X_0).
\end{align*}

From \eqref{eq:ric-lyap}, $(P_{\star}, K_{\star})$ are feasible for \eqref{eq:sdp-stab} and satisfy $\tr[P_{\star} X_0] = \op{J}^{\star}(X_0)$
by \cref{prop:lqr-infinite}. Therefore this pair achieves the lower bound and is optimal. 

Assume $X_0 \sgt 0$ and that there is a feasible $(P', K') \neq (P_{\star}, K_{\star}) \colon \tr[P' X_0] = \tr[P_{\star} X_0]$. 
By feasiblity and \cref{cor:ric-prop:b}: $P' \sgeq \adj{\Pi}_{K'}(H + \adj{\op{E}}(P')) \sgeq \op{R}(P')$. So,
by repeatedly applying \cref{cor:ric-prop:a} we have $P' \sgeq \op{R}(P') \sgeq \op{R}^2(P) \sgeq \dots$.
By induction we then have $P' \sgeq \lim_{k \to \infty} \op{R}^k(P') = P_{\star}$. Hence, by $X_0 \sgt 0$,
$P' \neq P_{\star}$ and \cref{lem:trivialinequality} we have $\tr[P' X_0] > \tr[P_{\star} X_0]$ which is a contradiction
since we assumed $\tr[P'X_0] = \tr[P_\star X_0]$ before. 
Thus $P' = \adj{\Pi}_{K'}(H + \adj{\op{E}}(P')) = \op{R}(P')$, which by \cref{cor:ric-prop:b} and 
\cref{prop:lqr-infinite} holds iff $(P', K') = (P_{\star}, K_{\star})$.

\end{arxiv}

\section{Vector and matrix concentration}
\begin{lemma} \label{lem:mathfd}
    For some $\gamma \in \Re^{\nsample}$ and a random, independent sequence $\{X_i\}_{i\in \N_{1:\nsample}} \subset \sym{d}$ with $\nrm{X_i}_2 \leq [\gamma]_i$ a.s. 
    and $\E[X_i] = 0$. 
    Then, for $\beta \geq 0$, $Y = \ssum_{i=1}^{\nsample} X_i$, 
    \begin{align*}
        \prob\left[\lambda_{\update*{\mathrm{max}}}{\left( Y \right)} \geq \beta\right] &\leq \delta; &
        \prob\left[\lambda_{\update*{\mathrm{min}}}{\left( Y \right)} \leq -\beta\right] &\leq \delta,
    \end{align*}
    for $\delta = d \exp({-\beta^2}/{2\nrm{\gamma}_2^2 })$. 
    Moreover, $\prob\left[\nrm{Y}_2 \geq \beta\right] \leq 2\delta$.
\end{lemma}
\begin{proof}
    We extend the proof of \cite[Lem~II.2]{Coppens2021}. Let $X_i = V_i \Lambda_i \trans{V_i}$ be the eigenvalue decomposition. 
    Hence
    \begin{align*}
        &\E[\exp(\theta X_i)] = \E[V_i \exp(\theta \Lambda_i) \trans{V_i}] \\
        &\quad \sleq \E\left[\frac{[\gamma]_i I + X_i}{2[\gamma]_i} e^{\theta [\gamma]_i} + \frac{[\gamma]_i I - X_i}{2[\gamma]_i}  e^{-\theta [\gamma]_i}\right]
    \end{align*}
    by convexity. Also from $\E[X_i] = 0$ we have, $\E[\exp(\theta X_i)] \sleq \cosh(\theta [\gamma]_i) I \sleq \exp(\theta^2 [\gamma]_i^2/2) I$. 

    Applying \cite[Thm.~3.6.1]{Tropp2015} results in
    \begin{align*}
        \prob[\lambda_{\update*{\mathrm{max}}}(Y) \geq \beta] &\leq \inf_{\theta > 0} d \exp(\theta \nrm{\gamma}_2^2 /2 - \theta \beta) \\
        \prob[\lambda_{\update*{\mathrm{min}}}(Y) \leq -\beta] &\leq \inf_{\theta < 0} d \exp(\theta \nrm{\gamma}_2^2 /2 - \theta \beta).
    \end{align*}
    Minimizing over $\theta$ gives the required result for $\lambda_1$ and $\lambda_d$. Taking a union bound gives the spectral norm bound. 
\end{proof}

\begin{lemma} \label{lem:vechfd}
    Let  $\gamma \in \Re^{\nsample}$ and consider a random, independent sequence $\{x_i\}_{i \in \N_{1:\nsample}} \subset \Re^d$
    with $\nrm{x_i} \leq [\gamma]_i$ a.s. Then, for $\beta \geq 0$, $y = \ssum_{i=1}^N (x_i - \E[x_i])$,
    \begin{equation*}
        \prob[\nrm{y} \geq 2 \nrm{\gamma}_2 + \beta] \leq \delta,
    \end{equation*}
    for $\delta = \exp(-\beta^2/2 \nrm{\gamma}_2^2)$. 
\end{lemma}
\begin{proof}
    We use the techniques of \cite[\S6]{Ledoux1991}. Introducing a filtration, let $\F_i$ the $\sigma$-algebra generated by $\{x_i\}_{i\in \N_{1:\nsample}}$ with $\F_0$ the trivial algebra.
    Let $\E_i[x] = \E[x \mid \F_i]$, $\Delta \E_i[x] = \E_i[x] - \E_{i-1}[x]$ and $d_i = \Delta E_{i}[\nrm{y}_2]$. 
    Then, by telescoping over $i$, we have $\ssum_{i=1}^{\nsample} d_i = \nrm{y} - \E[\nrm{y}]$. Noting that $y - x_i$ 
    does not depend on $x_i$ and letting $\delta_i = \nrm{y}_2 - \nrm{y - x_i}_2$ gives
    \begin{equation*}
        \Delta \E_i[\delta_i] = d_i - \E_{i}[\nrm{y - x_i}_2] + \E_{i-1}[\nrm{y - x_i}_2] = d_i. 
    \end{equation*}
    Then, by the triangle inequality, $|\delta_i| = |\nrm{y}_2 - \nrm{y - x_i}_2| \leq \nrm{y - y + x_i}_2 = [\gamma]_i$. 
    Letting $e_i \dfn \E_{i-1}[\delta_i]$ then gives $|d_i - e_i| = |\E_{i} \delta_i| \leq [\gamma]_i$. 

    Hence we have constructed a sequence of zero-mean, independent random variables $\{d_i\}_{i\in\N_{1:\nsample}}$ that take their value a.s. in $[-[\gamma]_i + e_i, [\gamma]_i + e_i]$. 
    Applying a classical Hoeffding bound \cite[Thm.~2.8]{Boucheron2013} gives
    \begin{equation} \label{eq:meanbnd1}
        \prob[\ssum_{i=1}^{\nsample} d_i \geq \beta] \leq \exp(-\beta^2/2\ssum_{i=1}^\nsample \nrm{\gamma}_2^2),
    \end{equation}
    where, as shown before, $\ssum_{i=1}^N d_i = \nrm{y}_2 - \E[\nrm{y}_2]$. We introduce a sequence of independent random vectors $\{v_i\}_{i \in \N_{1:\nsample}}$
    which is identically distributed to $\{x_i\}_{i\in \N_{1:\nsample}}$. Let $\E_{v}[x]$ denote the expectation conditioned on $\{x_i\}_{i\in \N_{1:\nsample}}$. 
    Then, by Jensen's inequality, we can bound $\E[\nrm{y}_2]$ as
    \begin{align*}
        \E[\nrm{\ssum_{i=1}^{\nsample} x_i - \E[x_i]}_2] &= \E[\nrm{\E_{v}[\ssum_{i=1}^\nsample x_i - v_i]}_2] \\
                                                       &\leq \E[\nrm{\ssum_{i=1}^\nsample (x_i - v_i)}_2].
    \end{align*}
    Using a classical symmetrization argument, introducing independent Rademacher random variables $\{\epsilon_i\}_{i\in\N_{1:\nsample}}$,
    gives $\E[\nrm{\ssum_{i=1}^\nsample (x_i - v_i)}_2] \leq 2\E[\nrm{\ssum_{i=1}^{\nsample} x_i \epsilon_i}_2]$. Applying Jensen once more,
    followed by a triangle inequality gives
    \begin{align*}
        \E[\nrm{y}_2]   &= 2\E[\nrm{\ssum_{i=1}^{\nsample} x_i \epsilon_i}_2] \leq 2\sqrt{\E[\nrm{\ssum_{i=1}^\nsample x_i \epsilon_i}_2^2]} \\
                        &\leq 2\sqrt{\E[\ssum_{i=1}^{\nsample} \nrm{x_i}_2^2]} \leq 2 \nrm{\gamma}_2.
    \end{align*}
    Plugging into \eqref{eq:meanbnd1} concludes the proof. 
\end{proof}

\begin{arxiv}
\section{Rank of random matrices} \label{app:rank}
We introduce a fundamental lemma, regarding the span of random i.i.d. vectors \cite[Thm.~3.3]{Eaton1973}:
\begin{lemma} \label{lem:rank-condition}
    Given a sequence of i.i.d. copies $\{z_1, \dots, z_N\}$ of a random vector $z \colon \Omega \to \Re^d$ with $N \geq d$, the following statements are equivalent:
    \begin{enumerateq}
        \item $\forall w \in \Re^d$, $w \neq 0$, $\prob[\trans{w} z = 0] = 0$;
        \item $\prob[\rk([z_1, \dots, z_N]) = d] = 1$.
    \end{enumerateq}
    Moreover if either (a) or (b) holds, then $\E[z \trans{z}] \sgt 0$.
\end{lemma}
\begin{proof}
    The proof of \cite[Thm.~3.3]{Eaton1973} stands in a more general setting then necessary. We provide a simplified proof 
    for completeness. Without loss of generality we assume $N = d$. Let $Z_k \dfn [z_1, \dots, z_k]$ and $\set{B}[k]$ denote the 
    event $\rk[Z_k] = k$. 
    \paragraph*{(a) $\Leftarrow$ (b)} 
    We proceed by contradiction, so assume $\exists w \neq 0\colon \prob[\trans{w} z = 0] > 0$
    and that (b) holds (i.e. $\prob[\set{B}[d]] = 1$). Note that $\prob[\neg \set{B}[d]]=
    \prob[\exists v \neq 0 \colon \trans{v} Z_d = 0]$ (i.e., the kernel of $\trans{Z_d}$ is nontrivial). 
    We have,
    \begin{align*}
        &\prob[\exists v \neq 0 \colon \trans{v} Z_d = 0] \\
            &\quad \geq \prob[\trans{w} Z_d = 0]  = \Pi_{i=1}^d \prob[\trans{w}z_i = 0] > 0,
    \end{align*}
    where the first inequality follows by fixing $v = w$, the equality follows by independence and the final inequality follows by 
    the assumption that (a) does not hold. So $\prob[\set{B}[d]] \neq 1$ and we have shown (a) $\Leftarrow$ (b) by contradiction.
    \paragraph*{(a) $\Rightarrow$ (b)}
    We can decompose $\prob[\set{B}[d]]$ into
    \begin{align*}
        \prob[z_d \notin \img(Z_{d-1}) \mid \set{B}[d-1]] \cdot \prob[\set{B}[d-1]]
    \end{align*}
    We know that $\exists w \in \ker(\trans{Z}_{d-1})/\{0\}$ since $Z_{d-1}$ is not full rank. Hence the first factor is equal to 
    $\prob[\trans{w} z_d \neq 0 \mid \rk(Z_{d-1}) = d-1]$, where $w$ depends on $Z_{d-1}$. This factor equals $1$ since (a) holds\todoright{if this is iff. we can shorten the proof}. 
    We can repeatedly apply the same decomposition to the second factor until we arrive at $\prob[\rk(z_1) = 1]$,
    which equals to $1$ since (a) also implies $\prob[z_1 = 0] = 0$. Hence $\prob[\rk(Z_d) = d] = 1$ so (b) holds.

    Finally note that $\forall w$, $\trans{w} \E[z \trans{z}] w = \E{(\trans{w} z)^2}$. Since $\prob[\trans{w} z = 0] = 0$ by (a),
    $\E[(\trans{w} z )^2] > 0$. So $\E[z \trans{z}] \sgt 0$, completing the proof. 
\end{proof}

We can consider some examples of distributions satisfying (a). A clear sufficient condition is that the density of $z$ 
is dominated by the Lebesgue measure. 
Another example are uniform distributions over spheres. Meanwhile, any atomic measure will not satisfy (a). 


\begin{corollary} \label{cor:rank-kron-condition}
    With the setup of \cref{lem:rank-condition}, assume that $z$ is dominated by 
    the Lebesgue measure. Then, if $N \geq \sd{d}$,
    \begin{equation*}
        \prob[\rk([z_1 \skron z_1, \dots, z_N \skron z_N]) = \sd{d}] = 1
    \end{equation*}
    and $\E[( z \trans{z}) \skron ( z \trans{z})] \sgt 0$.
\end{corollary}
\begin{proof}
    Applying \cref{lem:rank-condition} to $z \kron z$ instead of $z$ gives:
    \begin{equation} \label{eq:exact-kron-rank}
        \prob[\trans{w}(z \skron z) = 0] = 0, \quad \forall w \in \Re^{\sd{d}}, w \neq 0.
    \end{equation}
    The condition $\trans{w}(z \skron z) = 0 \Leftrightarrow \trans{z} W z$ describes the roots of a quadratic polynomial,
    with $W = \unsvec(w) \in \sym{d}$. It is well known\todo{cite} that 
    such a set is of Lebesgue measure zero, if the polynomial is not identically zero. 
    Since $W \in \sym{d}$, $\trans{z}W z = 0$ for all $z$ iff $W = 0$. Hence, by our domination 
    assumption, $\prob[\trans{w}(z \skron z) = 0] = 0$ follows. Applying \cref{lem:rank-condition} and using $(z \kron z) \trans{(z \kron z)} = (z\trans{z}) \kron (z \trans{z})$
    completes the proof. We can apply similar reasoning to prove the second result. 
\end{proof}
\begin{remark}
    We could have directly used \eqref{eq:exact-kron-rank} as the assumption on $z$ to get a tight result. We 
    avoided doing so, since \eqref{eq:exact-kron-rank} is more difficult to verify in practice. 
\end{remark}
\end{arxiv}

\section{Identification model} \label{app:identification}
We begin by showing some auxiliary results. 
\begin{lemma} \label{lem:kernel-zn}
    For $\opm{Z}_N$ as in \eqref{eq:zmatdef} and $N \geq \sd{n_z}$, 
    \begin{equation*}
        \pinv{\opm{Z}_N} \opm{Z}_N = \pinv{(\trans{\ten{W}_{(3)}})} \trans{\ten{W}_{(3)}},
    \end{equation*}
    with probability one over data $(z_i)_{i=1}^N$ satisfying \cref{asm:data}. 
\end{lemma}
\begin{proof}
    We can rewrite \eqref{eq:zmatdef} as
    \begin{equation} \label{eq:zmat-stacked}
        \opm{Z}_N = (\trans{[z_1 \skron z_1, \dots, z_N \skron z_N]} \kron I_{\sd{n_x}}) \trans{\ten{W}_{(3)}},
    \end{equation}
    by $[\ten{T} \btimes{2} x; \ten{T} \btimes{2} y] = \ten{T} \ttimes{2} \trans{[x, y]}$ (cf. \cref{eq:modendef}) and \cref{prop:kronunfold}.

    Let $Z_N \dfn \trans{[z_1 \skron z_1, \dots, z_N \skron z_N]}$. Then by $N \geq \sd{n_z}$\ilpub{ and }\ilarxiv{, }
    \cref{asm:data} \ilarxiv{and \cref{cor:rank-kron-condition}}, $\rk(Z_N) = \sd{n_z}$ w.p. $1$.
    Hence $Z_N \kron I$ is also left-invertible \cite[Thm.~4.2.15]{Horn1991} w.p. $1$. 

    \shorten*{For two matrices $U$, $V$ of conformable dimensions, with $\pinv{U} U = I$, we mimic
    the proof of \cite[Thm.~1]{Greville1966}.}
    \begin{align*}
        \pinv{V} V  &= \pinv{V} \pinv{U} U V \labelrel={step:moorp} \pinv{V} \pinv{U} (UV) \trans{(UV)} \pinv{(\trans{(UV)})} \\
                    &= \pinv{V} V \trans{V} \trans{U} \pinv{(\trans{(UV)})} \\
                    &\labelrel={step:moorpb}  \trans{V} \trans{U} \pinv{(\trans{(UV)})} = \pinv{(UV)} (UV),
    \end{align*}
    where both \ref{step:moorp} and \ref{step:moorpb} use $\pinv{A} A \trans{A} = \trans{A}$ (cf. \cite[Eq.~2]{Greville1966}).

    Applying this to $\opm{Z}_N$ then gives the required result. 
\end{proof}

\begin{lemma} \label{lem:kernel-sm}
    Take $\svec(\widetilde{W}) = (I - \pinv{(\trans{\ten{W}_{(3)}})} \trans{\ten{W}_{(3)}}) \svec(W)$ for any $W \in \sym{n_w}$.
    Consider $\op{E}$ as in \cref{eq:dynsm_}. Then
    \begin{equation*}
        \op{E}({W} + \widetilde{W}; Z) = \op{E}({W}; Z),
    \end{equation*}
    for all $W \in \sym{n_w}$ and $Z \in \sym{n_z}$. 
\end{lemma}
\begin{proof}
    From \ilarxiv{\cref{prop:cpunfold:c}}\ilpub{\cite[Prop.~A.1(iii)]{Arxiv}} and \cref{prop:kronunfold} we have
    \begin{align*}
        \tr[P \op{E}(\widetilde{W}; Z)] &= \tucker{\ten{M} \skron \ten{M}; \svec(P), \svec(Z), \svec(\widetilde{W})} \\
                            &= (\svec(P) \kron \svec(Z)) \trans{\ten{W}_{(3)}} \svec(\widetilde{W}),
    \end{align*}
    for all $P \in \sym{n_x}$ and $Z \in \sym{n_z}$. By assumption, and $A \pinv{A} A = A$,
    \begin{equation*}
        \trans{\ten{W}_{(3)}} \svec(\widetilde{W}) = \trans{\ten{W}_{(3)}}(I - \pinv{(\trans{\ten{W}_{(3)}})} \trans{\ten{W}_{(3)}}) \svec(W) = 0.
    \end{equation*}
    Therefore $\tr[P \op{E}(\widetilde{W}; Z)]  = 0$ for all $P$ and $Z$, which is only possible if $\op{E}(\widetilde{W}; Z) = 0$
    for all $Z \in \sym{n_z}$ (otherwise we can take $P = \unvec(e_i)$, with $e_i$ the canonical basis vector
    and $i$ the index of the nonzero element in $\vec(\op{E}(\widetilde{W}; Z))$).
\end{proof}

Next, we derive properties of $\op{W}$ and $\op{H}_i$ defined in \eqref{eq:hopmat}.
\begin{lemma} \label{lem:wop}
    Let $\op{W}(Z) \dfn \ten{W}_{(3)} ( (Z \skron Z) \kron I_{\sd{n_x}}) \trans{\ten{W}_{(3)}}$.
    Then
    \begin{enumerate}
        \item \label{lem:wop:a} $\op{W}(z \trans{z}) = \trans{(\ten{W} \btimes{2} (z \skron z))} (\ten{W} \btimes{2} (z \skron z))$;
        \item \label{lem:wop:b} $\ssum_{i=1}^N \op{W}(z_i \trans{z_i}) = \trans{\opm{Z}_N} \opm{Z}_N$;
        \item \label{lem:wop:c} $\nrm{\op{W}}_2 = \sup_{z} \{\nrm{\op{W}(z \trans{z})}_2 \colon \nrm{z}_2 \leq 1\} \leq \nrm{\ten{W}}_2$,
    \end{enumerate}
    with $\opm{Z}_N$ as in \cref{eq:zmatdef}.
\end{lemma}
\begin{proof}
    The first result (i) follows by using \cref{prop:kronunfold} to write
    \begin{align*}
        &\trans{(\ten{W} \btimes{2} (z \skron z))} (\ten{W} \btimes{2} (z \skron z)) \\
        &\quad = \ten{W}_{(3)} [(z \skron z) \kron I] [\trans{(z \skron z)} \kron I] \trans{\ten{W}}_{(3)}.
    \end{align*}
    Using $(A \kron X)(B \kron Y) = AB \kron XY$ \cite[Thm.~E.1.3]{DeKlerk2002} and \cite[Lem.~E.1.2]{DeKlerk2002} to argue
    $(z \skron z) \trans{(z \skron z)} = (z \trans{z} \skron z\trans{z})$,
    proves (i). We can show (ii) by using \eqref{eq:zmat-stacked} and applying similar tricks. Finally (iii)
    is shown by using (i) to argue
    \begin{align}
        \nrm{\op{W}}_2  &= \sup_{z \in \ball{}} \{\nrm{(\ten{W} \btimes{2} (z \skron z))}_2^2\} \nonumber \\
                        &= \sup_{x, w, z \in \ball{}} \{ (\trans{x} (\ten{W} \btimes{2} (z \skron z)) w)^2\}, \label{eq:vartnorm} 
    \end{align}
    where we used the variational representation of the spectral norm for the second equality
    and with $x \in \Re^{\sd{n_x}}$, $w \in \Re^{\sd{n_w}}$ and $\ball{}$ the unit Euclidean ball of generic dimension. The squared quantity
    can be rewritten by noting 
    \begin{align*}
        \trans{x} (\ten{W} \btimes{2} (z \skron z)) w = \tucker{\ten{W}; x, z \skron z, w}.
    \end{align*}
    Noting that $\trans{(z \skron z)} (z \skron z) = \nrm{z}_2^4$ implies that we can relax \cref{eq:vartnorm} to take 
    the supremum over $z' \in \ball{\sd{n_z}}$ instead of over $z \skron z$. So
    \begin{equation*}
        \nrm{\op{W}}_2 \leq \left(\sup_{x, w, z' \in \ball{}} \{ \tucker{\ten{W}; x, z', w}\}\right)^2,
    \end{equation*}
    which, plugging in the definition of the tensor spectral norm \cite{Chen2020a}, implies (iii). 
\end{proof}

A direct consequence of \cref{lem:wop} is that
\begin{corollary} \label{cor:error-model-wop}
    Given $\opm{Z}_N$, $E_N$ as in \cref{eq:zmatdef} and \eqref{eq:ematdef} respectively and $\op{W}$ 
    as in \cref{lem:wop}. Then
    \begin{align*}
        \pinv{\opm{Z}_N} E_N &= \pinv{[\ssum_{j=1}^N \op{W}(z_j \trans{z_j})]} ( \ssum_{i=1}^N \op{W}(z_i \trans{z_i})  \eta_i ) \\
                             &= \svec(\ssum_{i=1}^N \op{H}_i(w_i \trans{w_i} - W)),
    \end{align*}
    with $\op{H}_i$ the linear operator with matrix $\opm{H}_i$ as in \cref{eq:hopmat}. 
\end{corollary}
\begin{proof}
    Note that
    \begin{equation*}
        \pinv{\opm{Z}_N}E_N = \pinv{(\trans{\opm{Z}_N} \opm{Z}_N)} \trans{\opm{Z}_N}E_N.
    \end{equation*}
    Then, the first equality follows by application of \cref{lem:wop:b} to replace $(\trans{\opm{Z}_N} \opm{Z}_N)$
    and \cref{lem:wop:a} to show that $\trans{\opm{Z}_N} E_N = \sum_{i=1}^{N} \op{W}(z_i \trans{z_i}) \eta_i$. 
    The second equality follows by definition of $\opm{H}_i$ and $\eta_i = \svec(w_i \trans{w_i} - W)$. 
\end{proof}

We are now ready to prove the data-driven bound.
\paragraph{Proof of \cref{thm:error-bound-full}} 
By the classical LS error equation \cref{eq:error-model-full}:
\begin{equation*}
    \svec(\smomenth - \smoment) = (I - \pinv{\opm{Z}_N} \opm{Z}_N) \svec(\smoment) + \pinv{\opm{Z}}_N E_N,
\end{equation*}
where we need \cref{asm:model} to imply \eqref{eq:measurement-model} holds using \cref{lem:model-eq}. 
Note that, by \cref{cor:error-model-wop} with $\op{H}_i$ as in \cref{eq:hopmat}:
\begin{equation*}
    \unsvec(\pinv{\opm{Z}}_N E_N) = \ssum_{i=1}^N \op{H}_i(w_i \trans{w_i} - W).
\end{equation*}

Note that $-r_w^2 I \sleq -W \sleq w_i \trans{w_i} - W \sleq  w_i \trans{w_i} \sleq r_w^2 I$. 
Hence, by definition of $\nrm{\op{H}_i}_2$, the spectral norm of $X_i \dfn \op{H}_i(w_i \trans{w_i} - W)$
is bounded as $\nrm{X_i}_2 \leq \nrm{\op{H}_i}_2 r_w^2 \dfn [\gamma]_i$. The matrices $X_i$ 
are therefore bounded and i.i.d. (by \cref{asm:data}). Applying \cref{lem:mathfd}
and solving for $\beta$ shows
\begin{equation} \label{eq:stoch-nrm-bnd}
    \prob[\nrm{\ssum_{i=1}^N \op{H}_i(w_i \trans{w_i} - W)}_2 \leq \beta] \geq 1 - \delta.
\end{equation}

We have by \cref{asm:data} and \cref{lem:kernel-zn} that 
\begin{align*}
  (I - \pinv{\opm{Z}_N} \opm{Z}_N) \svec(\smoment) \nonumber = (I - \pinv{(\trans{\ten{W}_{(3)}})} \trans{\ten{W}_{(3)}}) \svec(W).
\end{align*}
Applying \cref{lem:kernel-sm} thus implies:
\begin{align*}
    \op{E}(\smomenth - \smoment; Z) &= \op{E}(\unsvec(\pinv{\opm{Z}}_N E_N); Z).
\end{align*}
The bound on $\nrm{\unsvec(\pinv{\opm{Z}}_N E_N)}_2 = \nrm{\ssum_{i=1}^N \op{H}_i(w_i \trans{w_i} - W)}_2$ in \eqref{eq:stoch-nrm-bnd} therefore proves the result. 
\qed{}


We can also derive the sample complexity for the simplified setting where $\nrm{z}_2 \leq r_z$. 

\paragraph{Proof of \cref{lem:moment-sample-complexity}}
    To deal with non-invertibility issues, we introduce $U \in \Re^{n_w \times d_{W}}$, a unitary matrix whose columns span $\img(\ten{W}_{(3)})$.
    For all $i \in \N_{1:\nsample}$, we consider the i.i.d. random matrices $Y_i = \trans{U} \op{W}(z_i \trans{z_i}) U$
    and $Y = \ssum_{i=1}^{\nsample} Y_i$. 
    
    By \cref{asm:data} and \cref{lem:kernel-zn}: $\img(\opm{Z}_N) = \img(\ten{W}_{(3)})$. 
    So $\lambda_{\update*{\mathrm{min}}}(Y) = \lambda_{\update*{\mathrm{min}}}(\trans{U} \trans{\opm{Z}_N} \opm{Z}_{N} U) = \lambda_{\update*{\mathrm{min}}}(\trans{\opm{Z}_N} \opm{Z}_{N})$. 
    Moreover $Y \sgt 0$. Similarly $\E[\op{W}(z \trans{z})] = \ten{W}_{(3)} (\E[(z \skron z) \trans{(z \skron z)}] \kron I) \trans{\ten{W}_{(3)}}$,
    with $\E[(z \skron z) \trans{(z \skron z)}] \sgt 0$ by \cref{asm:data}\ilarxiv{ and \cref{cor:rank-kron-condition}}. So $\lambda_{\update*{\mathrm{min}}}(\E[Y_i]) = \lambda_{\update*{\mathrm{min}}}(\E[\op{W}(z \trans{z})])$.
    
    Moreover, since $U$ is unitary, $\nrm{{Y}_i}_2 \leq r_z^4 \nrm{\op{W}}_2 \leq r_z^4 \nrm{\ten{W}}_2^2$ where we use 
    \cref{lem:wop:c} for the second inequality.

    Therefore the matrices $Y_i$ are i.i.d. and bounded as $\nrm{{Y}_i}_2 \leq r_z^4 \nrm{\ten{W}}_2^2$
    with $\lambda_{\update*{\mathrm{min}}}( \E[Y_i] ) = \lambda_{\update*{\mathrm{min}}}(\E[\op{W}(z \trans{z})]) = \gamma_W r^4_z$. 
    Therefore, \update*{\cref{lem:mathfd}} can be applied to show
    \begin{equation*}
        \prob[\lambda_{\update*{\mathrm{min}}}(Y - \E[Y]) \geq -\beta] \geq 1-\delta,
    \end{equation*}
    with $\beta = r_z^2 \nrm{\ten{W}}_2^2 \sqrt{N} \tau_W$ and $\tau_W =  \sqrt{2 \log(2d_W /\delta)}$. Hence, 
    \begin{equation*}
        \lambda_{\update*{\mathrm{min}}}(Y) \geq \lambda_{\update*{\mathrm{min}}}(\E[Y]) - \beta = N \gamma_W r^4_z - \beta.
    \end{equation*}
    So, since $\nrm{\pinv{Y}}_2 = \lambda_1(\pinv{Y}) = (\lambda_{\update*{\mathrm{min}}}(Y))^{-1}$,
    \begin{equation*}
        \nrm{\pinv{Y}}_2 \leq \frac{r_z^{-4}}{N\gamma_W -  \sqrt{N} \nrm{\ten{W}}_2^2 \tau_W}.
    \end{equation*}

    Using \cref{lem:sopbnd} we have 
    \begin{equation*}
        \zeta_W^2 = \ssum_{i=1}^N \nrm{\op{H}_i}^2_2 \leq N \max_i \nrm{\op{H}_i}^2_2,
    \end{equation*}
    where 
    \begin{align*}
        \nrm{\op{H}_i}_2^2  &\leq n_w \nrm{\pinv{Y} Y_i}_2^2 \nonumber \\
                            &\leq n_w \nrm{\pinv{Y}}_2^2 \nrm{Y_i}_2^2 \leq n_w (\nrm{\pinv{Y}}_2 r_z^4 \nrm{\ten{W}}_2^2)^2.
    \end{align*}
    Therefore $\nrm{\op{H}_i}_2^2 \leq \update*{n_w} (\nrm{\ten{W}}_2^{\update*{2}} / (N\gamma_W -  \sqrt{N} \nrm{\ten{W}}_2^2 \tau_W))^2$. 
    Multiplying by $N$ then gives us the claimed bound for $\update*{\zeta_W}$. 
    \qed{}

\begin{arxiv}
We continue with the proofs of \cref{sec:mean-id}. 
\paragraph*{Proof of \cref{lem:mean-dd}} By the classical LS error equation:
\begin{equation} \label{eq:err-model-mean-full}
    {\mean}_\star - \hat{{\mean}} = (I - \pinv{(\opm{Z}_N^\mu)} \opm{Z}_N^\mu)  {\mean}_\star + \pinv{(\opm{Z}_N^\mu)} E^\mu_N, 
\end{equation}
with $E^\mu_N = [\stucker{\widetilde{\ten{M}}; z_1, \epsilon_1}, \dots, \stucker{\widetilde{\ten{M}}; z_N, \epsilon_N}]$ and $\epsilon_i$
i.i.d. zero-mean random vectors. Equation \cref{eq:err-model-mean-full}
holds by \cref{asm:model-structured} and \cref{lem:model-eq-structured}. We can generalize \cref{lem:kernel-zn}, which requires \cref{asm:data}, to $\opm{Z}_N^\mu$ 
to show $\pinv{(\opm{Z}_N^\mu)} \opm{Z}_N^\mu = \pinv{(\trans{\widetilde{\ten{M}}}_{(3)})} \trans{\widetilde{\ten{M}}}_{(3)}$,
which equals the identity by \cref{asm:model} as $\widetilde{\ten{M}}_{(3)}$ is full rank. Expanding the second term of \cref{eq:err-model-mean-full},
inserting the definition of $G_i$ and $\op{M}$ gives \eqref{eq:error-model-mean}:
\begin{equation*}
    {\mean}_\star - \hat{{\mean}} = \ssum_{i=1}^N G_i \epsilon_i = \ssum_{i=1}^N \E[G_i \widetilde{w}_i] - G_i \widetilde{w}_i.
\end{equation*}
Note that $G_i \widetilde{w}_i$ is a sequence of i.i.d. random vectors with $\nrm{G_i \widetilde{w}_i}_2 \leq \nrm{G_i}_2 r_w$,
by the definition of the spectral norm. Hence \cref{lem:vechfd} is applicable. The final result is then 
recovered by solving for $\beta$. \qed{}

\paragraph*{Proof of \cref{thm:structured-ambigutiy}}
Again, let $\mu_\star$ and $\widetilde{W}_\star$ denote the true value of $\E[\widetilde{w}]$ and $\E[\widetilde{w} \trans{\widetilde{w}}]$ respectively
and $\meanh$ and $\hat{\widetilde{W}}$ their estimates.
Defining $\hat{\Sigma} \dfn \hat{\widetilde{W}} - \meanh \trans{\meanh}$ and $\Sigma_\star \dfn \widetilde{W}_\star - \mean_\star \trans{\mean_\star}$, 
allows us to expand $\op{E}_{\star}$ as:
\begin{align*}
    \op{E}_\star(Z) &= \op{E}\left( \begin{bmatrix}
        1 & \trans{\mu_\star} \\ \mu_\star & \Sigma_\star + \mean_\star \trans{\mean_\star}
    \end{bmatrix}, Z \right) \\
    &= \widetilde{\op{E}}(\Sigma_\star + \mean_\star \trans{\mean_\star}, Z)  
                + \op{E}\left( \begin{bmatrix} 
                    1 & \mu_\star \\ \trans{\mu_\star} & 0 
                \end{bmatrix}, Z \right),
\end{align*}
by linearity
and where $\widetilde{\op{E}}(\widetilde{W}; Z) \dfn \widetilde{\ten{M}}_{(3)} (\widetilde{W} \kron Z) \trans{\widetilde{\ten{M}}_{(3)}}$. 
Clearly $\widetilde{\op{E}}(\Sigma_\star + \mean_\star \trans{\mean_\star}; Z) = \widetilde{\op{E}}(\widetilde{W}_\star; Z)$. 

We reconsider \cref{eq:error-model-full}:
\begin{equation*}
    \svec(\widetilde{\smoment}_\star - \hat{\widetilde{W}}) = (I - \pinv{\opm{Z}_N} \opm{Z}_N) \svec(\widetilde{\smoment}_\star) + \pinv{\opm{Z}}_N E_N,
\end{equation*}
and label the first term $\Delta W_b$ and the second $\Delta W_\eta$. Then
\begin{equation*}
    \widetilde{\op{E}}(\widetilde{W}_\star; Z) = \widetilde{\op{E}}(\hat{\widetilde{W}}; Z) + \widetilde{\op{E}}(\Delta W_\eta; Z) + \widetilde{\op{E}}(\Delta W_b; Z).
\end{equation*}
By \cref{lem:kernel-sm} the final term is zero. So without loss of generality we assume $\Delta W_b = 0$.

So it is sufficient to show that 
\begin{equation} \label{eq:sigma-expand}
    \nrm{\Sigma_\star - \hat{\Sigma}}_2 = \nrm{\Delta \smoment_\eta - (\mean_\star \trans{\mean_\star} - \meanh \trans{\meanh}) }_2 \leq \beta_\Sigma
\end{equation}
and $\nrm{\mean_\star - \meanh}_2 \leq \beta_\mu$ with probability at least $1 - \delta_\mu - \delta_W$. The second 
bound follows directly by \cref{lem:mean-dd} (which claims (i) $\prob[\nrm{\mean_\star - \meanh}_2 \leq \beta_\mu] \geq 1 - \delta_\mu$), 
while we have (ii) $\prob[\nrm{\Delta W_\eta}_2 \leq \beta_W] \geq 1 - \delta_W$ by \cref{eq:stoch-nrm-bnd} in the proof of \cref{thm:error-bound-full}.
A union bound shows that both (i) and (ii) hold w.p. at least $ 1 - \delta_\mu - \delta_W$. 

We have, by a triangle inequality and \cref{eq:sigma-expand} that 
\begin{align*}
    \nrm{\Sigma_\star - \hat{\Sigma}}_2 &\leq \nrm{\Delta\smoment_\eta}_2 + \nrm{\mean_\star \trans{\mean_\star} - \meanh \trans{\meanh} }_2 \\
     &\leq \delta_W + \nrm{\mean_\star \trans{\mean_\star} - \meanh \trans{\meanh} }_2.
\end{align*}
We can rewrite $\mean_\star \trans{\mean_\star} - \meanh \trans{\meanh}$ as 
\begin{equation*}
    (\mean_\star - \meanh) \trans{(\mean_\star - \meanh)} + (\mean_\star - \meanh) \trans{\meanh} + \meanh \trans{(\mean_\star - \meanh)}.
\end{equation*}
So, by another triangle inequality 
\begin{equation*}
    \nrm{\Sigma_\star - \hat{\Sigma}}_2 \leq \beta_W + \beta_\mu (\beta_\mu + 2 \nrm{\meanh}_2) = \beta_\Sigma. 
\end{equation*}
We have therefore shown the claimed result. \qed{}
\end{arxiv}

\begin{arxiv}
\section{Sample Complexity} \label{app:sample-complexity}
\begin{prepupdate}
    We adjust the results of \cite{Coppens2020-TR} to our setting, proving \cref{thm:sample-complexity}. 
    This analysis proceeds in three steps, the first two of which are related to \cref{thm:drlqrcp} and the third evaluates 
    the closed-loop cost of $\bar{K}$. To be more specific, \cref{thm:drlqrcp} first finds $\bar{P}$ such that $\op{R}(\bar{W}; \bar{P}) = \bar{P}$
    and then computes $\bar{K} = -(R + \adj{\op{G}}(\bar{P}))^{-1} \adj{\op{H}}(\bar{P})$. Here $(P_\star, K_{\star})$ are the solutions to the nominal problem in \cref{thm:cplqr}. 
    We then compute the following bounds: 
        \emph{(i)} $\nrm{\bar{P} - P_{\star}}_2 \leq \epsilon_{P}$; 
        \emph{(ii)} $\nrm{\bar{K} - K_{\star}}_2 \leq \epsilon_K$;
        and \emph{(iii)} $\tr[\sum_{t=0}^{\infty} Z_t H] - \mathrm{Val}(\eqref{eq:lqrcp})$. 

    For simplicity we will ignore the bias estimation bias discussed in \cref{sec:error-analysis}. This is without 
    loss of generality due to \cref{asm:model}. Under this simplification the Riccati equations and optimal controllers of both the estimated system 
    (using $\bar{W}$) and the true system have the same structure, only differing in their value of $W$.

    The estimation error is then fully characterized by \cref{lem:moment-sample-complexity} and \cref{thm:error-bound-full}:
    \begin{equation} \label{eq:beta-w-def}
        \nrm{W_{\star} - \hat{W}}_2 \leq r_w^2 \frac{\sqrt{n_w} \nrm{\ten{W}}_2^2}{\sqrt{N} \gamma_W - \nrm{\ten{W}}_2^2 \tau_w} \sqrt{2 \log(2n_w / \delta)},
    \end{equation}
    with probability at least $1 - \delta$. 
    We label the right-hand side $\beta_W$. Note that $\bar{W} \dfn \hat{W} + \beta_W I_{n_w}$. Therefore $\nrm{\bar{W} - W_\star}_2 \leq 2 \beta_W$. 

    Throughout this section we heavily rely on the arguments in \cite{Coppens2020-TR}. Whenever steps therein are repeated, 
    more specific parts are referenced to aid the reader.
        
    \paragraph*{Riccati Perturbation Analysis}
    To bound $\nrm{\bar{P} - P_{\star}}$ we need to evaluate the effect of a change in $W$ on the solution to the Riccati equation $\op{R}(P) = P$, with $\op{R}$ as in \cref{eq:ric-def}.
    The dependency on $W$ is made explicit through \cref{rem:parameter}. We then find $\Delta P$ such that
    \begin{equation} \label{eq:ric-perturbed}
        \op{R}(W_\star + \Delta W; P_\star + \Delta P) = P_\star + \Delta P, \, \op{R}(W_\star; P_\star) = P_\star. 
    \end{equation}
    Here $W_{\star}$ is the true second moment $\E[w \trans{w}]$ as in \cref{sec:construction} and $\Delta W$
    is the difference between $W_{\star}$ and its data-driven estimate, i.e. $\bar{W}$ in \eqref{thm:error-bound-full} (or $\hat{W}$ in the certainty equivalent setting). 
    
    We can then state the error bound.
    \begin{lemma} \label{lem:riccati-bound}
        Given \cref{eq:beta-w-def}, where the right-hand side is labeled $\beta_W$, then $\nrm{\bar{P} - P_{\star}}$ is upper bounded by:
        \begin{equation*}
            \frac{
                \nrm{(\adj{\op{L}_\star})^{-1}}_2 -  \beta_W c_1 - \sqrt{ (\nrm{(\adj{\op{L}_\star})^{-1}}_2 - \beta_W c_1)^2 - 4 \beta_W c_1 c_2}
            }{
                2 c_2
            },
        \end{equation*}
        for $c_2 = \nrm{\ten{A}_{\star, (2)}}_2^2 \nrm{\ten{B}_{(2)}}_2^2 \nrm{W_{\star}}_2^2 / \lambda_{\mathrm{min}}(R)$ and $c_1 = 2\nrm{\ten{A}_{(2)}}^2$ and
        where we additionally assume that $\beta_W$ is sufficiently small such that the bound is non-negative and 
        smaller than $\lambda_{\mathrm{min}}({P_\star})$. 
    \end{lemma}
    \begin{proof}
        The perturbed Riccati equation \cref{eq:ric-perturbed} can be written as a fixed-point equation whenever both the original system 
        and the perturbed system are stabilizable: 
        \begin{equation} \label{eq:fixed-point-eq}
            \Phi(\Delta P) \dfn - (\adj{\op{L}_\star})^{-1}\left( \op{R}_{0}(\Delta P) + \op{R}_{\Delta}(\Delta P) \right) = \Delta P,
        \end{equation}
        with $\adj{\op{L}}_{K_{\star}}$ as in \cref{lem:opadj} and its inverse existing due to the stabilizability assumption and \cref{thm:cplqr}, $K_\star$ the optimal controller of \cref{thm:cplqr} and
        \begin{align*}
            \op{R}_0(\Delta P) &\dfn (P_\star + \Delta P) - \op{R}(W_{\star}; P_\star + \Delta P) - \adj{\op{L}}_{K_{\star}}(\Delta P), \\
            \op{R}_\Delta(\Delta P) &\dfn \op{R}(W_{\star}; P_\star + \Delta P) - \op{R}(W_{\star} + \Delta W; P_\star + \Delta P).
        \end{align*}

        Following the steps of \cite[App.~A]{Coppens2020-TR} (involving several application of the matrix inversion lemma) 
        we claim (cf. \cite[Eq.~32]{Coppens2020-TR} and \cite[Eq.~34]{Coppens2020-TR} respectively):
        \begin{align*}
            \nrm{\op{R}_0(\Delta P)}_2 &\leq \nrm{\ten{A}_{\star, (2)}}_2^2 \nrm{\ten{B}_{(2)}}_2^2 \nrm{W_{\star}}_2^2 \nrm{\Delta P}_2^2 / \lambda_{n_u}(R), \\
            \nrm{\op{R}_{\Delta}(\Delta P)}_2 &\leq \ten{A}_{(2)} Y^{-\top} (\Delta W \otimes (P_\star + \Delta P)) Y^{-1} \trans{\ten{A}}_{(2)}.
        \end{align*}
        where $\ten{A}_{\star} = \ten{A} + \ten{B} \ttimes{2} \trans{K_\star}$ the closed-loop model tensor and $Y \dfn I + \ten{B}_{(2)} R^{-1} \trans{\ten{B}_{(2)}} \left( W_{\star} \otimes (P_\star + \Delta P) \right)$. 
        The second expression needs to be simplified further. We argued earlier that $\nrm{\Delta W}_2 = \nrm{\bar{W} - W_{\star}}_2 \leq 2\beta_W$ with high probability. Specifically $0 \sleq \bar{W} \sleq W_{\star} + 2\beta_W I_{n_w}$. 
        Thus $\nrm{\op{R}_{\Delta}(\Delta P)}_2 \leq 2 \beta_W \ten{A}_{(2)} Y^{-\top} (I_{n_w} \otimes (P + \Delta P)) Y^{-1} \trans{\ten{A}}_{(2)}$. 
        Using \cite[Lem.~A.1(iii)]{Coppens2020-TR} (and the reasoning below \cite[Eq.~34]{Coppens2020-TR}) 
        we conclude \[\nrm{\op{R}_{\Delta}(\Delta P)}_2 \leq 2 \beta_W \nrm{\ten{A}_{(2)}}^2 (\nrm{P_\star}_2 + \nrm{\Delta P}_2).\]

        Putting everything together we have shown that $\Phi$ in \cref{eq:fixed-point-eq} is bounded as:
        \begin{equation*}
            \nrm{(\adj{\op{L}_\star})^{-1}}_2 \nrm{\Phi(\Delta P)}_2 \leq c_2 \nrm{\Delta P}_2^2 + \beta_W c_1 \nrm{\Delta P}_2 + \beta_W c_0,
        \end{equation*}
        with $c_2 = \nrm{\ten{A}_{\star, (2)}}_2^2 \nrm{\ten{B}_{(2)}}_2^2 \nrm{W_{\star}}_2^2 / \lambda_{\mathrm{min}}(R)$, $c_1 = c_0 = 2\nrm{\ten{A}_{(2)}}^2$. 
        This is analogous to the result of \cite[Lem.~V.1]{Coppens2020-TR} (symmetry also follows from analogous arguments to the ones in \cite[App.~A]{Coppens2020-TR}). Therefore the Brouwer Fixed-Point 
        Theorem can still be applied as in \cite[Sec.~V]{Coppens2020-TR} to bound $\nrm{\Delta P}_2$. This gives us the bound stated in the lemma. 
    \end{proof}

    \paragraph*{Controller Perturbation}
    We next bound $\nrm{\bar{K} - K_\star}_2$ by following a procedure similar to the one of \cite[\S{}VI.]{Coppens2020-TR}. 

    \begin{lemma}\label{lem:controller-bound}
        Given \cref{eq:beta-w-def}, where the right-hand side is labeled $\beta_W$, and letting $\Delta P = \bar{P} - P_{\star}$ then 
        \begin{equation*}
            \nrm{\bar{K} - K_\star}_2 \leq \frac{\beta_W \nrm{\ten{A}_{\star, (2)}}^2_2}{\lambda_{\mathrm{min}}(R)} \left( 2 \nrm{\Delta P}_2 + \nrm{P_\star} \right).
        \end{equation*}
    \end{lemma}
    \begin{proof}
        Take $x \in \Re^{n_x}$ with $\nrm{x}_2 = 1$. By bounding $\nrm{(\bar{K} - K_\star) x}_2$ for all such $x$ 
        the required upper bound is proven. Let $\bar{u} \dfn \bar{K} x$ and $u_\star \dfn K_\star x$, then setting the gradient to zero and 
        referring back to \cref{thm:drlqrcp} and \cref{thm:cplqr} respectively shows that $\bar{u} = \argmin_u \, \bar{f}(u)$ and $u_{\star} = \argmin_u \, f_{\star}(u)$ with 
        \begin{align*}
            \bar{f}(K_\star x) &= \trans{x} \trans{K_\star} R K_\star x + \trans{x} \ten{A}_{\star, (2)} \left( \bar{W} \otimes \bar{P} \right) \trans{\ten{A}_{\star, (2)}} x, \\
            {f}_\star(K_\star x) &= \trans{x} \trans{K_\star} R K_\star x + \trans{x} \ten{A}_{\star, (2)} \left( W_{\star} \otimes P_\star \right) \trans{\ten{A}_{\star, (2)}} x,
        \end{align*}
        with $\ten{A}_{\star} = \ten{A} + \ten{B} \ttimes{2} \trans{K_\star}$ the closed-loop model tensor.
        Both functions are $\lambda_{\mathrm{min}}(R)$-strongly convex. Thus we can apply \cite[Lem.~1]{Mania2019} to show 
        \begin{align*}
            &\nrm{(\bar{K} - K_\star) x}_2 \leq \frac{1}{\lambda_{\mathrm{min}}(R)} \nrm{\nabla \bar{f}(K_{\star} x) - \nabla f_{\star}(K_\star x)}_2 \\
                                          &\quad \leq \frac{\nrm{\ten{A}_{\star, (2)}}^2_2}{\lambda_{\mathrm{min}}(R)}  \nrm{ \bar{W} \otimes \bar{P} -  W_{\star} \otimes P_\star }_2 \nrm{x}_2.
        \end{align*}
        Note that $\nrm{x}_2 = 1$ by construction and $\bar{W} \otimes \bar{P} -  W_{\star} \otimes P_\star = W_\star \otimes \Delta P + \Delta W \otimes (P_\star + \Delta P)$ 
        using $\bar{P} = P_{\star} + \Delta P$ and $\bar{W} = W_{\star} + \Delta W$. 
        Thus, after applying the triangle inequality twice and using the fact that $\nrm{A \otimes B}_2 \leq \nrm{A}_2 \nrm{B}_2$ for matrices $A$ and $B$, 
        \begin{equation*}
            \nrm{(\bar{K} - K_\star) x}_2 \leq \frac{\beta_W \nrm{\ten{A}_{\star, (2)}}^2_2}{\lambda_{\mathrm{min}}(R)}  \left( 2 \nrm{\Delta P}_2 + \nrm{P_\star} \right),
        \end{equation*}
        where we used $\nrm{\Delta W}_2 \leq \beta_W$. 
    \end{proof}

    \paragraph*{Suboptimality}
    Given the controller perturbation from the previous part, the sub-optimality of that controller follows directly from \cite[Thm.~III.3]{Coppens2020-TR}.
    \begin{lemma} \label{lem:suboptimality-bound}
        The suboptimality is bounded as:
        \begin{align*}
            &\E\left[ \sum_{t=0}^\infty \trans{x_t} Q x_t + u_t R u_t \right] - \mathrm{Val}\eqref{eq:slqr}  \\
            &\,\, \leq\min(n_x, n_u) \epsilon_K^2 \nrm{x_0}_2^2 \frac{\nrm{\ten{B}_{(2)}}_2^2 \nrm{P_{\star}}_2 + \nrm{R}_2}{\nrm{\op{L}_{\star}^{-1}}_2} \\
            &\,\, \hphantom{\leq}\boldsymbol{\cdot}  \left[ 
                1 + \tfrac{2 \nrm{\ten{B}_{(1)}}_2 \nrm{\ten{A}_{(1)}}_2 \nrm{W_{\star}}_2 \epsilon_K + \nrm{\ten{B}_{(1)}}^2_2 \nrm{W_{\star}}_2 \epsilon_K^2}{\nrm{\op{L}_{\star}^{-1}}_2 - 2 \nrm{\ten{B}_{(1)}}_2 \nrm{\ten{A}_{(1)}}_2 \nrm{W_{\star}}_2 \epsilon_K - \nrm{\ten{B}_{(1)}}^2_2 \nrm{W_{\star}}_2 \epsilon_K^2}
             \right],
        \end{align*}
        where $x_{t+1} = A(v_t) x_t + B(v_t) u_t$ as in \cref{eq:dyn}, $u_t = \bar{K} x_t$. So the first term 
        is the cost achieved when applying $\bar{K}$ to the true multiplicative noise dynamics \cref{eq:dyn}. Moreover $\epsilon_K = \nrm{K_{\star} - \bar{K}}_2$,
        which should be sufficiently small such that the fraction in square brackets is positive. 
    \end{lemma}

    \paragraph*{Asymptotic Behavior}
    We can study the asymptotic behavior by computing
    \begin{equation} \label{eq:asymptotic-complexity}
        \lim_{N \to \infty} N \cdot \left( \E\left[ \sum_{t=0}^\infty \trans{x_t} Q x_t + u_t R u_t \right] - \mathrm{Val}\eqref{eq:slqr} \right).
    \end{equation}
    Applying in \cref{lem:suboptimality-bound} gives that \cref{eq:asymptotic-complexity} equals 
    \begin{equation*}
        \min(n_x, n_u) \left(\lim_{N \to \infty} \sqrt{N} \epsilon_K\right)^2 \nrm{x_0}^2_2 \tfrac{\nrm{\ten{B}_{(2)}}_2^2 \nrm{P_{\star}}_2 + \nrm{R}_2}{\nrm{\op{L}_{\star}^{-1}}_2}.
    \end{equation*}
    Applying \cref{lem:controller-bound}
    \begin{equation*}
        \lim_{N \to \infty} \sqrt{N} \epsilon_K = \tfrac{\nrm{\ten{A}_{\star, (2)}}_2^2 \nrm{P_\star}_2}{\lambda_{\mathrm{min}}(R)} \left(\lim_{N \to \infty} \sqrt{N} \beta_W\right),
    \end{equation*}
    where we used \cref{lem:riccati-bound} and \cref{eq:beta-w-def} to argue that $\lim_{N \to \infty} \sqrt{N} \beta_W \nrm{\Delta P}_2 = 0$. 
    Finally, using \cref{eq:beta-w-def} we get 
    \begin{equation*}
        \lim_{N \to \infty} \sqrt{N} \beta_W = r_w^2 \sqrt{2 n_w\log(2n_w/\delta)}\tfrac{\nrm{\ten{W}}_2^2}{\gamma_W}. 
    \end{equation*}
    Note that $\adj{\op{L}}_{\star}(P_{\star}) = Q$. Therefore $\nrm{P_{\star}}_2 \leq \nrm{(\adj{\op{L}}_{\star})^{-1}}_2 \nrm{Q}_2$.

    Putting everything together gives that \cref{eq:asymptotic-complexity} equals
    \begin{align*}
        &2 \min(n_x, n_u) n_w \log(2n_w/\delta) \nrm{x_0}_2^2 r_w^2 \\ 
        & \boldsymbol{\cdot}\tfrac{ \nrm{\ten{W}}_2^4}{\gamma_W^2} \tfrac{\nrm{(\adj{\op{L}}_{\star})^{-1}}_2^2}{\nrm{(\op{L}_{\star})^{-1}}_2} \tfrac{ \nrm{\ten{A}_{\star, (2)}}_2^4 \nrm{Q}_2^2 \left( \nrm{\ten{B}_{(2)}}_2^2 \nrm{(\adj{\op{L}}_{\star})^{-1}}_2 \nrm{Q}_2 + \nrm{R}_2 \right)}{\lambda_{\mathrm{min}}^2(R)}.
    \end{align*}
    In the model--free setting -- when no knowledge is available on the modes -- the value of $n_w = n_x (n_x + n_u)$. Thus, the overall dimensional dependency 
    is 
    \begin{equation*}
        \min(n_x, n_u) n_x (n_x + n_u) \log(n_x (n_x + n_u)) \approx n_x^2 (n_x + n_u),
    \end{equation*}
    for the usual setting where $n_x \geq n_u$. In the additive noise, certainty equivalent setting \cite{Mania2019} establish a bound in the certainty equivalent with $\sqrt{(n_x + n_u)^3}$.
    This dimensional dependency is further improved to $\sqrt{n_u n_x^2}$ by \cite{Simchowitz2020}, which also establishes a lower bound with matching dimensional dependency. 
    Our analysis therefore implies that learning multiplicative noise is more difficult to learn with respect to the dimensions. However a lower bound is needed to confirm this with certainty. 

    Also of note are the norms of Lyapunov operators $\nrm{(\adj{\op{L}}_{\star})^{-1}}_2$ and $\nrm{(\op{L}_{\star})^{-1}}_2$. These are CP, so \cref{lem:sopbnd-kron} is applicable. 
    Therefore 
    \begin{align*}
        \nrm{(\op{L}_{\star})^{-1}}_2 &= \nrm{(\op{L}_{\star})^{-1}(I_{n_x})}_2 \\
            &= \sup_{x} \left\{ \tr\left[ \sum_{t=0}^{\infty}(\adj{\op{E}_{\star}})^t(x \trans{x}) \right] \colon \nrm{x}_2 \leq 1 \right\}, \\
            &= \sup_{\adj{x}_0} \left\{\sum_{t=0}^{\infty} \E[\trans{(\adj{x}_t)} \adj{x}_t ]\colon \nrm{\adj{x}_0}_2 \leq 1 \right\}
    \end{align*}
    where the second equality follows from \cref{prop:lyapcp} and the variational definition of a spectral norm.
    The final equality follows by taking $\adj{x}_{t+1} = \trans{(A(w_t) + B(w_t) K_{\star})} \adj{x}_t$, the adjoint 
    multiplicative noise dynamics with second moment dynamics described by $\adj{\op{E}_{\star}}$. 
    Similarly $\nrm{(\adj{\op{L}_{\star}})^{-1}}_2 = \sup_{x_0} \left\{\sum_{t=0}^{\infty} \E[x_t \trans{x_t}]\colon \nrm{x_0}_2 \leq 1 \right\}$,
    with ${x}_{t+1} = (A(w_t) + B(w_t) K_{\star}) {x}_t$, the 
    multiplicative noise dynamics with second moment dynamics described by $\op{E}_{\star}$. For linear systems with additive noise, these values can be related 
    to the $\op{H}_{\infty}$ norm and the operator defined in \cite[Eq.~3]{Mania2019}.
\end{prepupdate}
\end{arxiv}

\begin{arxiv}
\section{Semidefinite programming}

We will require the following technical lemmas:
\begin{lemma} \label{lem:trivialinequality}
    For some $P \in \sym{d}$ then 
    \begin{equation*}
        \tr[P X] > 0, \, \forall X \nsgeq 0 \quad \Leftrightarrow \quad P \sgt 0.
    \end{equation*}
\end{lemma}
\begin{proof}
        ($\Rightarrow$) Assume $P$ is not positive definite, which implies $\exists x \colon \trans{x} P x \leq 0$.
        So $X = x \trans{x} \nsgeq 0$ gives a contradiction.
        
        ($\Leftarrow$) Any $X \nsgeq 0$ can be written as $\sum_{i=1}^r x_i \trans{x_i}$ for some $r > 0$ and nonzero $x_i$. 
        Then by the cyclic property of the trace $\tr[P X] = \sum_{i=1}^r \trans{x_i} P x_i > 0$ by $P \sgt 0$. 
\end{proof}
\begin{prepupdate}%
\begin{lemma} \label{lem:monotone-convergence}
    Consider a sequence $\{P_t\}_{t =0}^{\infty} \in \sym{d}$ such that 
    for any $X \in \sym{d}$,
    \begin{equation*}
        \tr[X P_t] \leq \tr[X P_{t+1}], \, \forall t \in \N \text{ and } \lim_{t \to \infty} \tr[XP_t] < +\infty
    \end{equation*}
    Then $P_{\infty} = \lim_{t \to \infty} P_t$ is well defined. 
\end{lemma}
\begin{proof}
    This lemma is used in the proof of \cite[Prop.~4.4.1]{Bertsekas2005V1}. Consider an
    orthonormal basis $\{X_i\}_{i=1}^{\sd{d}}$ for $\sym{d}$. Then 
    \begin{equation*}
        \lim_{t \to \infty} \tr[X_i P_t] =  \tr[X_i \lim_{t \to \infty} P_t] =  \tr[X_i P_{\infty}]
    \end{equation*}
    is well defined by the monotone convergence theorem (since $\tr[X_i P_t] \leq \tr[X_i P_{t+1}]$ and $\tr[X_i P_t] < +\infty$
    for all $t \in \N$). Completing the same argument for every $i$ gives us a linear system of equations 
    with $P_{\infty}$ the unique solution.
\end{proof}


\end{prepupdate}

We extend the result of \cite[Thm.~6.2.1, 6.2.2]{Ben-Tal2000}.
This result was previously integrated in \cite[Thm.~10]{Coppens2019}. 
\begin{lemma} \label{lem:rob-sdp}
    Consider the parametric family of matrices,
    \begin{equation*}
        \set{U}[\rho] = \left\{ 
            T + \rho \ssum_{i=1}^d [\theta]_i (\trans{L} F_i R + \trans{R} \trans{F}_i L) \colon \theta \in \ball{d},
         \right\}
    \end{equation*}
    with $T \in \sym{\ell}$, $F_i \in \Re^{m \times n}$. Consider the predicate
    \begin{subequations} \label{eq:pred}
        \begin{align}
            \exists \Lambda \in \psd{n}, \Gamma \in \psd{m} \colon &T \sgeq \trans{L} \Gamma L + \trans{R} \Lambda R, \label{eq:pred:a}\\
            &\begin{bmatrix}
                \Lambda & \rho \trans{F_1} & \dots & \rho \trans{F_d} \\ 
                \rho F_1 & \Gamma \\
                \vdots && \ddots \\
                \rho F_d &&& \Gamma
            \end{bmatrix} \sgeq 0. \label{eq:pred:b}
        \end{align}
    \end{subequations}
    Then we have the following implication:\todo{rank assumption on $R$ and $L$?}
    \begin{equation*}
        \text{if \eqref{eq:pred} holds then $\set{U}[\rho] \sgeq 0$;}
    \end{equation*}
\end{lemma}
\begin{proof}
    We follow the proof of \cite[Thm.~6.2.1, Thm.~6.2.2]{Ben-Tal2000}.
    Assume \eqref{eq:pred} holds for 
    some $\Lambda, \Gamma \sgeq 0$. We assume $\Lambda = \Lambda^{1/2} \Lambda^{1/2}$ and $\Gamma = \Gamma^{1/2} \Gamma^{1/2}$ are invertible without loss of generality. 
    Otherwise, use a Cholesky decomposition to construct a right invertible substitute for $\Lambda^{1/2}$ and $\Gamma^{1/2}$
    and use these instead. 

    For $x \in \Re^{\ell}$, let $u = \sqrt{2} \Lambda^{1/2} R x$ and $v = \sqrt{2} \Gamma^{1/2} L x$. Then,
    \begin{equation*}
        2\trans{x} (\ssum_{i=1}^d [\theta]_i \trans{L} F_i R) x = \trans{v} (\ssum_{i=1}^d [\theta]_i \Gamma^{-1/2} F_i \Lambda^{-1/2}) u.
    \end{equation*}
    Therefore $\set{U}[\rho] \sgeq 0$ if
    \begin{equation} \label{eq:pred-first}
        \trans{x} A x + \rho \trans{v} (\ssum_{i=1}^d [\theta]_i \Gamma^{-1/2} F_i \Lambda^{-1/2}) u \geq 0,
    \end{equation}
    for all $x \in \Re^{\ell}$, $u \in \Re^{n}$, $v \in \Re^m$ and $\theta \in \ball{d}$. 

    We lower bound the second term of \eqref{eq:pred-first} as
    \begin{align*}
        &\rho \trans{v} (\ssum_{i=1}^d [\theta]_i \Gamma^{-1/2} F_i \Lambda^{-1/2}) u \\
        &\qquad \labelrel\geq{eq:pred-second:a} - \rho \nrm{v}_2 (\ssum_{i=1}^d |[\theta]_i| \nrm{\Gamma^{-1/2} F_i \Lambda^{-1/2} u}_2)\\
        &\qquad \labelrel\geq{eq:pred-second:b} - \rho \nrm{v}_2 \sqrt{\ssum_{i=1}^d \trans{u} \Lambda^{-1/2} \trans{F_i} \Gamma^{-1} F_i \Lambda^{-1/2}u}.
    \end{align*}
    Here \ref{eq:pred-second:a} used the triangle inequality and \ref{eq:pred-second:b} used $\trans{|\theta|} \beta \geq -\nrm{|\theta|}_2 \nrm{\beta}_2 = -\nrm{\beta}_2$
    for $[\beta]_i = \nrm{\Gamma^{-1/2} F_i \Lambda^{-1/2} u}_2$.

    Note that \eqref{eq:pred:b} implies $\rho^2 \ssum_{i=1}^d\Lambda^{-1/2} \trans{H_i} \Gamma^{-1} F_i \Lambda^{-1/2} \sleq I$,
    by invertibility of $\Lambda, \Gamma$ and a Schur complement. Combining this with \eqref{eq:pred-first} gives that $\set{U}[\rho] \sgeq 0$ holds if
    $
        \trans{x} T x \geq  \nrm{v}_2 \nrm{u}_2,
    $
    for all $x \in \Re^{\ell}$, $u \in \Re^{n}$, $v \in \Re^m$. By the inequality of arithmetic and geometric means (i.e., $(\alpha + \beta)/2 \geq \sqrt{\alpha \beta}$)
    we have the final sufficient condition for $\set{U}[\rho] \sgeq 0$:
    \begin{equation*}
        \trans{x} T x \geq (\nrm{v}_2^2 + \nrm{u}_2^2)/2 = \trans{x} (\trans{L} \Gamma L + \trans{R} \Lambda R) x,
    \end{equation*}
    which holds $\forall x \in \Re^{\ell}$ by \eqref{eq:pred:a}. 
    %
\end{proof}

\end{arxiv}


\finalpage{}  

\end{document}